\documentclass{amsbook}
\usepackage{amsmath,amsthm}
\usepackage{amsfonts,amssymb}
\usepackage{enumerate}
\usepackage{mathrsfs}

\usepackage[usenames,dvipsnames]{color}

\usepackage{hyperref}

\usepackage[numeric]{amsrefs}

\expandafter\def\expandafter\label\expandafter#\expandafter1\expandafter{\expandafter\lowercase\expandafter{\label{#1}}}
\expandafter\def\expandafter\ref\expandafter#\expandafter1\expandafter{\expandafter\lowercase\expandafter{\ref{#1}}}
\expandafter\def\expandafter\eqref\expandafter#\expandafter1\expandafter{\expandafter\lowercase\expandafter{\eqref{#1}}}

\newtheorem{thm}{Theorem}[chapter]
\newtheorem{lem}[thm]{Lemma}

\newtheorem{cor}[thm]{Corollary}

\theoremstyle{definition}
\newtheorem{defn}[thm]{Definition}

\theoremstyle{remark}
\newtheorem{rem}[thm]{Remark}

\numberwithin{section}{chapter}
\numberwithin{equation}{chapter}

\def\Ld{\Lambda}
\def\bfi {\mathbf{i}}
\def\bfj{\mathbf{j}}

\def\f{\frac}

\def\vi{\varphi}
\def\({\left(}
\def \){ \right)}

\def\Bl{\Bigl}
\def\Br{\Bigr}

\def\ee{{\textnormal{e}}}

\def\tr{{\triangle}}
\def\Ga{\Gamma}

\def\ta{\theta}
\def\al{{\alpha}}
\def\da{{\delta}}
\def\sa{{\sigma}}

\def\b{{\beta}}

\def\ga{{\gamma}}
\def\k{{\kappa}}
\def\t{{\theta}}

\def\va{\varepsilon}

\def\ib{{\mathbf i}}

\def\kb{{\mathbf k}}

\def\nb{{\mathbf n}}

\def\EEE{{\mathcal E}}

\def\NN{{\mathbb N}}

\def\RR{{\mathbb R}}
\def\SS{{\mathbb S}}

\def\ZZ{{\mathbb Z}}

\def\proj{\operatorname{proj}}

\def\sph{\mathbb{S}^{d-1}}

\def\Og{\Omega}

\def\al{\alpha}
\newcommand{\wt}{\widetilde}
\newcommand{\wh}{\widehat}
\def\sub{\substack}

\def\osc{\operatorname{osc}}

\def\p{\partial}
\def\ld{\lambda}
\def\bl{\bigl}
\def\br{\bigr}
\def\og{\omega}
\def\Ld{\Lambda}

\newcommand{\R}{{\mathbb{R}}}
\newcommand{\eps}{\varepsilon}

\newcommand{\diam}{{\rm diam}}
\newcommand{\dist}{{\rm dist}}
\newcommand{\spn}{{\rm span}}

\def\be{\begin{equation}}
\def\ee{\end{equation}}

\begin{document}
\frontmatter
\title{Polynomial approximation on $C^2$-domains}


\author{Feng Dai}
\address{Department of Mathematical and Statistical Sciences\\
	University of Alberta\\ Edmonton, AB, T6G 2G1, Canada}
\email{fdai@ualberta.ca}

\author{Andriy Prymak}
\address{Department of Mathematics, University of Manitoba, Winnipeg, MB, R3T2N2, Canada}
\email{prymak@gmail.com}

\thanks{
	The first author was supported by  NSERC of Canada
	grant RGPIN 04702-15, and the second author  was supported by NSERC of Canada grant RGPIN 04863-15.
}

\date{}
\subjclass[2010]{Primary 41A10, 41A17, 41A27, 41A63;\\Secondary 41A55, 65D32}
\keywords{$C^2$-domains, polynomial approximation, modulus of smoothness, Jackson inequality, inverse theorem, tangential Bernstein inequality, Marcinkiewicz-Zygmund inequality, positive cubature formula 
 }

\begin{abstract}
	
	 We  introduce appropriate computable moduli of smoothness
	to 
	characterize
	the rate of best  approximation  by multivariate
	polynomials on a connected and  compact $C^2$-domain $\Omega\subset \mathbb{R}^d$. This  new modulus of smoothness is defined via finite differences along the directions of coordinate axes, and along a number of  tangential directions from the boundary.  With this modulus, we  prove  both the  direct Jackson inequality    and the corresponding  inverse for   best polynomial approximation in  $L_p(\Omega)$. The Jackson inequality is established for the full range of $0<p\leq \infty$,  while  its proof  relies on  (i)  Whitney type  estimates with constants depending only on certain  parameters; and  (ii)  highly localized  polynomial partitions of  unity on a $C^2$-domain. Both (i) and (ii)  are of independent interest. In particular, our   Whitney type estimate (i) is established for directional moduli of smoothness rather than the ordinary  moduli of smoothness, and is applicable to functions on  a very wide class of  domains  (not necessarily convex). It  generalizes  an earlier result of Dekel and Leviatan on     Whitney type   estimates   on  convex domains.
	The inverse inequality is established for $1\leq p\leq \infty$, and its  proof relies on a   new  Bernstein type inequality associated with the tangential derivatives on the boundary of $\Omega$.  Such an  inequality also allows us to establish the Marcinkiewicz-Zygmund type inequalities, positive cubature formula, as well as the inverse theorem for Ivanov's  average moduli of smoothness on general compact $C^2$-domains.
	
\end{abstract}

\maketitle

\setcounter{page}{4}
\tableofcontents

\mainmatter
%
%
%


\chapter{Introduction}

One of the primary questions of approximation theory is to characterize the
rate of approximation by a given system in terms of some  modulus
of  smoothness.   It is  well known (\cite{Ni}) that  the quality of approximation  by algebraic polynomials  increases towards the boundary of the underlying domain. As a result,  characterization of the class of functions with a prescribed rate of best  approximation by algebraic polynomials on a compact domain with nonempty boundary 
cannot be described by the ordinary moduli of smoothness. 
Several  successful  moduli of smoothness were introduced to solve  this problem in the setting of one variable. Among them the most established   ones are the  Ditzian-Totik moduli of smoothness  \cite{Di-To} and the average moduli of smoothness of K. Ivanov  \cite{Iv2}  (see the survey paper  \cite{Dit07} for details).  
Successful  attempts were also   made to solve the problem in more variables,    the  most notable  being the work  of K. Ivanov   for polynomial approximation on piecewise $C^2$-domains in $\RR^2$  \cite{Iv}, and the recent works of Totik for polynomial approximation  on general  polytopes and algebraic domains  \cite{To14, To17}; we will describe~\cite{Iv} and~\cite{To17} in more details below. The following list is not meant to be exhaustive, but we would like to also mention several other related works: results for simple polytopes by Ditzian and Totik~\cite{Di-To}*{Chapter~12}, an announcement of a characterization of approximation classes by Netrusov~\cite{Ne}, possibly reduction to local approximation by Dubiner~\cite{Du}, results for simple polytopes for $p<1$ by Ditzian~\cite{Di96}, a new modulus of smoothness and characterization of approximation classes on the unit ball by the first author and Xu~\cite{DX}, and a different alternative approach on the unit ball by Ditzian~\cite{Di14a,Di14b}.

The  main aim in this paper is to introduce a  computable  modulus of smoothness for functions on  $C^2$-domains,  for which both the direct Jackson inequality and  the corresponding  converse hold. As is well known,  the definition of  such a modulus must take into account  the boundary of the underlying domain. 

We start with some necessary notations. 
Let $L^p(\Og)$, $0<p<\infty$ denote the Lebesgue $L^p$-space defined with respect to the Lebesgue measure on a compact domain $\Og\subset \RR^d$. In the limit case we set $L^\infty(\Og)=C(\Og)$, the space of all continuous functions  on $\Og$ with the uniform norm $\|\cdot \|_\infty$.
Given  $\xi, \eta\in \RR^d$, and  $r\in\NN$,  we define 
$$ \tr_\xi^r f(\eta) :=\sum_{j=0}^r (-1)^{r+j} \binom{r} {j} f(\eta+j\xi ),$$
where we assume that $f$ is defined everywhere on the set  $\{\eta+j\xi:\  \   j=0,1,\dots, r\}$.
For a function $f: \Og\to\RR$, we also define  
\begin{equation}\label{finite-diff}
\tr_\xi^r (f, \Og, \eta):=\begin{cases}
\tr_\xi^r f(\eta),\   \ &\text{if  $[\eta, \eta+r\xi]\subset \Og$,}\\
0, &\   \  \text{otherwise},
\end{cases}
\end{equation}
where $[x,y]$ denotes the line segment connecting any two points $x,y\in\RR^d$. 
The best approximation of  $f\in L^p(\Og)$ by means of algebraic polynomials of total degree at most $n$ is defined as
$$ E_n(f)_p=E_n(f)_{L^p(G)}:=\inf\Bl\{ \|f-Q\|_p:\   \  Q\in \Pi^d_n\Br\},$$
where $\Pi^d_n$ is the space of algebraic polynomials of total degree $\le n$ on $\RR^d$.
Given a set  $E\subset \RR^d$, we denote by $|E|$ its Lebesgue measure in $\RR^d$, and   define  $\dist(\xi, E):=\inf_{\eta\in E}\|\xi-\eta\|$ for $\xi\in \RR^d$, (if $E=\emptyset$, then define $\dist(\xi, E)=1$). Here and throughout the paper,   $\|\cdot\|$ denotes the Euclidean norm. 
Finally, let $\sph\subset \RR^d$ be the unit sphere of $\RR^d$, and let $e_1=(1,0,\dots, 0), \dots, e_d =(0, \dots, 0, 1)$ denote the standard canonical basis  in $\RR^d$.

Next, we describe   the work of K. Ivanov in \cite{Iv}, where  a new modulus of smoothness was introduced  to study the best algebraic polynomial approximation
for functions of two variables on a bounded domain   with piecewise $C^2$ boundary.  To avoid technicalities,  we always   assume  that  $\Og\subset \RR^d$ is  the closure of an open,  bounded, and connected  domain in $\R^d$    with $C^2$ boundary $\Ga$.   Consider the following metric on $\Og$:
\begin{equation}\label{metric}\rho_\Og (\xi,\eta):=\|\xi-\eta\|+ \Bl|\sqrt{\dist(\xi, \Gamma)} -\sqrt{\dist(\eta, \Gamma)}\Br|,\   \  \xi, \eta\in\Og. \end{equation}
For $\xi\in\Og$ and $t>0$, set 
$ U( \xi, t):= \{\eta\in\Og:\  \  \rho_\Og(\xi,\eta) \leq t\}$.
For $0< q\leq p\leq \infty$,  the average   $(p,q)$-modulus of order $r\in\NN$ of $f\in L^p(\Og)$  was defined in  \cite{Iv} by \footnote{Both the metric $\rho_\Og$ and  the average moduli of smoothness $\tau_r(f,t)_{p,q}$  were defined in \cite{Iv} for a more general domain $\Og\subset \RR^2$. } 
\begin{equation}\label{eqn:ivanov} \tau_r (f; \da)_{p,q} :=\Bl\| w_r (f, \cdot, \da)_q \Br\|_p,\end{equation}
where    
$$ w_r (f, \xi, \da)_q : =\begin{cases}
\displaystyle \Bl( \f 1 {|U(\xi,\da)|} \int_{U(\xi,\da)} |\tr_{(\eta-\xi)/r} ^r (f,\Og,\xi)|^q \, d\eta\Br)^{\f1q},\  \  & \text{if $0<q <\infty$};\\
\sup_{\eta\in U( \xi,\da)} |\tr_{(\eta-\xi)/r}^r (f,\Og,\xi)|,\   \ &\text {if $q=\infty$}.\end{cases}$$
With this modulus,  the following result was  announced  without proof in  \cite{Iv} for a   bounded domain  in the plane   with piecewise $C^2$ boundary.
\begin{thm}\label{thm:ivanov}\cite{Iv} Let $\Og$ be the closure of a bounded open domain in the plane $\RR^2$ with piecewise $C^2$-boundary $\Ga$. 	
	If $f\in L^p(\Og)$, $1\leq q\leq p \leq \infty$ and $r\in\NN$, then
	\begin{equation}\label{1-4-0} E_n (f)_p \leq C_{r, \Og} \tau_r (f, n^{-1})_{p,q}.\end{equation}
	Conversely, if either $p=\infty$ or $\Og$ is a parallelogram or a disk and $1\leq p\leq \infty$, then
	\begin{equation}\label{1-5-0}\tau_r (f, n^{-1})_{p,q} \leq C_{r,\Og} n^{-r} \sum_{s=0}^n (s+1)^{r-1} E_s (f)_p.\end{equation}
\end{thm}

It remained open in \cite{Iv} whether the inverse inequality~\eqref{1-5-0} holds for the full range of $1\leq p\leq \infty$  for more general  $C^2$-domains other than parallelograms and  disks.  The methods developed in this paper allow us to give a positive answer to this question.  In fact, we shall prove the Jackson inequality~\eqref{1-4-0} for $0<p\leq \infty$ and the inverse inequality~\eqref{1-5-0} for $1\leq p\leq \infty$ for all compact, connected $C^2$-domains $\Og\subset \RR^d$. Our results   apply to higher dimensional domains as well.

Finally, we describe the recent  work of Totik  in \cite{To17}, where a new modulus of smoothness using the univariate moduli of smoothness on circles and line segments was introduced to study polynomial approximation on  algebraic domains.
Let  $\Og\subset \RR^d$ be the closure of a bounded, finitely connected domain with $C^2$ boundary $\Ga$. Such a domain  is called an algebraic domain  if for   each connected  component $\Ga'$ of the boundary $\Ga$,   there  is  a  polynomial $\Phi(x_1,\dots, x_d)$ of $d$ variables  such that $\Ga'$  is one of the components of  the surface $\Phi(x_1,\dots, x_d)=0$ and $\nabla \Phi(\xi) \neq 0$ for each $\xi\in\Ga'$.   
The $r$-th order  modulus of smoothness of $f\in C(\Og)$ on a circle $\mathcal{C}\subset \Og$ is defined  as in the classical trigonometric approximation theory by  
\begin{align*}
\wh{\og}_{\mathcal{C}} ^r (f,t) :
&=\sup_{0\leq \ta \leq t} \sup_{0\leq \vi\leq 2\pi } \Bl| \sum_{k=0}^r (-1)^k \binom{r} k   f_{\mathcal{C}}\bl (\vi+(\f r2-k)\ta\br)\Br|
\end{align*}
where  we identify the circle $\mathcal{C}$ with the interval $[0, 2\pi)$ and  $f_{\mathcal{C}}$ denotes   the restriction of $f$ on $\mathcal{C}$.
Similarly, if $I=[a,b]\subset \Og$ is a line segment and $e\in\SS^{d-1}$ is the direction of $I$, then with  $\wt{d}_I (e, z): =\sqrt{\|z-a\|\|z-b\|}$,  we may define 
the modulus of smoothness of $f\in C(\Og)$   on $I$  as  
\begin{align*}
\wh{\og}_I^r (f,\da)
&=\sup_{0\leq h\leq \da} \sup_{z\in I} \Bl|\sum_{j=0}^r (-1)^j \binom{r} j f\Bl( z+(\f r2 -j) h \wt{d}_I (e, z)e\Br)\Br|,
\end{align*}
where it is agreed that the sum on the right hand side is zero  if either of the points $z\pm (r/2-j) h \wt{d}_I (e, z)$, $j=0, 1,\dots, r$ does not belong to $I$. 
Now we  define the $r$-th order  the modulus of smoothness of $f\in C(\Og)$ on the domain $\Og$ as   $$ \wh{\og}^r(f,\da)_\Og=\max\Bl( \sup_{\mathcal{C}_\rho} \wh{\og}_{\mathcal{C}_\rho}^r (f,\da), \sup_I \wh{\og}_I^r (f,\da)\Br),$$
where the supremums are taken for all circles $\mathcal{C}_\rho\subset \Og$ of some radius $\rho$  which are parallel with a coordinate plane,  and for all segments $I\subset \Og$ that are parallel with one of the coordinate axes. 
With this modulus of smoothness,  Totik \cite{To17}  proved  

\begin{thm}\label{Totik:thm}\cite{To17}
	If $\Og\subset \RR^d$ is an algebraic domain and $f\in C(\Og)$, then 
	\begin{equation} \label{1-6-0}E_n (f)_{C(\Og)} \leq C \wh{\og}^r (f, n^{-1})_{\Og},\   \ n\ge rd,\end{equation}
	and 
	\begin{equation} \label{1-7-0} \wh{\og}^r(f, n^{-1})_{\Og}  \leq C n^{-r} \sum_{k=0}^n (k+1)^{r-1} E_k (f)_{C(\Og)} \end{equation}
	with a constant $C$ independent of $f$ and $n$.
\end{thm}

From the classical inverse inequalities in one variable, and 
the way the moduli of smoothness  $\wh{\og}^r(f,t)_\Og$ are defined,  one can easily show that the inverse inequality~\eqref {1-7-0} in fact  holds on more general $C^2$-domains $\Og$. On the other hand,  however,    it is much harder to show   the direct Jackson inequality~\eqref{1-6-0} even on algebraic domains (see \cite{To17}). In fact,   it remains open whether the Jackson inequality~\eqref{1-6-0}
holds on more general $C^2$-domains  as well. Moreover, it  is unclear   how to extend the results of Theorem~\ref{Totik:thm} to  $L^p$ spaces with $p<\infty$.

In this paper, we will introduce a new computable modulus of smoothness on a connected, compact $C^2$-domain $\Og\subset \RR^d$. Our new modulus of smoothness is defined via finite differences along the directions of coordinate axes, and along finitely many tangential directions on the boundary.  With this modulus, we shall prove  the  direct Jackson inequality   for the full range of $0<p\leq \infty$, and the corresponding  inverse for $1\leq p\leq \infty$.  The proof of the Jackson inequality relies on a Whitney type estimate on certain domains of special type, and a polynomial partition of unity on $\Og$. On the other hand, the proof of the inverse inequality is more difficult. It 
relies on a new tangential Bernstein inequality on $C^2$-domains, which  is of independent interest. 

We  give some preliminary materials in  Section~\ref{preliminary}.  After that, we define the new modulus of smoothness in Section~\ref{modulus:def}.  The main results of this paper are summarized in  Section~\ref{summary}, where we also  describe briefly the organization of the  remaining chapters.

%
%
%

\chapter{New moduli of smoothness and main results}

\section{Preliminaries }\label{preliminary}

We start with a brief description of some necessary notations. To simplify our discussion, we  find it   more convenient  to  work on the $(d+1)$-dimensional Euclidean space $\RR^{d+1}$ rather than    the  $d$-dimensional  space $\RR^d$. We shall often    write a point in $\RR^{d+1}$ in the form $(x, y)$  with  $x=(x_1,\dots, x_d)\in\RR^d$ and $y=x_{d+1}~\in~\RR$.     
Let $\SS^d:=\{\xi\in\RR^{d+1}:\  \ \|\xi\|=1\}$ denote the unit sphere of $\R^{d+1}$, where $\|\cdot\|$ denotes the Euclidean norm. 
Let $B_r[\xi]$ (resp., $B_r(\xi)$ )  denote  the closed ball (resp., open ball)  in $\R^{d+1}$ centered at $\xi\in\RR^{d+1}$   having radius $r>0$.  We denote by $[\xi, \eta]$ the line segment connecting any two points $\xi, \eta\in\RR^{d+1}$; that is, $[\xi, \eta]=\{ t\xi+(1-t) \eta:\   \  t\in [0,1]\}$. A   rectangular box in $\RR^{d+1}$ is a set that  takes the form $[a_1,b_1]\times \dots\times [a_{d+1}, b_{d+1}]$ with $-\infty<a_j<b_j<\infty$,  $j=1,\dots, d+1$.  
We  always assume that the sides of a  rectangular box   are parallel with  the coordinate axes.   If $R$ denotes either a parallelepiped or a ball in $\RR^{d+1}$, then  we  denote by $cR$ the dilation of  $R$ from its center by a factor $c>0$.
Let $e_1:=(1,0,\dots, 0)$, $\dots, e_{d+1} =(0,\dots, 0,1)\in\RR^{d+1}$ be the standard basis in $\RR^{d+1}$.  Given  $1\leq i\neq j\leq d+1$,  we call   the coordinate  plane spanned by the vectors $e_i$ and $e_j$  the $x_ix_j$-plane.  Finally, we use the notation  $A_1\sim A_2$ to mean that there exists a positive constant $c>0$ such that $c^{-1}A_1\leq A_2\leq c A_1$.

\subsection{Directional Moduli of smoothness} 

The $r$-th order   directional modulus of smoothness  on a domain   $\Og\subset \RR^{d+1}$  along  a  set $\mathcal{E}\subset \SS^d$ of directions    is defined  by 
$$ \og^r (f,  t; \mathcal{E})_p:=\sup_{\xi\in \mathcal{E}} \sup_{0<u\leq t } \|\tr_{u\xi}^r (f, \Og, \cdot)\|_{L^p(\Og)}=
\sup_{\xi\in \mathcal{E}} \sup_{0<u\leq t } \|\tr_{u\xi}^r f\|_{L^p(\Og_{ru\xi})},$$
where $\tr_{u\xi}^r f=\tr_{u\xi}^r(f,\Og, \cdot)$ is  given in~\eqref{finite-diff}, and 
$\Og_{\eta}:=\{\xi\in \Og:\  \  [\xi, \xi+\eta]\subset \Og\}$ for $\eta\in\RR^{d+1}$. 
Let 
$$ \og^r(f, \Og; \mathcal{E})_p:=
\og^r (f,  \diam(\Og); \mathcal{E})_p,$$ where $\diam (\Og) :=\sup_{\xi, \eta \in \Og} \|\xi-\eta\|$. 
If $\mathcal{E}=\SS^d$, then we write $\og^r (f,  t)_p=\og^r (f,  t; \SS^d)_p$ and  $\og^r(f, \Og)_p= \og^r(f, \Og; \SS^d)_p$, whereas if $\mathcal{E}=\{e\}$ contains only one direction $e\in\SS^d$,  we write  $\og^r (f,  t; e)_p=\og^r (f,  t; \mathcal{E})_p$ and  $\og^r(f, \Og; e)_p= \og^r(f, \Og; \mathcal{E})_p$.
We shall frequently use the following two properties of these directional moduli of smoothness, which can be easily verified from the definition: \begin{enumerate}[\rm (a)] 
	\item For each $\mathcal{E}\subset \SS^d$,   \begin{equation}\label{2-1-18}\og^r (f,  \Og; \mathcal{E})_p=\og^r (f,  \Og; \mathcal{E}\cup(-\mathcal{E}))_p.\end{equation}
	\item   If $T$ is an affine mapping given by  $T\eta =\eta_0 +T_0\eta$ for all $\eta\in\R^{d+1}$ with   $\eta_0\in\R^{d+1}$ and   $T_0$ being  a nonsingular linear mapping on $\R^{d+1}$,   then 
	\begin{equation}\label{5-1-eq}
	\og^r (f,   \Og; \mathcal{E})_p =\Bl|\text{det} \ (T_0)\Br|^{-\f1p}  \og^r( f\circ T^{-1},  T(\Og); \mathcal{E}_{T})_p,
	\end{equation}
	where  
	$\mathcal{E}_{T}=\bl\{\f{ T_0 x}{\|T_0 x\|}:\  \ x\in \mathcal{E}\br\}$.   Moreover, if   $\xi, e\in\SS^d$ is such that $e=T_0(\xi)$, then for any $h>0$,
	\begin{equation}\label{2-3-18}
	|\text{det} \  (T_0)|^{\f1p} 	\|\tr_{he} ^r ( f, \Og)\|_{L^p(\Og)}^p=\|\tr_{h\xi}^r (f\circ T^{-1}, T(\Og))\|_{L^p(T(\Og))}.
	\end{equation}	 
	
\end{enumerate} 
Next, we  recall that  the  analogue of the  Ditzian-Totik modulus on   $\Og\subset \RR^{d+1}$  along a direction $e\in\SS^d$    is defined as  (see \cite{To14, To17}):
\begin{equation}\label{2-3-DT}\og_{\Og, \vi}^r(f,t; e)_{p}:=
\sup_{\|h\|\leq \min\{t,1\}} \Bl\|\tr_{h\vi (e, \cdot) e}^r (f, \Og, \cdot)\Br\|_{L^p(\Og)},\      \   t>0,\end{equation}
where 
\begin{equation}\label{funct-vi} \vi_\Og (e, \xi):=\max\Bl\{ \sqrt{l_1l_2}:\  \  l_1, l_2\ge 0,\  \ [\xi-l_1 e, \xi+l_1 e]\subset \Og\Br\},\  \ \xi\in\Og.\end{equation}
For simplicity, we also define  $\vi_\Og (\da e, \xi)=\vi_\Og (e, \xi)$ for $e\in\SS^d$,  $\da>0$ and $\xi\in\Og$.

\subsection{Domains of special type}




A set  $G\subset \RR^{d+1}$ is called  an {\sl upward}  $x_{d+1}$- domain with base size $b>0$ and parameter $L\ge 1$  if 
it can be written in the form 
\begin{equation} \label{2-7-special}G=\xi+\Bl\{( x,  y):\  \  x\in (-b,b)^d,\   \  g(x)- L b< y \leq g(x)\Br\} \end{equation}
with   $\xi\in\RR^{d+1}$ and   $g\in C^2(\RR^d)$. 
For such a domain $G$,   and a parameter  $\ld\in (0, 2]$,  we define 
\begin{align}
G(\ld):&=\xi +\Bl\{ (x,   y):\  \  x\in (-\ld b, \ld b)^d,\   \   g(x)-\ld L b < y \leq g(x)\Br\},\\
\p'G(\ld)&:=\xi +\Bl\{ (x,  g(x)):\  \  x\in (-\ld b, \ld b)^d\Br\}.
\end{align} 
Associated with the set $G$ in~\eqref{2-7-special}, we also define 
\begin{align}
G^\ast:&=\xi+\Bl\{( x,   y):\  \  x\in (-2b,2b)^d,\   \ \min_{u\in [-2b, 2b]^d} g(u)-4Lb <  y \leq g(x)\Br\}.\label{G}
\end{align}


For later applications, we give the following remark on the above definition.
\begin{rem}\label{rem-2-1-0}
	In the above definition,
	we may choose the base size  $b$    as small as we wish, and we   may also  assume the parameter $L$ in~\eqref{2-7-special} satisfies 
	\begin{equation}\label{parameter-2-9}
	L\ge L_G:=4\sqrt{d} \max_{x\in [-2b, 2b]^d} \|\nabla g(x)\| +1,
	\end{equation}
	since otherwise we may consider a   subset of the form 
	\begin{equation*} G_0=\xi+\Bl\{( x,  y):\  \  x\in (-b_0,b_0)^d,\   \  g(x)- L_0 b_0< y \leq g(x)\Br\}\end{equation*}
	with  $L_0=Lb/b_0$ and $b_0\in (0, b)$ being  a sufficiently small constant. 
	Unless otherwise stated, we will always assume that  the condition~\eqref{parameter-2-9} is satisfied for each upward $x_{d+1}$-domain.

	%
	%
	%
\end{rem}

We  may  define 
an upward $x_j$-domain $G\subset \RR^{d+1}$    and the associated   sets $G(\ld)$,   $\p' G(\ld)$, $G^\ast$  for  $1\leq j\leq d$  in a similar manner, using  the reflection  
$$\sa_j (x) =(x_1,\dots, x_{j-1}, x_{d+1}, x_{j+1},\dots, x_d, x_j),\  \  x\in\RR^{d+1}.$$
Indeed,  $G\subset \RR^{d+1}$
is  an  {\it upward} $x_j$-domain with base size $b>0$  and parameter  $L\ge 1$ if  $E:=\sa_j (G)$ is an  upward $x_{d+1}$-domain with base size $b$ and parameter $L$, in which case we define   
$$G (\ld) = \sa_j \bl( E (\ld)\br),\    \ 
\p' G(\ld)= \sa_j\bl(\p' E (\ld)\br),\   \   \   G^\ast =\sa_j (E^\ast).$$

We can also  define a  {\it downward} $x_j$-domain   and the associated   sets $G(\ld)$,    $\p' G(\ld)$, using  the reflection with respect to the coordinate plane $x_j=0$:  
$$\tau_j (x) :=(x_1, \dots, x_{j-1}, -x_j, x_{j+1}, \dots, x_{d+1}),\   \ x\in\RR^{d+1}.$$
Indeed,   $G\subset \RR^{d+1}$
is  an  {\it downward} $x_j$-domain with base size $b>0$ and parameter $L\ge 1$  if  $H:=\tau_j (G)$ is an  upward $x_{j}$-domain with base size $b$ and parameter $L\ge 1$, in which case we define 
$$G (\ld) = \tau_j \bl( H (\ld)\br),\   \
\p' G(\ld)= \tau_j\bl(\p' H (\ld)\br),\   \    G^\ast =\tau_j(H^\ast).$$

%
%
%

We  say     $G\subset  \RR^{d+1}$  is  a domain  of special type   if it  is an upward or downward $x_j$-domain for some $1\leq j\leq d+1$, in which case   we call $\p' G(\ld)$ the essential boundary of $G(\ld)$, and write      $\p' G =\p' G(1)$ and $\p' G^\ast =\p' G (2)$.

\begin{defn}\label{Def-2-1}   Let $\Og\subset \RR^{d+1}$ be a bounded domain with boundary $\Ga=\p \Og$, and let $G\subset \Og$ be a domain of special type. 
	We say $G$ is attached to $\Ga$ if $\overline{G}^\ast\cap \Ga =\overline{\p' G^\ast}$ 
	and there exists an open rectangular box $Q$ in $\RR^{d+1}$  such that $ G^\ast =Q\cap \Og$.
\end{defn}


\subsection{$C^2$-domains}

In this paper, we shall mainly  work on   $C^2$-domains, defined as follows: 

\begin{defn}
	A bounded   domain  $\Og\subset \RR^{d+1}$ is called  $C^2$ if there exist numbers $\da>0$, $M>0$ and a finite cover of the boundary $\Ga:= \p \Og$  by connected open sets $\{ U_j\}_{j=1}^J$ such that: {\rm (i)}  for every  $x\in\Og$ with   $\dist(x, \Ga)<\da$,  there exists an index $j$ such that $x\in U_j$,  and $\dist(x, \p U_j)>\da$;  {\rm (ii)} for each $j$ there exists a  Cartesian coordinate system $(\xi_{j,1},\dots, \xi_{j,d+1})$ in $U_j$ such that the set $\Og\cap U_j$ can be represented by the inequality $\xi_{j,d+1}\leq f_j(\xi_{j,1}, \ldots, \xi_{j, d})$, where $f_j:\RR^{d}\to \RR$ is a $C^2$-function satisfying
	$\max_{1\leq i, k\leq d}\|\p_i\p_k f_j\|_\infty\leq M.$
	
\end{defn}

\section{New moduli of smoothness on $C^2$ domains}\label{modulus:def}

Let $\Og\subset \RR^{d+1}$ be  the closure of an open,   connected, bounded  $C^2$-domain  in $\RR^{d+1}$ with boundary $\Ga=\p \Og$. In this section, we shall give the definition of  our    new moduli  of smoothness on  the   domain $\Og$.

The definition requires  a tangential  modulus of smoothness $\wt{\og}^r_G (f, t)_p$ on a domain $G\subset \Og$ of special type, which is described  below.  We start with an upward  $x_{d+1}$-domain $G$   given in~\eqref{2-7-special} with  $\xi=0$. 
Let 
$$\xi_j(x):=  {e_j + \p_j g(x) e_{d+1}}\in\RR^{d+1},\  \ j=1,\dots, d,\   \  x\in (-2b, 2b)^d.$$
Clearly, $\xi_j(x)$ is the tangent vector to the essential boundary $\p' G^\ast$ of $G^\ast$  at the point $(x, g(x))$ that is parallel to the $x_jx_{d+1}$-coordinate plane.  Given  a parameter    $A_0>1$,  we  set 
\begin{equation}\label{eqn:a0}
G^t:= \bl\{\xi\in G:\  \ \dist(\xi, \p' G) \ge A_0 t^2\br\},\   \ 0\leq t\leq 1.
\end{equation} 
We then   define 
the    $r$-th order tangential   modulus of smoothness  $\wt{\og}^r_G (f, t)_p$, ($0<t\leq 1$)  of $f\in L^p(\Og)$   by 
\begin{align}\label{modulus-special}
\wt{\og}^r_G (f, t)_p&:= \sup_{\sub{0<s\leq t\\
		1\leq j\leq d}} \Bl(\int_{G^t}
\Bl[\f 1{(bt)^d}  \int_{I_x(tb)} |\tr_{s \xi_j(u)}^r (f, \Og,(x,y))|^p  \, du\Br] dxdy\Br)^{\f1p},
\end{align}
where $I_x(tb):=\{ u\in (-b, b)^d:\   \  \|u-x\|\leq tb\}$, and we use $L^\infty$-norm to replace  the $L^p$-norm  when $p=\infty$.
For $t>1$, we define   $ \wt{\og}^r_G (f, t)_p=\wt{\og}^r_G (f, 1)_p$.
Next, if  $G\subset \Og$ is a general domain of special type,  then we  define the tangential  moduli $\wt{\og}^r_G (f, t)_p$ 
through the identity,
$$ \wt{\og}^r_G (f, t)_p=\wt{\og}^r_{T(G)} (f\circ T^{-1}, t)_p,$$
where $T$ is   a  composition of a  translation and the reflections $\sa_j, \tau_j$ for some $1\leq j\leq d+1$ which takes $G$ to an upward $x_{d+1}$-domain  of  the form~\eqref{2-7-special} with  $\xi=0$.

To define the  new moduli of smoothness on $\Og$, we also  need  the following covering lemma, which  will be proved in  Chapter~\ref{decom-lem}:


\begin{lem}\label{lem-2-1-18}
	There exists a finite cover of the boundary $\Gamma=\p \Og$ by domains of special type  $G_1, \dots, G_{m_0}\subset \Og$ that are attached to $\Ga$. In addition,  we may select the domains $G_j$   in such a way that    the  size of each  $G_j$  is as small as we wish, and  the  parameter  of  each  $G_j$  satisfies  the condition~\eqref{parameter-2-9}.
\end{lem}

Now we are in a position to define the    new moduli of smoothness    on $\Og$.
\begin{defn}\label{def:modulus}
	Given  $0<p\leq \infty$,  the $r$-th order  modulus of smoothness  of $f\in L^p(\Og)$ is defined by   
	\begin{equation}\label{eqn:defmodulus}
	\og_\Og^r(f, t)_p:=\og^r_{\Og, \vi} (f, t)_p +\og^r_{\Og, \tan} (f, t)_p,  
	\end{equation}
	where 
	$$
	\og_{\Og, \vi}^r (f,t)_p:=\max_{1\leq j\leq d+1} \og^r_{\Og, \vi} (f, t; e_j)_p\  \ \text{and}\  \ 
	\og_{\Og, \tan}^r (f, t)_p :=\sum_{j=1}^{m_0} \wt{\og}_{G_j}^r (f,t)_p.
	$$
	Here $G_1,\dots, G_{m_0}\subset 	\Og$ are  the domains of special type  from  Lemma~\ref{lem-2-1-18}.
\end{defn}

Note that  the second term on the right hand side of~\eqref{eqn:defmodulus} 
is defined via  finite differences   along certain  tangential directions of  the boundary $\Ga=\p \Og$. As a result, 
we call $\og^r_{\Og, \tan} (f, t)_p$  the tangential part of the $r$-th order modulus of smoothness on $\Og$.

We conclude this section with the following remark.

\begin{rem}\label{rem-3-2} The moduli of smoothness defined in this section rely on the parameter $A_0$ in~\eqref{eqn:a0}.  To emphasize  the dependence on this parameter,   we  often    write 
	$$\wt{\og}^r_G (f, t; A_0)_p:=\wt{\og}^r_G (f, t)_p,\   \    \  {\og}^r_{\Og} (f, t; A_0)_p:={\og}^r_{\Og} (f, t)_p.$$	
	By the Jackson theorem (Theorem~\ref{Jackson-thm}) and the univariate Remez inequality (see \cite{MT2}), it can be easily shown that given any two parameters $A_1, A_2\ge 1$,
	$$ {\og}^r_{\Og} (f, t; A_1)_p\sim {\og}^r_{\Og} (f, t; A_2)_p,\   \  t>0,\  \  0<p\leq \infty.$$	
\end{rem}

%

\section{Summary of main results}\label{summary}

In this section, we shall summarize the main results of this paper. As always, we assume that $\Og$ is the closure of  an open,  connected and  bounded $C^2$-domain in $\RR^{d+1}$. 
For simplicity, we identify with $L^\infty(\Og)$ the space $C(\Og)$ of continuous functions on $\Omega$.

First, the  main aim of  this paper is to prove the   Jackson type inequality and  the corresponding  inverse inequality for  the modulus of smoothness $\og_\Og^r(f,t)_p$  defined in~\eqref{eqn:defmodulus}, as  stated   in the following two theorems.

\begin{thm}\label{Jackson-thm} If   $r,n\in\NN$, $0<p\leq \infty$ and $f\in L^p(\Og)$,   then  
	$$ E_n (f)_{L^p(\Og)} \leq C \og_\Og^r(f, \f 1 n)_p,$$
	where the constant $C$ is independent of $f$ and $n$.  
	
\end{thm}

\begin{thm}\label{inverse-thm} If   $r, n\in\NN$,  $1\leq p\leq \infty$ and $f\in L^p(\Og)$,  then 
	$$\og_{\Og}^r (f, n^{-1})_{p} \leq\f{ C}{n^r}   \sum_{j=0}^n (j+1)^{r-1} E_j (f)_{L^p(\Og)},$$
	where the constant $C$ is independent of $f$ and $n$. 
\end{thm}

Note that  the Jackson inequality stated in  Theorem~\ref{Jackson-thm} holds  for the full range of $0<p\leq \infty$.

Second, we describe  two main ingredients in the proof of the direct Jackson theorem:   multivariate    Whitney type  inequalities  on certain  domains (not necessarily convex);  and localized  polynomial partitions  of   unity on $C^2$-domains. Both of these two results  are of independent interest.  

Indeed, the  Whitney type inequality  gives an upper estimate  for the error of local polynomial   approximation of a function via the behavior of its finite differences.
A   useful  multivariate Whitney type inequality was   established by Dekel and Leviatan  \cite{De-Le}  on a convex body (compact convex set with non-empty interior)  $G\subset \R^{d+1}$  asserting  that 
for any  $0<p\leq \infty$ ,   $r\in\NN$, and  $f\in L^p(G)$,
\begin{equation}\label{7-1-18-00}E_{r-1}  (f)_{L^p(G)} \leq C(p,d,r) \og^r (f, \diam (G))_{L^p(G)}.\end{equation}
It is remarkable that the constant $C(p,d,r)$ here  depends only on the three parameters $p,d,r$, but  is  independent of  the particular shape of the convex body $G$.  However, the  Whitney inequality~\eqref{7-1-18-00}  is NOT enough for our purpose because  our domain $\Og$ is not necessarily  convex, and the definition of our  moduli of smoothness   uses local  finite differences along a finite number of  directions only.
In Chapter~\ref{Whitney:sec}, we shall develop a new method to study  the following Whitney type inequality  for directional  moduli of smoothness on a more general domain $G\subset\RR^{d+1}$ (not necessarily convex):
\begin{equation}\label{7-8-18-00}E_{r-1} (f)_{L^p(G)} \leq C  \og^r(f, \diam(G); \mathcal{E})_p.\end{equation}
We  shall mainly work on the following two problems: (i)  the dependence of the constant $C$ in~\eqref{7-8-18-00} on certain geometric parameters of the  domain $G$;  and (ii) the minimum number of directions required in the set $\mathcal{E}\subset \SS^d$.
The  method developed in Chapter~\ref{Whitney:sec}  is simpler and can be used to deduce   the   Whitney type inequality on a  more   complicated   domain   from  the Whitney inequality  on  cubes or some other simpler domains. In particular,
our method  yields a simpler proof of the inequality~\eqref{7-1-18-00}  of \cite{De-Le}. Moreover,  in the case when $G$ is a compact convex domain  in $\R^2$,
we prove that   there exist two linearly independent unit vectors $\xi_1, \xi_2\in\R^2$  such that for all $0<p\leq \infty$ and  $f\in L^p(G)$, 
\begin{equation*}
E_{r-1}(f)_{L^p(G)}\leq C(p,r) \og^r (f, G; \{\xi_1, \xi_2\})_p,
\end{equation*}
where the constant $C(p,r)$ depends only on $p$ and $r$.

On the other hand, a polynomial partition of unity is a useful tool  to patch together local polynomial  approximation.  For simplicity,  we say 
a set  $\Ld$  in a metric space $(X,\rho)$   is   $\va$-separated  for some $\va>0$  if $\rho(\og, \og') \ge \va$ for any two distinct points $\og, \og'\in\Ld$,  and we call  an  $\va$-separated  subset $\Ld$  of $X$   maximal  if 
$\inf_{\og\in\Ld} \rho(x, \og)<\va$ for any $x\in X$. 
In Section~\ref{unity:sec}, relying on ideas by Dzjadyk and Konovalov~\cite{Dz-Ko}, we shall prove the following  localized  polynomial partitions of  unity on $C^2$-domains:
\begin{thm}\label{polyPartition00}Given  any parameter  $\ell>1$ and  positive integer $n$,   there exist a $\f 1n$-separated subset $\{\xi_j\}_{j=1}^{m_n}$ of $\Og$  with respect to the metric $\rho_\Og$ defined in~\eqref{metric}  and  a sequence  of   polynomials  $\{P_j\}_{j=1}^{m_n}\subset \Pi_{c_0 n}^{d+1}$   such that 	$\sum_{j=1}^{m_n} P_j (\xi) =1$  and 
	$  |P_j(\xi)| \leq C_1 (1+n\rho_\Og(\xi,\xi_j))^{-\ell}$, $j=1,2,\dots, m_n$ 
	for every $\xi\in \Og$,
	where the constants $c_0$ and $C_1$ depend only on $d$ and $\ell$.  	
\end{thm}

Third, we describe a new Bernstein inequality associated with the  tangential derivatives on the boundary $\Ga$, which plays a crucial role in the proof of our inverse theorem  (i.e., Theorem~\ref{inverse-thm}).
To this end, we need to introduce several notations. 
Given $\eta\in\Ga$, we denote by $\mathbf{n}_\eta$  the unit outer normal vector   to  $\Ga$ at  $\eta$, and $\mathcal{S}_\eta$  the set of all unit tangent vectors to $\Ga$ at $\eta$; that is, $\mathcal{S}_\eta:=\{ \pmb{\tau}\in\SS^d:\  \ \tau\cdot \mathbf{n}_\eta=0\}$.
Given nonnegative integers $l_1, l_2$ and a parameter $\mu\ge 1$, we define the maximal operators $\mathcal{M}_{n,\mu}^{l_1, l_2} $, $n=1,2,\dots$ by
$$ \mathcal{M}_{n,\mu}^{l_1, l_2} f(\xi):=\max_\eta \max_{\pmb{\tau}_\eta\in\mathcal{S}_\eta}
\bl|  \p_{\pmb{\tau}_\eta}  ^{l_1} \p_{\nb_\eta}^{l_2} f(\xi)\br|,\  \  \xi\in\Og,\   \  f\in C^\infty(\Og)$$
with    the left most maximum being  taken over all $\eta\in\Ga$ such that 
\begin{equation}\label{12-1-00}
\|\eta-\xi\|\leq  \mu \Bl(\sqrt{\dist(\xi, \Ga)} +n^{-1}\Br)=: \mu \vi_{n,\Ga} (\xi). 
\end{equation}
Here $\p_\xi^\ell=(\xi\cdot \nabla)^\ell$ for $\ell\in\NN$ and $\xi\in\RR^{d+1}$. 
In Section~\ref{sec:14}, we shall prove the following Bernstein inequality:

\begin{thm}\label{thm-10-1-00}  If     $f\in \Pi_n^{d+1}$,   $\mu>1$ and $0<p\leq \infty$, then 
	\begin{equation}\label{Bernstein-tan-0}
	\Bl\|\vi_{n, \Ga}^j \mathcal{M}_{n,\mu}^{r, j+l} f \Br\|_{L^p(\Og)}\leq C(\mu, p, d) n^{r+j+2l} \|f\|_{L^p(\Og)},\  \ r,j,l=0,1,\dots, 
	\end{equation}  
	where the function $\vi_{n,\Ga} $ is given in~\eqref{12-1-00}, and  the constant $C(\mu, p, d)$ is independent of $f$ and $n$.
\end{thm}


Fourth,  the Bernstein inequality stated in Theorem~\ref{thm-10-1-00}  allows us to establish Marcinkiewicz-Zygmund type inequalities and positive cubature formulas  on $\Og$.  To describe these results, let   $ U( \xi, t):= \{\eta\in\Og:\  \  \rho_\Og(\xi,\eta) \leq t\}$ for  $\xi\in\Og$ and $t>0$.  As can be easily shown,  
\begin{equation}\label{4-5-00}
|U(\xi, t)| \sim t^{d+1} ( t +\sqrt{\dist (\xi, \Ga)}),\  \  \xi\in\Og,\   \  t\in (0,1).
\end{equation}

In Chapter~\ref{sec:16}, we shall prove the following two results: 

\
\begin{thm}\label{cor-16-2-0}
	There  exists a number $\va_0\in (0,1)$ which depends only on $\Og$ and $d$  and has  the following property:  if $\Ld$ is a maximal $\f \va n$-separated subset of $\Og$ with respect to the metric $\rho_\Og$ with  $\va \in (0, \va_0)$, then for any $1\leq p\leq \infty$ and $f\in \Pi_n^{d+1}$,
	\begin{align*}
	C_1 \Bl(\sum_{ \xi\in\Ld} \Bl|U(\xi, \f \va n)\Br| \max_{\eta\in U(\xi, \va/n)} |f(\eta)|^p \Br)^{\f 1p}\leq \|f\|_{L^p(\Og)} \leq 	C_2 \Bl(\sum_{ \xi\in\Ld} \Bl|U(\xi, \f \va n)\Br| \min_{\eta\in U(\xi, \va/n)} |f(\eta)|^p \Br)^{\f 1p},
	\end{align*}
	with the usual change of $\ell^p$-norm when $p=\infty$,  where  the constants $C_1, C_2>0$ are independent of $f$, $\va$  and $n$.	
\end{thm}

\begin{thm}\label{cor-16-3-0}There  exists a number $\va_0\in (0,1)$ which depends only on $\Og$ and $d$  and has  the following property:   if $\Ld$ is a maximal $\f \va n$-separated subset of $\Og$ (with respect to the metric $\rho_\Og$) with $\va \in (0, \va_0)$,  then there exists a sequence $\{ \ld_\xi:\  \ \xi\in\Ld\}$ of positive numbers such that 
	$$\ld_\xi \sim \Bl|U(\xi, \f \va n)\Br|,\   \  \forall \xi\in\Ld,$$
	and 
	$$ \int_{\Og} f(\eta)\, d\eta =\sum_{\xi\in\Ld} \ld_\xi f(\xi),\   \   \  \forall f\in \Pi_n^{d+1}.$$

\end{thm}

In these results, the cardinality of $\Ld$ is of order $n^{d+1}$ (see Remark~\ref{rem:cardinality-max-sep}), which is optimal because the dimension of $\Pi_n^{d+1}$ is of the same order.

Fifth,  we  compare the  moduli of smoothness $\og_\Og^r(f,t)_p$  with 
the average $(p,q)$-moduli of smoothness  $ \tau_r (f, t)_{p,q}$ introduced by Ivanov \cite{Iv}.  
It turns out  that the  moduli  $\og_\Og^r(f,t)_p$ 
can be controlled above by the average moduli  $ \tau_r (f, t)_{p,q}$, as shown in the following theorem that will be proved in  Chapter~\ref{ch:IvanovModuli}:

\begin{thm}\label{thm-9-1-00}
	For any $0<q\leq p\leq \infty$ and $f\in L^p(\Og)$, 
	$$\og_\Og^r(f, t; A_0)_p \leq C \tau_r (f, c_0 t)_{p,q},\   \   0<t\leq 1,$$
	where the constant $C$ is independent of $f$ and $t$. 
\end{thm}

As an immediate consequence of Theorem~\ref{thm-9-1-00} and Theorem~\ref{Jackson-thm}, we  obtain a  Jackson type  inequality for the average moduli of smoothness  for  any dimension   $d\ge 1$ and  the full range of $0<q\leq p\leq \infty$.
\begin{cor}  \label{cor-4-4-0}	
	If $f\in L^p(\Og)$, $0< q\leq p \leq \infty$ and $r\in\NN$, then
	$$ E_n (f)_p \leq C_{r, \Og} \tau_r (f, \f {c_0} n)_{p,q}.$$
\end{cor}

As mentioned  in the introduction,  Corollary~\eqref{cor-4-4-0} for $1\leq p\leq \infty$ and $d=1$  was  announced  in  \cite{Iv}  for   a piecewise $C^2$-domain $\Og\subset \RR^2$.  

We shall   prove a corresponding  inverse theorem for the average moduli of smoothness $\tau_r(f, t)_{p,q}$ as well:

\begin{thm} \label{thm-15-1-00}If $r\in\NN$, $1\leq q\leq  p\leq \infty$ and $f\in L^p(\Og)$, then 
	$$\tau_r (f, n^{-1})_{p,q} \leq C_{r} n^{-r} \sum_{s=0}^n (s+1)^{r-1} E_s (f)_p.$$
\end{thm}

In the case when $\Og\subset \RR^2$ (i.e., $d=1$),  Theorem~~\ref{thm-15-1-00}  was announced   without detailed proofs  in  \cite{Iv}  for  the case   $p=\infty$  and  the case when  $1\leq p\leq \infty$ and   $\Og$ is a parallelogram or a disk.

The rest of the paper is organized as follows. Chapter~\ref{decom-lem} is devoted to the proof of the covering lemma,  Lemma~\ref{lem-2-1-18}, which gives a finite cover of the boundary by domains of special type attached to $\Ga$. This lemma  allows us to reduce  considerations near the boundary to certain problems on domains of special type.

Chapter~\ref{sec:reduction} - Chapter~\ref{ch:direct} are devoted to the proof of the Jackson theorem (Theorem~\ref{Jackson-thm}). To be more precise,  in Chapter~\ref{sec:reduction}, we reduce the proof of Theorem~\ref{Jackson-thm} to certain Jackson type estimates on domains of special type.  The proofs  of  these Jackson estimates also  require two  additional  results:  (i)  the Whitney type inequality  for directional  moduli of smoothness on  non-convex  domains, which is established in  Chapter~\ref{Whitney:sec};  and (ii) the localized  polynomial partitions of the  unity on $C^2$-domains, established in  Sections~\ref{sec:5}--\ref{unity:sec}.

In Chapter~\ref{ch:IvanovModuli}, we compare our moduli of smoothness $\og_\Og^r(f,t)_p$ with the average moduli of smoothness $\tau_r(f,t)_{p,q}$. The main result of Chapter~\ref{ch:IvanovModuli} is stated in   Theorem~\ref{thm-9-1-00}. 

Sections~\ref{sec:12}-\ref{sec:14} are devoted to the proof of the  Bernstein inequality stated in Theorem~\ref{thm-10-1-00}. 
The proof  is rather involved and is divided into several parts. Indeed,  Theorem~\ref{thm-10-1-00}  can be deduced from  certain Bernstein inequalities on domains of special type, which are proved in Chapter~\ref{sec:12} for $r=1$ and in Section~\ref{sec:13} for $r>1$. The proof of  Theorem~\ref{thm-10-1-00} is given in Section~\ref{sec:14}. 

In Chapter~\ref{sec:15}, we prove  the inverse theorems as stated in Theorem~\ref{inverse-thm} and  Theorem~\ref{thm-15-1-00}.

Finally, in 
the last chapter, Chapter~\ref{sec:16}, we prove Theorem~\ref{cor-16-2-0} and Theorem~\ref{cor-16-3-0},  which give   the  Marcinkiewicz-Zygmund type inequalities and positive cubature formulas  on compact $C^2$-domains.

\chapter{A covering  lemma}\label{decom-lem}

The main purpose in this chapter is to prove    Lemma~\ref{lem-2-1-18}, the covering  lemma  that is used in the definition of the moduli of smoothness  $\og_\Og^r(f,t)_p$  in Section~\ref{modulus:def}.

\begin{proof}[Proof of Lemma~\ref{lem-2-1-18}]Since $\Og$ is the closure of an open $C^2$-domain,  
	there exists a number  $r_0\in (0,1)$  depending only on $\Og$ such that    for each $\xi\in\Ga$, we can find a closed ball of radius $8(d+1)r_0$ in $\RR^{d+1}$  which   touches $\Ga$ at $\xi$  but lies entirely in  $\Og$.  
	
	Fix   $\xi=(\xi_x, \xi_y)\in \Ga$  temporarily, and let $n_\xi$  denote the unit outer normal to $\Ga$ at $\xi$.  
	Without loss of generality, we may assume that  $n_\xi\cdot~ e_{d+1}=\max_{1\leq i\leq d+1} |n_\xi\cdot e_i|$.  Then $n_\xi\cdot e_{d+1} \ge \f 1{\sqrt{d+1}}$.   By  the implicit function theorem,  there exist  an open  rectangular box  $V_\xi=\xi+\bl(I_\xi \times (-a_\xi', a_\xi)\br)$ with $I_\xi:=(-\delta_\xi, \delta_\xi)^d$, $\delta_\xi\in (0, r_0)$, $ a_\xi, a_\xi'>0$  and a $C^2$-function $h_\xi$ on $\RR^d$ such that 	$ \Ga_\xi:=V_\xi\cap \Ga=\xi+\{(x, h_\xi(x)):\  \  x\in I_\xi\},$ $0 =h_\xi(0)$ 
	and   $y\leq h_\xi(x)$ for every  $(x,y) \in (\Og\cap V_\xi)-\xi$.  
	By continuity, we may  choose the constant  $\delta_\xi$   small enough  so that   $n_\eta\cdot e_{d+1} \ge \f 1{2\sqrt{d+1}}$ for every $\eta\in \Ga_\xi$, and 
	\begin{equation}\label{5-1-0-18}
	\|\nabla h_\xi\|_{L^\infty(I_\xi)}\le 10^{-4} r_0 d^{-1} \delta_\xi^{-1}. 
	\end{equation}
	Next, we define
	\begin{align*}
	\Og_\xi :&=\xi+\{ (x, y)\in\RR^{d+1}:\  \  x\in  I_\xi,\   \  h_\xi(0) -6r_0 \leq y  \leq  h_\xi(x)\}.\end{align*}
	Note that $\Og_\xi$ can also be written in the form 
	$$\Og_\xi=	\varsigma +\Bl \{ (x, y):\  \  x\in  (-\delta_\xi, \delta_\xi)^d,\   \  0 \leq y  \leq  g_\xi(x)\Br\},		
	$$
	where     $\varsigma:=(\xi_x, h_\xi(0)-6r_0)$, and 
	$
	g_\xi(x):= h_\xi ( x)- h_\xi(0)+6r_0$ for $ x\in \RR^d$.
	We  claim that   $\Og_\xi\subset \Og$. 
	Indeed, if $\xi+(x,y)\in \Og_\xi$, then   $\eta=\xi+(x, h_\xi(x))\in\Ga_\xi$ and we can find a  closed  ball  $B_\eta\subset \Og$ of  radius  $8(d+1)r_0$  which  touches $\Ga$ at $\eta$.   
	A  simple geometric argument shows  that  if  $0\leq s \leq 16 \sqrt{d+1} r_0 (n_\eta \cdot e_{d+1})$, then $\eta-s e_{d+1}\in B_\eta\subset \Og$. 
	Since  $n_\eta\cdot e_{d+1} \ge \f 1{2\sqrt{d+1}}$,  it follows that $\eta -s e_{d+1}\in  \Og$ whenever   $0\leq s\leq 8\sqrt{d+1} r_0$. Now we write 
	$\xi+(x,y)=\eta -(h_\xi(x)-y) e_{d+1}$ and note that  by~\eqref{5-1-0-18}, 
	$$ 0\leq h_\xi(x)-y \leq h_\xi(x)-h_\xi(0) +6r_0 \leq (6+4\cdot 10^{-4})r_0 \leq 8\sqrt{d+1} r_0.$$
	This implies   $(x,y)\in  \Og$, and hence proves the claim.

	Next, 	we  define 	an upward $x_{d+1}$-domain $G_\xi$ by  $$G_\xi =\varsigma+  \Bl\{ ( x, y):  \  \ x\in (- b_\xi, b_\xi)^d,\   \  g_\xi(x)-L_\xi b_\xi<  y \leq g_\xi(x)\Br\},$$
	where 	$b_\xi\in (0, \da_\xi/2)$, and $L_\xi=\da_\xi / b_\xi$. We may choose $b_\xi$ small enough so that $L_\xi$ satisfies $$ L_\xi \ge 4\sqrt{d}  \|\nabla h_\xi\|_{L^\infty(I_\xi)} +1.$$
	Note that by~\eqref{5-1-0-18},  $g_\xi(x) >5r_0$ for $x\in [-2b_\xi, 2b_\xi]^d$. Since $L_\xi b_\xi=\da_\xi\leq r_0$,  it  follows that 
	$$ \k_\xi:= \min _{x\in [-2b_\xi, 2b_\xi]^d} g_\xi(x) -4L_\xi b_\xi>r_0> 0.$$ 
	Thus, 
	\begin{align*}G_\xi^\ast &=\varsigma+  \Bl\{ ( x, y):  \  \ x\in (- 2b_\xi,2 b_\xi)^d,\   \  \k_\xi<  y \leq g_\xi(x)\Br\} \subset \Og_\xi \subset \Og.\end{align*}
	Moreover,  it is easily seen that 
	$ \p G_\xi^\ast \cap \Ga =\overline{\p' G^\ast}$, and 
	$G_\xi^\ast =\Og\cap Q_\xi$ with 
	$Q_\xi=\varsigma + (-2b_\xi, 2b_\xi)^d \times (\k_\xi, a_\xi'')$, and $ a_\xi''=a_\xi -h_\xi(0) +6r_0$.  This shows that  $G_\xi\subset \Og$ is an upward  $x_{d+1}$-domain attached to $\Ga$.

	Finally, we may define the set $G_\xi$ for a general $\xi\in\Ga$ in a similar manner. Indeed,  if $1\leq j\leq d+1$ is such that $|n_\xi\cdot e_j| =\max_{1\leq i\leq d+1} |n_\xi\cdot e_i|$, then $G_\xi\subset \Og$ is an upward or downward $x_j$-domain attached to $\Ga$ according to whether $n_\xi\cdot e_j>0$ or $n_\xi\cdot e_j<0$. 
	Since  $\Ga=\bigcup_{\xi\in\Ga} \p' G_\xi$ and each  $\p' G_\xi$ is open relative to the topology of $\Ga$, we may
	find  a finite cover $\{\p' G_{\xi_i} \}_{i=1}^{m_0}$ of $\Ga$. This completes the proof of  Lemma~~\ref{lem-2-1-18}.
\end{proof}

%
%
%


%

\chapter{Geometric reduction near the boundary}\label{sec:reduction}

Our main goal in this chapter is to   show that    the Jackson inequality in   Theorem~\ref{Jackson-thm} can be deduced from  the following     Jackson-type estimates on domains  of special type.

\begin{thm}\label{THM-4-1-18} 	
	If $0<p\leq \infty$, and   $G\subset \Og$ is   an upward or downward $x_j$-domain  attached to $\Ga$ for some $1\leq j\leq d+1$,   then 
	\begin{equation}\label{Jackson:special}
	E_n (f)_{L^p (G)}\leq C \Bl[\wt{\og}_{G}^r (f, \f \tau n)_p+ \og_{\Og,\vi}^r (f, \f 1n;  e_j)_{p}\Br],\   
	\end{equation}
	where 
	the constants $C, \tau>1$ are independent of $f$ and $n$.  
\end{thm}


The proof of Theorem~\ref{THM-4-1-18} will be given in Section~\ref{Sec:8}. In this chapter, we  will  show how Theorem~\ref{Jackson-thm} can be deduced from Theorem~\ref{THM-4-1-18}. The idea of our proof is close to that in   \cite[Chapter 7]{To17}.  

\section{Lemmas and geometric reduction} 


We need  a series of  lemmas, the first of which gives a well known Jackson type estimate on a rectangular box. 

\begin{lem}\label{lem-6-1-0}\cite[Lemma 2.1]{Di96}
	Let $B$ be a compact  rectangular box  in $\RR^{d+1}$.  Assume that  $f\in L^p(B)$ if  $0<p< \infty$ and $f\in C(B)$ if $p=\infty$.
	Then for $0<p\leq \infty$, 
	$$\inf_{P\in \Pi_n^{d+1}} \|f-P\|_{L^p(B)} \leq C \max_{1\leq j\leq d+1}\og_{B,\vi}^r (f, \f 1n, e_j)_{p},$$
	where $C$ is independent of $f$ and $B$.
\end{lem}

Our second lemma is  a simple observation on domains of special type. 
Recall that  unless otherwise stated we  always assume that the parameter $L$
of a domain  of special type   satisfies the condition~\eqref{parameter-2-9}. 

\begin{lem}\label{lem-6-2:Dec} Let $G\subset \Og$ be an (upward or download) $x_j$- domain  of special type attached to $\Ga$ for some $1\leq j\leq d+1$. Then for  each parameter $\mu\in (\f12, 1]$, there exists an open rectangular box $Q_\mu$ in $\RR^{d+1}$ such that 
	\begin{equation}\label{eqn:decomp} \p' G(\mu) \subset S_{\mu, G}:=Q_\mu\cap \Og \subset G(\mu)\   \ \text{and}  \  \  \overline{Q_\mu}\subset Q_1\   \  \text{provided $\mu<1$}.   \  \end{equation}
\end{lem}

\begin{proof}
	Without loss of generality, we may assume that $G$ is given in~\eqref{2-7-special} with $\xi=0$.  Let
	$g_{\max}:=\max_{x\in [-b, b]^d} g(x)$ and $g_{\min} :=\min_{x\in [-b, b]^d} g(x)$.
	Using~\eqref{parameter-2-9}, we have 
	$$g_{\max}-g_{\min} \leq 2 \sqrt{d} b \max_{x\in [-b, b]^d}\|\nabla g(x)\|\leq \f 12 Lb.$$
	Thus, given each parameter $\mu\in (\f 12, 1]$, we may find a constant $a_{1,\mu}$ such that
	$$ g_{\max} -\mu L b < a_{1,\mu} <g_{\min}.$$
	We may choose the constant $a_{1,\mu}$ in such a way that $a_{1,1}<a_{1,\mu}$ if $\mu<1$. 
	On the other hand, since $G$ is attached to $\Ga$, we may find an open rectangular box $Q$ of the form $(-2b, 2b)^d\times (a_1, a_2)$ such that $G^\ast =Q\cap \Og$, where $a_1, a_2$ are two constants and    $a_2>g_{\max}$. 
	Let  $a_{2, 1} =a_2$ and let  $a_{2,\mu}$ be a constant so  that $g_{\max} < a_{2,\mu} <a_2$ for $\mu\in (\f12, 1)$.   Now  setting 
	\begin{equation}\label{6-3-Dec}
	Q_\mu :=(-\mu b, \mu b)^d\times (a_{1,\mu}, a_{2,\mu})\   \   \ \text{and}\   \ S_{\mu, G}: = Q_\mu \cap \Og,
	\end{equation}
	we obtain~\eqref{eqn:decomp}.
\end{proof}


\begin{rem}
	Note that~\eqref{eqn:decomp} implies that $\proj_j (Q_\mu) = \proj_j (G(\mu))$ for $\mu\in (\f 12, 1]$, where $\proj_j$ denotes the orthogonal projection onto the coordinate plane $x_j=0$.
\end{rem}

Now let $G_1,\dots, G_{m_0}\subset \Og$ be the domains of special type   in Lemma~\ref{lem-2-1-18}. 
Note that  for  every  domain $G$ of special type, its essential boundary  can be expressed as 
$\p ' G=\bigcup_{n=1}^\infty \p' G (1-n^{-1})$.  
Since $\Ga$ is compact and each $\p' G_j$ is open relative to the topology of $\Ga$, there exists $\ld_0\in (\f 12, 1)$ such that $\Ga =\bigcup_{j=1}^{m_0} \p' G_j (\ld_0)$. 
For convenience, we call $S\subset \Og$ an admissible subset of $\Og$ if either $S=S_{G_j, \ld_0}$ for some $1\leq j\leq m_0$ or $S$ is an open cube in $\RR^{d+1}$ such that $4S\subset \Og$. 

Our third lemma  gives a useful decomposition of the domain $\Og$.

\begin{lem}\label{LEM-4-2-18-0}  There exists  a sequence $\{\Og_s\}_{s=1}^J$ of admissible subsets of $\Og$ such that   $\Og=\bigcup_{j=1}^J \Og_j,$ and 
	$\Og_s \cap \Og_{s+1}$ contains an open ball of radius $\ga_0>0$ in $\RR^{d+1}$ 	for  each  $s=1,\dots, J-1$, where 
	the  parameters $J$ and $\ga_0$ depend only on the domain $\Og$. 	
\end{lem}

To state the fourth  lemma,  let  $\{\Og_s\}_{s=1}^{J}$  be the sequence of sets     in Lemma~~\ref{LEM-4-2-18-0}, and let   $H_m:=\bigcup_{j=1}^m \Og_j$ for $m=1,\dots, J$.
For $1\leq j\leq J$, define $\wh{\Og}_{j}=G_i$ if $\Og_j=S_{G_i, \ld_0}$ for some $1\leq i\leq m_0$;  and $\wh{\Og}_j =2Q$ if $\Og_j$ is an open  cube $Q$ such that $4 Q\subset \Og$.

\begin{lem}\label{REDUCTION}   
	If $0<p\leq \infty$ and $1\leq j<J$, then there exist constants $c_0,C>1$ depending only on $p$ and $\Og$  such that     
	$$ E_{c_0n} (f)_{L^p(H_{j+1})} \leq C \max \Bl\{ E_n (f)_{L^p(\wh{\Og}_{j+1})}, \   E_n (f)_{L^p (H_j)}\Br\}. $$
\end{lem}

Now we take Theorem~\ref{THM-4-1-18}, Lemma~\ref{LEM-4-2-18-0} and Lemma~\ref{REDUCTION}  for granted and proceed with the proof of  Theorem~\ref{Jackson-thm}.

\begin{proof}[Proof of Theorem~\ref{Jackson-thm}]
	Applying    Lemma~\ref{REDUCTION} $J-1$ times and recalling $H_J =\Og$, we obtain 
	\begin{equation}\label{6-3-0}E_{c_1n} (f)_{L^p(\Og)} \leq  C \max_{1\leq j\leq J} E_n (f)_{L^p (\wh{\Og}_j)},\end{equation}
	where $C, c_1>1$ depend only on  $p$ and $\Og$. 
	If ${\Og}_j =S_{G_i,\ld_0}$ for some $1\leq i\leq m_0$, then $\wh{\Og}_j =G_i$, and by  Theorem~\ref{THM-4-1-18}, 
	$$E_n (f)_{L^p(\wh{\Og}_j)} \leq \max_{1\leq i\leq m_0} E_n (f)_{L^p(G_i)}\leq C \og^r_\Og (f, n^{-1})_p.$$
	If  $\Og_j=Q$ is a cube such that $4Q\subset \Og$, then $\wh{\Og}_j =2Q$ and by Lemma~\ref{lem-6-1-0},
	$$E_n (f)_{L^p(2Q)}\leq C \max_{1\leq j\leq d+1}\og^r_{2Q,\vi} (f, n^{-1}; e_j)_p\leq C \og_{\Og,\vi}^r (f, n^{-1})_p,$$
	where the last step uses the fact that for any   $S\subset \Og$,  
	\begin{equation}\label{4-3-0-18}
	\max_{1\leq j\leq d+1}\og^r_{S,\vi} (f, t;  e_j)_{p}\leq C  \og_{\Og, \vi} ^r (f, t)_p.
	\end{equation}
	Thus, in either case, we have 
	$$ E_n(f)_{L^p(\wh{\Og}_j)}\leq C \og_{\Og}^r (f, n^{-1})_p.$$
	Theorem~\ref{Jackson-thm} then  follows  from the estimate~\eqref{6-3-0}.
\end{proof}

To  complete the reduction argument in this chapter, it remains  to prove  Lemma~\ref{LEM-4-2-18-0}  and    Lemma~\ref{REDUCTION}.

\section{Proof of Lemma~\ref{LEM-4-2-18-0}}
The proof of  Lemma~\ref{LEM-4-2-18-0}   is inspired by~\cite[p.~17]{To14} but written  in somewhat different language.
 Let $S_j =S_{G_j, \ld_0}$ for $1\leq j\leq m_0$.  Note that  $S_j$   is an  open neighborhood of $\p' G_j(\ld)$  relative to the topology of $\Og$.
	Since  $\p' G_j(\ld_0) \subset S_j\subset \Og$, and  $S_j$   is   open relative to the topology of $\Og$ for each $1\leq j\leq m_0$ ,  there exists $\va\in (0,r_0)$ such that  $$ \Ga_\va:=\{ \xi\in\Og:\  \ \dist (\xi,\Ga):<16\sqrt{d+1}\va\}\subset \bigcup_{j=1}^{m_0} S_j.$$
	Let us cover the remaining set $\Og\setminus \Ga_{\va}$ by finitely many  open cubes $Q_j$,  $j=m_0~+~1,\dots,  M_0$ of side length $\va$ such  that $4 Q_j\subset \Og$ for each $j$.
	Thus, setting $E_j=S_j$ for $1\leq j\leq m_0$, and $E_j=Q_j$ for $m_0<j\leq M_0$, 
	we have
	$ \Og=\bigcup_{j=1}^{M_0} E_j.$


	First, note that if $E_j\cap E_{j'}\neq \emptyset$ for some  $1\leq j,  j'\leq M_0$, then $E_j\cap E_{j'}$ must contain a nonempty open ball in $\RR^{d+1}$.  Indeed, 
	since  $E_j\cap E_{j'}$ is open relative to the topology of $\Og$, there exists an open set $V$ in $\RR^{d+1}$ such that $V\cap \Og=E_j\cap E_{j'}\neq \emptyset$. 
	Since $\Og$ is the closure of an open set in $\RR^{d+1}$,  the set $V\cap \Og$ must contain   an interior point of $\Og$.
	
	Next, we set  $\mathcal{A}=\{E_1,\dots, E_{M_0}\}$.  We say   two sets $A, B$ from the collection $\mathcal{A}$ are connected with each other  if there exists  a sequence of distinct sets $A_1, \dots, A_n$ from the collection $\mathcal{A}$  such that $A_1=A$, $A_n=B$ and $A_i\cap A_{i+1}\neq \emptyset$ for $i=1,\dots, n-1$, in which case we  write $[A:B]=\bigcup_{j=1}^n A_j$ and $(A: B)=\bigcup_{j=2}^{n-1} A_j$.
	We  claim   that every set in the collection $\mathcal{A}$ is connected with the set $E_1$.  Once this claim is proved, then  Lemma~\ref{LEM-4-2-18-0} will follow since 
	\begin{align*}\label{connected graph}
	\Og=\bigcup_{j=1}^{M_0} E_j = [E_1: E_2] \cup (E_2: E_1)\cup[E_1: E_3]\cup    \dots \cup [E_1: E_{M_0}]. 
	\end{align*}

	To show the claim, 
	let  $\mathcal{B}$ denote  the collection of all    sets $E_j$ from the collection $\mathcal{A}$   that are connected with $E_1$. 
	Assume that $\mathcal{A}\neq \mathcal{B}$. We obtain a contradiction as follows.  Let  $H:=\bigcup_{E\in\mathcal{B}} E$.  Then a set $E$ from the collection $\mathcal{A}$ is connected with $E_1$ (i.e., $E\in\mathcal{B}$)  if and only if $E\cap H\neq \emptyset$.
	Since $\mathcal{A}\neq \mathcal{B}$,   there exists $E\in\mathcal{A}$ such that $E\cap H=\emptyset$, which
	in particular, implies  that  $H$ is a proper subset of $\Og=\bigcup_{A\in \mathcal{A}} A$.
	Since $\Og$ is a connected subset of $\RR^{d+1}$, $H$ must have  nonempty boundary  relative to the topology of $\Og$. Let $x_0$ be a boundary point of $H$ relative to the topology of $\Og$. Since $H$ is open relative to  $\Og$,  $x_0\in \Og\setminus H=\bigcup_{A\in\mathcal{A}} A \setminus H$.  Let $A_0\in\mathcal{A}$ be   such that    $x_0\in A_0$. Then   $A_0$ is an open neighborhood of  $x_0$ relative to the topology of $\Og$, and hence   $A_0\cap H\neq \emptyset$, which  in turn implies   $A_0\in \mathcal{B}$ and $A_0\subset H$. But this is impossible as $x_0\notin H$.

%
%


\section{ Proof of  Lemma~\ref{REDUCTION}}

We now  turn to the proof of  Lemma~\ref{REDUCTION}.  The proof relies on three additional  lemmas. 
The first one  is similar to~\cite[Lemma~14.3]{To14}, however, we could not follow the conclusion of its proof in~\cite{To14}, where some averaging argument appears to be missing. Our proof below uses a multivariate Nikol'skii inequality which simplifies the transition to the multivariate case.
\begin{lem}\label{lem-4-1}
	If  $B$  is  a  ball in $\RR^{d+1}$ and $\ld>1$, then  for each $P\in\Pi_n^{d+1}$ and $0<q\leq \infty$,
	\begin{equation}\label{eq1}
	\|P\|_{L^q (\ld B)} \leq C_{d,q}(5 \ld)^{n+\f {d+1}q} \|P\|_{L^q (B)}.
	\end{equation}
\end{lem}
\begin{proof} By dilation and translation, we may assume that $B=B_1[0]$.~\eqref{eq1} with the   explicit constant  $(4\ld)^n$   was proved  in~\cite[Lemma~4.2]{To14} for $q=\infty$. 
	For $q<\infty$, we have 	
	\begin{align*}
	\|P\|_{L^q (\ld B)} &\leq C_d \ld^{\f{d+1}q} \|P\|_{L^\infty (\ld B)}\leq C_d \ld^{\f {d+1}q} (4\ld)^n \|P\|_{L^\infty (B)}\\
	&\leq C_{d,q} \ld^{\f {d+1}q} (4\ld)^n n^{\f {d+1}q} \|P\|_{L^q (B)}
	\leq C_{d,q}  (5\ld)^{n+\f {d+1}q} \|P\|_{L^q (B)},
	\end{align*}
	where we used H\"older's inequality in the first step,~\eqref{eq1}  for the already proven case $q=\infty$ in the second step, and Nikolskii's  inequality for algebraic polynomials on the unit ball (see~\cite{Da06} or \cite[Section~7]{Di-Pr16}) in the third step. 
\end{proof}

The second lemma  is probably well known. It can be proved in the same way as in~\cite[Lemma~4.3]{To14}.
\begin{lem}\label{lem-4-2} Let $I$ be a parallelepiped in $\RR^d$.  Then given parameters  $R>1$ and $\theta,\mu\in (0,1)$, there exists a polynomial  $P_n$ of  degree at most $ C(\theta, \mu, R, d) n$  such that  $ 0\leq P_n(\xi)\leq 1$ for  $ \xi\in B_R[0]$,  
	$ 1-P_n(\xi) \leq \theta^n$ for  $\xi\in \mu I$, and 
	$P_n(\xi)\leq \theta^n$ for $\xi\in  B_R[0]\setminus I$,
	where  $\mu I$ denotes the dilation of $I$  from its center by a factor $\mu$.
\end{lem}

As a consequence of Lemma~\ref{lem-4-2}, we have

\begin{lem}\label{lem-4-3}    Let   $G\subset \Og$  be   a domain of special type attached to $\Ga$, and $S_{G, \mu}: =\Og \cap Q_\mu$ be  as defined in Lemma~\ref{lem-6-2:Dec} with  $\mu\in (\f 12, 1]$. Let $R\ge 1$ be such that $Q_1 \cup \Og \subset B_R[0]$. 
	Then given   $\ld\in (\f 12, 1)$ and $\t \in (0,1)$, there exists a polynomial $P_n$ of degree at most $C(d,  \t,  R,  G, \ld)
	n$ with the properties that  $0\leq P_n(\xi)\leq 1$ for $\xi\in B_R[0]$,  $1-P_n(\xi)\leq \t^n$ for $\xi \in S_{G,\ld}$ and $P_n(\xi)\leq \t^n$ for $\xi\in \Og\setminus  S_{G,1}$.	\end{lem}

\begin{proof}
	Since $\ld<1$ and  $Q_\ld$ is an open  rectangular box such that  $\overline{Q_\ld}\subset Q_1$, it follows by Lemma~\ref{lem-4-2} that  there exists a  polynomial $P_n$  of degree at most $Cn$ such that 
	$0\leq P_n(\xi)\leq 1$ for all $\xi\in B_{R} [0]$,  $1-P_n(\xi) \leq \ta^n$ for all $\xi\in Q_\ld$ and $P_n(\xi) \leq \ta^n$ for all $\xi\in B_R[0]\setminus Q_1$.
	To complete the proof, we just need to observe that 
	$$ \Og \setminus S_{G,1} =\Omega \setminus (Q_1 \cap \Og)=\Omega\setminus Q_1\subset B_R[0]\setminus Q_1.$$
\end{proof}

We are now in a position to prove Lemma~\ref{REDUCTION}.

\begin{proof}[Proof of Lemma~\ref{REDUCTION}]

	The proof is essentially a repetition of that of~\cite[Lemma~4.1]{To14} or~\cite[Lemma~3.3]{To17} for our situation.   Let $R>1$ be such that $\Og\subset B_R[0]$, and set  $\theta:=\min\{\frac{\ga_0}{5R}, \f12\}$.   Write $H=H_j$ and  $S=\Og_{j+1}$.
	Without loss of generality, we may assume that $S=S_{G,\ld_0}$  for some  domain $G$ of special type attached to $\Ga$. (The  case  when  $S$ is a cube $Q$ such that $4Q\subset \Og$ can be proved similarly using Lemma~\ref{lem-4-2} instead of Lemma~\ref{lem-4-3}).  Then $S_{G,\ld_0}\cap H$ contains a ball $B$  of radius $\ga_0$. 	 
	By  Lemma~\ref{lem-4-3}, there exists  a polynomial $R_n$ of degree $\leq C(d, R, G)n$ such that
	$0\leq R_n(x)\leq 1$ for all $x\in B_{R}[0]$, $R_n(x)\leq \theta^{-n}$ for $x\in \Og\setminus  S_{G,1}$ and $1-R_n(x)\leq \theta^{-n}$ for $x\in  S_{G,\ld_0}$.
	Let 
	$P_1, P_2\in\Pi_n^{d+1}$ be such that 
	$$ E_n(f)_{L^p(S_{G,1})} =\|f-P_1\|_{L^p(S_{G,1})}\   \  \text{and}\  \   \   E_n(f)_{L^p(H)} =\|f-P_2\|_{L^p(H)}.$$
	Define  
	$$ P(x):=R_n(x) P_1(x) +(1-R_n(x)) P_2(x)\in \Pi_{cn}^{d+1}.$$
	Then  
	\begin{align*}
	E_{cn} (f)_{L^p(H_{j+1})}&\leq 
	\|f-P\|_{L^p(H\cup S_{G,\ld_0})}\\
	& \leq  \|f-P\|_{L^p(H\cap S_{G,1})}+\|f-P\|_{L^p(H\setminus S_{G,1})}+\|f-P\|_{L^p(S_{G,\ld_0})}.
	\end{align*}  
	
	First, we can estimate the term $\|f-P\|_{L^p(H\cap S_{G,1})}$ as follows: 
	\begin{align*}
	\|f-P\|_{L^p(H\cap S_{G,1})}&=\|R_n (f-P_1) +(1-R_n) (f-P_2)\|_{L^p(H\cap  S_{G,1})}\\
	&\leq C_p \max\Bl\{\|f-P_1\|_{L^p(S_{G,1})}, \  \|f-P_2\|_{L^p(H)}\Br\}\\
	&\leq C_p \max \Bl\{ E_n (f)_{L^p(S_{G,1})}, E_n(f)_{L^p(H)} \Br\}.
	\end{align*}
	
	Second, we show  
	\begin{align}\label{6-7-00}
	\|f-P\|_{L^p(H\setminus S_{G,1})}\leq  & C_{p,\ga_0,R} \max \Bl\{ E_n (f)_{L^p(S_{G,1})}, E_n(f)_{L^p(H)} \Br\}.
	\end{align}
	Indeed,   we have 
	\begin{align}
	\|f-P\|_{L^p(H\setminus S_{G,1})}&=\|( f-P_2) + R_n (P_2-P_1)\|_{L^p(H\setminus S_{G,1})}\notag\\
	&\leq C_p  E_n(f)_{L^p(H)}+ C_p \theta^{n} \|P_1-P_2\|_{L^p(\Og)}.\label{second}\end{align}
	However, by  Lemma~\ref{lem-4-1}, 
	\begin{align}
	\|P_1-P_2\|_{L^p (\Og)}& \leq \|P_1-P_2\|_{L^p (B_{R}[0])}\leq C\Bl( \f {5 R}{\ga_0}\Br)^{n+\f {d+1}p} \|P_1-P_2\|_{L^p(B)}\notag\\
	&\leq C\Bl( \f {5 R}{\ga_0}\Br)^{n+\f {d+1}p} \|P_1-P_2\|_{L^p(H\cap S_{G,1})}\notag\\
	&\leq C(R, d, \ga_0,p) \theta^{-n} \max\Bl\{E_n(f)_{L^p(S_{G,1})}, \  E_n(f)_{L^p(H)}\Br\}.\label{3-2-eq}
	\end{align}
	Thus, combining~\eqref{second} with~\eqref{3-2-eq}, we obtain~\eqref{6-7-00}.
	
	Finally, we estimate the term $\|f-P\|_{L^p(S_{G,\ld_0})}$  as follows: 
	\begin{align*}
	\|f-P\|_{L^p(S_{G,\ld_0})} & =\|f-P_1 +(1-R_n) (P_1-P_2)\|_{L^p(S_{G,\ld_0})} \\
	&\leq C_p \|f-P_1\|_{L^p(S_{G,1})}  + C_p \theta^{n} \|P_1-P_2\|_{L^p (\Og)}\\
	&\leq C_{p,\ga_0,R}\max\Bl\{E_n(f)_{L^p(S_{G,1})}, \  E_n(f)_{L^p(H)}\Br\},
	\end{align*}
	where the last step uses~\eqref{3-2-eq}.
	
	Now putting the above estimates together, and noticing $S_{G,1} \subset G=\wh{\Og}_{j+1}$, we  complete the proof of  Lemma~\ref{REDUCTION}.
\end{proof}

%
%
%


%
%
%

\chapter{Multivariate Whitney type inequalities}\label{Whitney:sec}

\section{Whitney inequalities for regular and directional moduli}

The  Whitney inequality gives an upper estimate  for the error of local polynomial   approximation of a function via the behavior of its finite differences.
The  constant in the Whitney inequality, called the Whitney constant,  plays a very important role  in various applications  for obtaining global upper estimates on the errors of  approximation.  In the case of one variable,   sharp Whitney constants  had been studied extensively in literature  (see~\cite{IT, GKS, Sen} and the references therein).  A remarkable    multivariate Whitney type inequality  was   established in  \cite{De-Le}  on a  general convex body  $G\subset \R^{d+1}$  asserting  that 
given $0<p\leq \infty$ and  $r\in\NN$, there exists a constant $C\equiv C(d,r, p) $ depending only on $d,p,r$ such that for every  function $f\in L^p(G)$,
\begin{equation}\label{7-1-18}
E_{r-1}  (f)_{L^p(G)} \leq C \og^r (f, G)_p.
\end{equation}
(Recall that $\og^r (f, G)_p=\og^r (f, \diam (G))_p$.) The crucial  point here   lies in the fact that  the constant $C$   does not depend on the shape of the convex body $G$. 

The   Whitney type inequality~\eqref{7-1-18} on convex domains will not be  enough for our purposes in this paper. 
The main goal in this chapter is to establish  a more general version of this result for  directional  moduli of smoothness on certain non-convex  domains. In particular, we will give a simpler proof of the inequality~\eqref{7-1-18}  of \cite{De-Le}.

The inequality~\eqref{7-1-18} is, in fact, an equivalence. Indeed, as $\og^r (Q, G)_p=0$ for any $Q\in\Pi^{d+1}_{r-1}$, choosing $Q$ to be the polynomial of best approximation to $f$, we get 
\[
\og^r (f, G)_p=\og^r (f-Q, G)_p\le C\|f-Q\|_{L^p(G)}\le C E_{r-1}  (f)_{L^p(G)}.
\]
For directional moduli, we may have that $\og^r (Q, G; \mathcal{E})_p=0$ is valid for a wider space of polynomials $Q$ than the polynomials of total degree $<r$. Therefore, a proper generalization of~\eqref{7-1-18} for directional moduli calls for approximation by polynomials from the space $\Pi^{d+1}_{r-1}(\mathcal{E})$ which we will now define and which may be different from $\Pi^{d+1}_{r-1}$.

For a set of directions $\mathcal{E}\subset\SS^{d}$ satisfying $\spn(\EEE)=\RR^{d+1}$, we define $\Pi^{d+1}_{r-1}(\mathcal{E})$ as the set of all algebraic polynomials $Q$ on $\RR^{d+1}$ such that for any fixed $x\in\RR^d$ and $e\in\EEE$ the function $Q_{x,e}(t):=Q(x+te)$ is an algebraic polynomial of degree $<r$ in $t\in\RR$. It is not hard to observe that $\Pi^{d+1}_{r-1}(\mathcal{E})$ is, in fact, a finite dimensional space of functions, moreover, $\Pi^{d+1}_{r-1}(\mathcal{E})\subset \Pi^{d+1}_{(d+1)(r-1)}$, which follows from the fact that polynomials of degree $<r$ in each variable have total degree $\le(d+1)(r-1)$. We also define the corresponding modification of the error of approximation:
\[
E_{r-1}(f;\EEE)_{L^p(G)}:=\inf\Bl\{ \|f-Q\|_p:\   \  Q\in \Pi^{d+1}_{r-1}(\EEE)\Br\}.
\]
With the new notations, an appropriate generalization of~\eqref{7-1-18} will have the form
\begin{equation}\label{eqn:new whitney type}
E_{r-1}  (f;\EEE)_{L^p(G)} \leq C \og^r (f, G;\EEE)_p
\end{equation}
for which the converse can be verified to be valid under certain mild and natural assumptions on $G$ and $\EEE$.

\section{General results}

Our method is   
based on the following result, showing  that   the   Whitney type inequality on a geometrically  complicated   domain can be deduced from that on  certain  relatively  simpler subdomain:  

\begin{thm} \label{lem-3-3} Assume that    $r\in\NN$, and  $K, J\subset \RR^{d+1}$ are two  measurable sets satisfying the following condition for a vector 
	$h\in\RR^{d+1}\setminus \{0\}$:
	
	\begin{equation}\label{7-2-0-18}
	J\subset  \bigcap_{j=1}^r (K+jh)	\  \  \text{and}\  \   \text{$[\xi-rh, \xi] \subset K\cup J$,\  \      $\forall \xi\in J$. }
	\end{equation}  	 	  
	Then for each   $0<p\leq \infty$,  $f\in L^p(K\cup J)$ and $\EEE\subset\SS^{d+1}$ satisfying $\spn(\EEE)=\RR^{d+1}$ and $e\in\EEE$  with $e:=h/\|h\|$ and $\ta:=\min\{p,1\}$, we have 
	\begin{align}\label{5-2-eq-18}
	E_{r-1} (f;\EEE)^\ta_{L^p(K\cup J)} &\leq \og^r(f, K\cup J; e)_p ^\ta + 2^r E_{r-1}(f;\EEE)_{L^p(K)}^\ta.
	\end{align}
\end{thm}
\begin{proof}     Note first that  for each function   $F:K\cup J \to \RR$,
	\begin{align*}
	(-1)^r\tr_h^r F (\xi-rh)=\tr_{-h}^r  F(\xi) =\sum_{j=0}^r (-1)^j \binom{r} j  F(\xi-jh),\   \ \xi\in J.
	\end{align*}
	It follows that
	$$|F(\xi)| \leq |\tr_{h}^r F(\xi-rh)|+  \sum_{j=1}^r\binom{r}j |F(\xi-jh)|,\   \  \xi\in J.$$
	Taking the $L^p$-norm over the set $J$ on both sides of this last inequality , and keeping in mind  that $J\subset \bigcap_{j=1}^r (K+jh)$,  we obtain
	$$ \|F\|_{L^p(J)}^\ta \leq  \|\tr_{h}^r F\|^\ta_{L^p (J-rh)}  +  (2^r-1)\|F\|^\ta_{L^p(K)}.$$
	This  in turn implies that
	\begin{equation}\label{2-3-18-55}
	\|F\|_{L^p(K\cup J)}^\ta \leq  \|\tr_{h}^r F\|^\ta_{L^p (J-rh)}  + 2^r \|F\|_{L^p(K)}^\ta.
	\end{equation}
	Let $q\in\Pi_{r-1}^{d+1}(\EEE)$ be such that $E_{r-1} (f) _{L^p(K)}=\|f-q\|_{L^p(K)}$. Then 
	applying~\eqref{2-3-18-55} with  $F=f-q$, we obtain the desired inequality~\eqref{5-2-eq-18}.
\end{proof}

\begin{rem}\label{rem:whitney key tool seq appl}
	According to Theorem~\ref{lem-3-3}, 	for a  pair of sets $(K, J)$ satisfying~\eqref{7-2-0-18}, provided $e\in\EEE$, the  Whitney type inequality   on the set $K$,
	\begin{equation}\label{whitney} E_{r-1} (f;\EEE)_{L^p(K)} \leq M \og^r (f, K; \mathcal{E})_p,\end{equation}
	implies the following   Whitney type inequality on  the   larger  set $K\cup J$:
	$$ E_{r-1} (f;\EEE)_{L^p(K\cup J)} \leq  C_{p,r} \Bl[ \og^r (f, K\cup J; e)_p+ M \og^r(f, K;\mathcal{E})_p\Br].
	$$
	In general,  	 	 	to  establish a Whitney type inequality on a   domain $G\subset \RR^{d+1}$, we   often need to  apply  Theorem~\ref{lem-3-3} iteratively, decomposing $G$ as a finite union  $G=\bigcup_{j=0}^m K_j$  with  $K_j\subset K_{j+1}$ and  the pairs of the sets 
	$(K,J):=(K_j,K_{j+1}\setminus K_j)$   satisfying the condition~\eqref{7-2-0-18} for possibly different  vectors  $h_j$,  $j=0,1,\dots, m-1$. We need to assume that the directions of $h_j$ are in $\EEE$, so if we require too many directions, we will end up with $\Pi^{d+1}_{r-1}(\EEE)=\Pi^{d+1}_{r-1}$.  With our approach, we often select  an   initial set  $K_0$  on which     a Whitney type inequality    holds.  For instance,  we may choose $K_0$ to be  any  compact 
	parallelepiped in $\RR^{d+1}$,  according to  Lemma~\ref{lem-5-3-0}  stated below.
\end{rem}

\begin{rem} \label{rem-7-2}The  above proof of Theorem~\ref{lem-3-3} actually yields a more general  estimate under a  weaker assumption. Indeed, it follows by~\eqref{2-3-18-55} that if $J\subset  \bigcap_{j=1}^r (K+jh)$ and $h/\|h\|\in\EEE$, then   for each $f\in L^p(K\cup J)$ and  $q\in \Pi_{r-1}^{d+1}(\EEE)$,
	\begin{align}\label{5-2-eq-0-18}
	\|f-q\|^\ta_{L^p(K\cup J)} &\leq \| \tr_{h}^r f\|^\ta_{L^p({J-rh})}  + 2^r \|f-q\|^\ta_{L^p(K)},\   \   \text{with $\t:=\min\{p,1\}$}.
	\end{align} 
	Note here that   we do not assume that $[\xi-rh, \xi] \subset K\cup J$ for each    $\xi\in J$.
\end{rem}

\begin{lem}\label{lem-5-3-0}\cite[Lemma 2.1]{Di96} Given $0<p\leq \infty$ and $r\in\NN$, there exists a constant $C(d,p,r)$ such that for every   compact 
	parallelepiped $S\subset \RR^{d+1}$ and each  $f\in L^p(S)$, 
	$$E_{r-1}(f;\EEE(S))_{L^p(S)} \leq C(d,p,r) \og^r(f, \diam (S), \mathcal{E}(S))_{L^p(S)},$$
	where $\mathcal{E}(S)$ denotes  the set of all edge directions of $S$. 
\end{lem}

\begin{proof}  Lemma~\ref{lem-5-3-0} was proved in \cite[Lemma 2.1]{Di96} in the case when $S$ is a rectangular box. The  more  general  case follows from~\eqref{5-1-eq}.
\end{proof}


For the  remainder  of this chapter, we will show how to apply   Theorem~\ref{lem-3-3} to establish   Whitney type inequalities    on  various   multivariate domains with constants depending only on certain given parameters.

We start with a  result that   will be used  in the next chapter.

\begin{thm}\label{cor-7-3}
	Let $ G:=\{(x, y):\  \  x\in D,\  \  g_1(x)\leq y\leq g_2(x)\}$ be a domain in $\R^{d+1}$ that lies  between the graphs of two functions $y=g_1(x)$ and $y=g_2(x)$ on  a   compact parallelepiped $D$ in $\R^d$.  Assume that there exist a linear function 
	$y=H(x)$ on $\RR^d$, a number $\da>0$  and a parameter $L>1$ such that 
	$$\da \leq (-1)^i(g_i(x)-H(x)) \leq L\da\quad \forall x\in D, \quad i=1,2.$$
	In other words, $S\subset G\subset S_L$, where $S$ and $S_L$ are the two   compact 
	parallelepipeds  given by 
	\begin{align*}
	S:&=\{(x,y):\  \  x\in D,\  \  H(x)-\da\leq y\leq H(x)+\da\},\\
	S_L:&= \{(x,y):\  \  x\in D,\  \  H(x)-L\da\leq y\leq H(x)+L\da\}.
	\end{align*}  	  
	Then for each $r\in\NN$, $0<p\leq \infty$ and $f\in L^p(G)$,
	$$E_{r-1} (f;\EEE\cup\{e_{d+1}\})_{L^p(G)} \leq C(p,d, r, L)\Bl[\og^r (f, S; \mathcal{E})_p+\og^r (f, G; e_{d+1})_p\Br],$$
	where 
	$\mathcal{E}$ is the set of edge directions of the  parallelepiped $D$.
\end{thm}

\begin{proof}
	Define $G_+:=\{(x,y)\in G: H(x)-\delta< y\le g_2(x)$.
	Since  $S\subset G_+\subset S_L$, 	we may  decompose the set $G_+$ as 
	$G_+=\bigcup_{j=1}^{(L-1)r} K_j$, where 
	$K_0:=S$, and for $j\ge1$,
	\begin{align*}
	K_j:&=\Bl\{(x,y)\in G:\  \ H(x)-\da \leq y\leq H(x)+\da+\f {j\da}{r}\Br\}.
	\end{align*}
	According to the assumptions on the set $G$,  
	for  $0\leq j\leq (L-1)r-1$,  $\xi\in K_{j+1}\setminus K_j$ implies that $$ 
	[\xi-\da e_{d+1}, \xi-r^{-1} \da e_{d+1}]\subset K_j\  \  \text{and}\  \  [\xi-\da e_{d+1}, \xi]\subset K_{j+1}.$$
	Since $K_0=S$ is a compact parallelepiped, using Lemma~\ref{lem-5-3-0}, and applying Theorem~\ref{lem-3-3} (see also Remark~\ref{rem:whitney key tool seq appl}) iteratively to the pairs of sets $K=K_j$ and $J=K_{j+1}\setminus K_j$, $j=0,\dots, (L-1)r-1$ with $h=\f \da r e_{d+1}$, we obtain the Whitney inequality as stated in Theorem~~\ref{cor-7-3} with $G_+$ in place of $G$. Repeating the same procedure for $G_-:=\{(x,y)\in G: g_1(x)< y\le H(x)+\delta\}$ completes the proof.
\end{proof}

\begin{rem}\label{rem-7-4} Clearly in the above proof of Theorem~\ref{cor-7-3}, we may choose the  initial set $K_0$ to be  any set $K$ containing the parallelepiped $S$  on which the  Whitney type inequality holds.  Thus,  
	if the  Whitney type inequality~\eqref{whitney} holds on a  set $K\supset S$, then  
	$$E_{r-1} (f;\EEE\cup\{e_{d+1}\})_{L^p(G\cup K)} \leq C(p,d,  r, L)\Bl[\og^r (f, G\cup K, e_{d+1})_p+ M\og^r (f,  K)_p\Br].$$
\end{rem}

For convenience, we introduce the following definition.

\begin{defn} 
	A set   $H\subset \RR^{d+1}$ is called a  special  directional  domain  in the direction of  $e\in \SS^d$ with parameter $L>1$ if   there exists a  rotation $\pmb{\rho}$ such that $\pmb{\rho} e=e_{d+1}$, and the set $G=\pmb{\rho} (H)$ satisfies the conditions of Theorem~\ref{cor-7-3} with    $ S\subset  \RR^{d+1}$ being a cube in $\R^{d+1}$  and $g_1$ a constant function, in which case we call the set  $\pmb{\rho}^{-1} (S)$  the base of $H$.    	
\end{defn}

Theorem~\ref{cor-7-3}, in particular,  allows us to give a simpler proof of the following  result of \cite{De-Le} on convex domains.  

\begin{cor}\label{De-Le} \cite{De-Le} Given $r\in\NN$ and $0<p\leq \infty$, there exists a constant $C(p, d,r)$ depending only on $p, d, r$ such that for every compact  convex body $G\subset \RR^{d+1}$, and every $f\in L^p(G)$,
	$$ E_{r-1}(f)_{L^p(G)} \leq C(p,d,r) \og^r (f, G)_p.$$
\end{cor}

\begin{proof}
	
	By  John's theorem (\cite[Proposition 2.5]{De-Le}),
	without loss of generality	we may assume  that $B_1[0] \subset G\subset B_{d+1}[0]$.
	Following  \cite{Di-Pr08}, we may   decompose $G$ as a finite union $G=\bigcup_{j=0}^m G_j$ of $(m+1)\leq C_d$ special directional domains $G_j$ with parameters $\leq C_d$ such that $G_0:=[-\f 1{\sqrt{d+1}}, \f 1{\sqrt{d+1}}]^{d+1}$, and   the  base of  each   $G_j$ is contained in $G_0$. Indeed, the decomposition  can be constructed as follows. First,   cover the sphere $\{ \xi\in\RR^{d+1}:\  \ \|\xi\|=d+1\}$ with 
	$m\leq C_d$ open balls  $B_1, \dots, B_m$  of radius $1/ (2(d+1))$ in $\RR^{d+1}$.    Denote by $\xi_{j}$ the unit vector pointing from the origin to the center of the ball $B_j$, and let $I_j$ denote a $d$-dimensional cube  with  side length $1/(d+1)$ that is centered at the origin and perpendicular to $\xi_j$. Define $ G_0:=[-\f 1{\sqrt{d+1}}, \f 1{\sqrt{d+1}}]^{d+1}$ and 
	$$ G_j:=\Bl\{ \eta+t\xi_j\in G:\  \  \eta\in I_j,\   \  t\ge 0\Br\},\  \ j=1,\dots, m.$$
	Since $G$ is convex and $B_1[0]\subset G\subset B_{d+1}[0]$, it is easily seen  that  $G=\bigcup_{j=0}^m G_j$ and for each  $1\leq j\leq m$, the set  $G_j$ is a special domain in the direction of $\xi_j$ with parameter $\leq C_d$ and a base parallelepiped
	$$S_j:=\{\eta+t\xi_j:\  \ \eta\in I_j, \  \  0\leq t\leq \sqrt{d}/(2(d+1))\}\subset G_0,$$
	where we used that $I_j\subset B_{\sqrt{d}/(2(d+1))}[0]$, $\sqrt{d}/(2(d+1))<1/\sqrt{d+1}$ and $B_{1/\sqrt{d+1}}[0]\subset G_0$.
	
	Finally, on $G_0$ we can invoke Lemma~\ref{lem-5-3-0}, then with $\EEE=\SS^d$, applying Theorem~\ref{lem-3-3} to the  decomposition $G=\bigcup_{j=0}^m G_j$ iteratively, and using Theorem~\ref{cor-7-3} and Remark~\ref{rem-7-4}, we  deduce   the inequality~\eqref{7-1-18} with constant $C(p,d,r)$ depending only on $p,d,r$. 
\end{proof}

The above proof also yields the Whitney inequality for the directional moduli of smoothness. However,  the number of directions required in the directional modulus  would be  very large (in fact, it is of the order $C d^{d}$ as $d\to \infty$), and the corresponding polynomial space would be simply $\Pi^{d+1}_{r-1}$.

\section{On choices of directions}

For the purpose of obtaining the direct estimate of polynomial approximation on $C^2$ domains, it will be sufficient to use Theorem~\ref{cor-7-3}. We believe it is an interesting question of independent interest to study for which domains $G$ and sets of directions $\EEE$ the Whitney-type inequality~\eqref{eqn:new whitney type} is valid. We will present some results here employing our technique and make some remarks hopefully not taking us too far from the main subject of this work.

\begin{rem} If the condition $\text{span}(\mathcal{E})=\RR^{d+1}$ in the definition of $\Pi^{d+1}_{r-1}(\EEE)$ is not assumed, then $\Pi^{d+1}_{r-1}(\EEE)$ may be an infinite-dimensional space. Indeed, if $\spn(\mathcal{E})\neq \RR^{d+1}$, then there exists $\xi_0\in \SS^d$ such that $\xi_0\cdot \xi=0$ for all $\xi\in\mathcal{E}$. Each function $f_j(\eta):=(\xi_0\cdot \eta)^{j}$, $j\ge0$, is an element of $\Pi^{d+1}_{r-1}(\EEE)$. In this situation, the element of best approximation by $\Pi^{d+1}_{r-1}(\EEE)$ may not exist and the closure of $\Pi^{d+1}_{r-1}(\EEE)$ may be rather large, e.g., will include all functions $\tilde f(\eta):=f(\xi_0\cdot \eta)$ for $f\in L^p(\R)$.
\end{rem}	 

\begin{rem}
	On the other hand, if $\EEE$ is sufficiently large and contains directions which are in a certain sense independent, then $\omega^r(f,G;\EEE)_p$ may be equivalent to $\omega^r(f,G)_p$ and $\Pi^{d+1}_{r-1}(\EEE)=\Pi^{d+1}_{r-1}$ making~\eqref{eqn:new whitney type} a regular Whitney-type inequality for polynomials of total degree $<r$ and regular (non-directional) modulus of smoothness. We refer the reader to~\cite{Di-Iv} where such equivalence is studied in the case $p\ge1$.
\end{rem}

\begin{rem}\label{rem:triangle counterexample for directional Whitney}
	The inequality~\eqref{eqn:new whitney type} is, generally speaking, invalid for arbitrary convex body in $\R^2$ and $\EEE\subset\SS^1$ consisting of arbitrary two linearly independent vectors, at least for $p=\infty$. Indeed, let $G$ be the triangle with the vertices $(0,0)$, $(1,2)$ and $(2,1)$, and $\EEE=\{(1,0),(0,1)\}$. It is straightforward to verify that for $f_n(x,y):=\max\{n,\ln(x+y)\}$ one has $E_{r-1}(f_n;\EEE)_{L^\infty(G)}\to\infty$ while $\omega^r(f_n,G;\EEE)_{p}<C$ as $n\to\infty$.
\end{rem}

The next result shows that for arbitrary convex body $G$ in $\R^2$ it is always possible to choose a two-element set of directions $\EEE\subset\SS^1$ depending on $G$ so that~\eqref{eqn:new whitney type} is satisfied.

\begin{thm}\label{cor-7-7}
	If $G$ is a convex body in $\R^2$, then there exists two linearly independent vectors $\xi_1, \xi_2\in\SS^1$  such that for all $0<p\leq \infty$ and  $f\in L^p(G)$, 
	\begin{equation}\label{key-13-11}
	E_{r-1}(f;\{\xi_1, \xi_2\})_{L^p(G)}\leq C(p,r) \og^r (f, G; \{\xi_1, \xi_2\})_p.
	\end{equation}
\end{thm}

\begin{proof} 
	A geometric result of Besicovitch~\cite{Be} asserts that it is possible to inscribe an affine image of a regular hexagon into any planar convex body. Since~\eqref{key-13-11} does not change under affine transform, we may assume the hexagon with the vertices $(\pm1,\pm1)$ and $(0,\pm2)$ (which are the vertices of a regular hexagon after a proper dilation along one of the coordinate axes) is inscribed into $G$, i.e. each vertex of the hexagon is on the boundary of $G$. Convexity of $G$ then implies that $G$ contains the hexagon and is contained in a non-convex $12$-gon (resembling the outline of the star of David) obtained by extending the sides of the hexagon, i.e. having additional nodes $(\pm2,0)$ and $(\pm1,\pm3)$. Choosing $\xi_1=(1,0)$ and $\xi_2=(0,1)$, the proof is now completed by two applications of Theorem~\ref{cor-7-3} for $S=[-1,1]^2$: first in the direction of $\xi_1$ with $L=2$, and second (see Remark~\ref{rem-7-4}) in the direction of $\xi_2$ with $L=3$.
\end{proof}

It would be interesting to see a generalization of Theorem~\ref{cor-7-7} to higher dimensional convex bodies without any additional smoothness requirements with directions allowed to depend on the domain. Our last result in this section shows that if the boundary of a convex body in $\R^{d+1}$ satisfies certain cone-type condition with sufficiently large angle, then one can take $\EEE$ to be arbitrary $d+1$ linearly independent directions (independent of $G$). 

\begin{defn}
	A cone with vertex $v\in\RR^{d+1}$ in the direction of $u\in\SS^d$ of height $h>0$ and angle $\arccos\lambda$, where $\lambda\in(0,1)$, is the set 
	\[
	K(v,u,h,\alpha):=v+\{x\in\RR^{d+1}:\lambda \|x\| \le x\cdot u\le h \}.
	\]
	We say that a domain $G\subset \RR^{d+1}$ satisfies cone condition with parameters $h>0$ and $\lambda\in(0,1)$ if for any $v\in G$ there exists $u\in\SS^d$ such that $K(v,u,h,\lambda)\subset G$. 
\end{defn}




\begin{thm}\label{cor-7-8-18}
Let $\mathcal{E}\subset \SS^d$ be a set of $d+1$ directions such that $\spn(\mathcal{E})=\RR^{d+1}$, define 	
	\begin{equation}\label{7-8-0-18}
	\va_0:=\min_{\xi\in\SS^d} \max_{\eta\in \mathcal{E}} |\xi\cdot \eta|.	
	\end{equation} 
	Let $G$ 
be a convex body in $\RR^{d+1}$ satisfying the cone condition with parameters $h>0$ and $\lambda<\frac{\va_0}{4\sqrt{r}}$. Then for any $0<p\leq \infty$, $r\in\NN$ and $f\in L^p(G)$,
	\begin{equation}\label{7-8-18}E_{r-1} (f;\EEE)_{L^p(G)} \leq C  \og^r(f, G; \mathcal{E})_p,
	\end{equation}
	where the constant $C$ depends only on $p, d, r, h, \diam(G)$ and the parameter $\va_0$ 
	if it is close to zero. 
\end{thm}

\begin{rem}
	In particular, any convex $C^2$ domain satisfies this cone-type condition, but we remark that it is not possible to drop the $C^2$ condition due to the example from Remark~\ref{rem:triangle counterexample for directional Whitney}.
\end{rem}

\begin{proof}
	Set   $\va:= \f {\va_0^2}{8r}$, $\delta_0:=\frac h3$ and $\alpha:=\arccos \lambda$. By $0<\lambda<\frac{\eps_0}{4\sqrt{r}}$, it is easy to verify that  $(\frac1{\sin\alpha}+1)\delta_0\le h$ and $\frac1{\sin\alpha}-1\le\frac\va2$. Thus, for arbitrary $v\in G$, if $u\in\SS^d$ is such that $K(v,u,h,\alpha)\subset G$, then the closed ball $B$ of radius $\delta_0$ centered at $v+(\frac1{\sin\alpha}-1)\delta_0 u$ satisfies $B\subset K(v,u,h,\alpha)\subset G$ and $\dist(v,B)\le\frac{\va\delta_0}{2}$. In other words, for any $v\in G$ there exists a closed ball $B\subset G$ of radius $\delta_0$ and a point $\xi\in\partial B$ such that $\|v-\xi\|\le\frac{\va\delta_0}{2}$. We will use this fact in a moment. 
	
	Next we remark that for any radius $\delta>0$ and any subset $\Gamma\subset G$, it is possible to find points $x_1,\dots,x_m\in\Gamma$ such that the union of the open balls of radius $\delta$ centered at these points cover $\Gamma$ while $m\le c(d,\delta,\diam(G))$. This can be done by a standard packing/covering argument. Namely, we begin with arbitrary $x_1\in \Gamma$ and if $x_1,\dots,x_k\in\Gamma$ are already chosen but do not provide the desired cover, we pick a point $x_{k+1}\in\Gamma$ with $\|x_j-x_{k+1}\|\ge\delta$ for all $j<k$. Then the balls of radius $\delta/2$ centered at $x_j$ have non-overlapping interior and the bound on $m$ follows by volume estimates. 
	
	We prove that the domain $G$ can be decomposed as follows:
	\begin{equation}\label{decom2}
	G= \bigcup _{j=1}^{m} G_j\cup K,
	\end{equation}
	where    $m=m(d,\da_0, \va, \diam(G))$, each $G_j$ is a convex set containing an open ball $B_j$ of radius $\da_0$ such that 
	$B_j \subset G_j\subset (1+\va) B_j$ and   $\f 12 B_j  \subset K$,  and the set $K$ is a union of $M=M(d, \da_0)$ open balls $B^\ast_j$, $1\leq j\leq M$ of radius $\da_0/2$ such that $B^\ast_j\cap B^\ast_{j+1}$ contains an open ball of radius $\da_0/4$ for each $1\leq j<M$. 
	To this end, we first cover the set   $\Ga_0:=\{\xi\in G:\  \  \dist(\xi, \p G) \leq \frac{15}{16} \da_0\}$
	with $m=m( d, \da_0, \va, \diam(G))$ open balls $B_{v_j} (\frac{\va \da_0}2)$, $j=1,\dots, m$, of radius $\frac{\va \da_0}2$ such that  $v_1,\dots, v_{m}\in \Ga_0$. Invoking the fact from the beginning of the proof, for each $j$ we find an open ball $B_j$  of radius $\da_0$ and a point $\xi_j\in G$ such that $\xi_j\in \overline{B_j}\subset G$ and $\|v_j-\xi_j\|\le\frac{\va\delta_0}{2}$, implying that  $\dist(\xi_j,\partial G)\le \delta_0$ (as $\eps\le\frac1{8}$) and $\Ga_0\subset\bigcup_{j=1}^mB_{\xi_j}$. Define 
	$G_j :=( (1+\va) B_j)\cap G$. It is then easily seen that $\Ga_0 \subset \bigcup_{j=1}^{m} G_j\subset G$ and 
	$\f 12 B_j \subset \Ga_1:= \{\xi\in G:\  \  \dist(\xi, \p G) \ge  \da_0/2\}$ for each $j$.
	Next,  we cover   the set $\Ga_1$ with $m_1=m_1(d, \da_0, \diam(G))$ open balls $B_{\zeta_j}(\da_0/4)$, $j=1,\dots, m_1$ of radius $\da_0/4$ such that 
	$\{\zeta_j\}_{j=1}^{m_1} \subset  \Ga_1$. 
	Note that  $B_{\zeta_j}(\da_0/2) \subset G$ for each  $1\leq j\leq m_1$. 
	Since $\Ga_1$ is convex and hence connected,  following the proof of Lemma~\ref{LEM-4-2-18-0}, we  can form a sequence of sets $\{E_j\}_{j=1}^M$ from possibly repeated copies of  the sets
	$B_{\zeta_j}(\da_0/4)\cap \Ga_1$, $1\leq j\leq m_1$  such that  	$M\leq 2m_1^2$, 
	$\Ga_1=\bigcup_{j=1}^M E_j$, and  $E_j\cap E_{j+1} \neq \emptyset $ for  $1\leq j<M$.
	We then define  
	$ B^\ast_j=B_{\zeta_i} (\da_0/2)\subset G$   for $1\leq j\leq M$, where $i=i_j\in[1, m_1]$ is the index  such that  $E_j=B_{\zeta_i} (\da_0/4)\cap \Ga_1$. Let
	$K:=\bigcup_{j=1}^M B^\ast_j$. Clearly, $B^\ast_j\cap B^\ast_{j+1}$ contains an open ball of radius $\da_0/4$ for each  $1\leq j<M$,  and moreover  $\Ga_1\subset K$. Thus,  
	$G=\bigcup_{j=1}^m (G_j\cup K)$ is the decomposition of $G$ with the  above mentioned properties. 
	
	In addition to the decomposition~\eqref{decom2}, we also need  the following    lemma.

	\begin{lem}\label{lem-3-2-0} 	 Let   $\mathcal{E}\subset \SS^d$ and $\va_0\in (0,1)$  be the same as above (as in Theorem~\ref{cor-7-8-18}). 				
		Let  $K_0\subset \R^{d+1}$ be  a set containing   an open ball $B_0$ of radius $\da>0$. Assume that   the following Whitney   inequality  holds  for a given function $f\in L^p(\RR^{d+1})$ with $p>0$: 
		\begin{equation}\label{2-1}
		E_{r-1}(f;\mathcal{E})_{L^p(K_0)} \leq C_0  \og^r(f, K_0; \mathcal{E})_p.
		\end{equation}  
		Let  $S\subset \R^{d+1}$ be   a convex subset,  and assume that  there exist a parameter  $L\ge 1$ and   an open ball $B$ of radius $L\da$ such that 			
		$B_0\subset B\subset S\subset \sa B$, where $\sa:=1+\f {\va_0^2}{4r}$.   Then 
		\begin{equation}\label{2-4-0}
		E_{r-1} (f;\mathcal{E})_{L^p (K_0\cup S)}\leq C(L, d,p, r, \va_0) C_0 \og^r (f, K_0\cup S; \mathcal{E})_p.
		\end{equation} 
	\end{lem}

	We  postpone the proof of Lemma~\ref{lem-3-2-0}  until the end of this section. For the moment, we take it for granted and show how to deduce Theorem~\ref{cor-7-8-18} from the decomposition~\eqref{decom2}. 
	We start with the special case when $G$ itself is a ball $B\subset \R^{d+1}$.  	
	Since  dilations   and translations  do not change the directions in the set $\mathcal{E}$, 
	we  may also assume in this case that $G=B=B_1(0)$. Let $A$ be the $(d+1)\times(d+1)$ matrix whose $j$-th column vector is $\xi_j$. Then by~\eqref{7-8-0-18}
	\[
	\|A^t\mu\|\ge \|A^t\mu\|_{\infty}\ge \eps_0\|\mu\|, \quad \forall \mu\in\R^{d+1},
	\] 
	so from the fact that a matrix and its transpose have the same singular values, we get 
	\begin{equation}\label{norm-equiv}
	\va_0\|x\|\leq \|\sum_{j=1}^{d+1} x_j \xi_j\|\leq \sqrt{d+1} \|x\|,\   \  \forall x=(x_1,\dots, x_{d+1})\in\R^{d+1}.
	\end{equation}
	In particular, this implies that the parallelepiped
	$$H:=\Bl\{ \sum_{j=1}^{d+1} x_j \xi_j:\   \  x=(x_1,\dots, x_{d+1})\in [-\f 1 {d+1},\f 1{d+1}]^{d+1}\Br\}$$
	whose edge directions lie in the set $\mathcal{E}$,  satisfies  
	$ B_{\va_0/(d+1)}[0]\subset H\subset B_1[0].$
	Thus, applying Lemma~\ref{lem-3-2-0} to the pair of sets  $K_0=H$ and  $S=B=B_1(0)$,  and taking  into account Lemma~\ref{lem-5-3-0}, we prove  the Whitney inequality~\eqref{7-8-18} in the case when $G$ is an open ball. 
	
	Next, we apply  Lemma~\ref{lem-3-2-0} iteratively to the pairs of sets  $K_0=\bigcup_{i=1}^j B^\ast_i$ and  $S=B^\ast_{j+1}$ for $j=1,\dots, M-1$,  and using the inequality~\eqref{7-8-18} for the already proven case of $G=B^\ast_1$, we deduce the inequality~\eqref{7-8-18} with $K$ in place of $G$. 
	
	Finally, 
	using the decomposition~\eqref{decom2}, and  
	applying   Lemma~\ref{lem-3-2-0} iteratively to the pairs of sets  $K_0=K$ and  $S=G_j$ for $j=1,\dots, m$, we establish  the inequality~\eqref{7-8-18} on a general convex $C^2$ domain  $G$.
	This completes the proof of Theorem~\ref{cor-7-8-18}. 
\end{proof}

It remains to prove Lemma~\ref{lem-3-2-0}.

\begin{proof}[Proof of Lemma~\ref{lem-3-2-0}]
	Without loss of generality, we may assume   that $B_0=B_\da (0)$, and $\mathcal{E}=\{ \xi_j\}_{j=1}^{2(d+1)}$, where $\xi_{d+1+j}=-\xi_j$ for $1\leq j\leq d+1$. 
	We start with the special case when  $B=B_0$.
	In this special  case, we use~\eqref{7-8-0-18} and the assumption $\mathcal{E}=-\mathcal{E}$  to  decompose the ball  $\sa B$ as 
	$\sa B  =\bigcup_{j=1}^{2(d+1)} E_j$, where   
	\begin{equation*}\label{key-13-5}
	E_j:=\Bl\{\eta\in \RR^{d+1}:    \|\eta\|\leq \sa \da,  \   \     \  \eta \cdot \xi_j\ge \va_0\|\eta\|\Br\}. 
	\end{equation*}
	Setting    $h_j:= \f {\va_0 \da}{2r}  \xi_j$,  we  claim that 
	\begin{equation}\label{3-3-0}
	[\eta-rh_j, \eta-h_j]\subset B_0,\   \    \forall \eta\in E_j. 
	\end{equation}
	Indeed,  if  $\xi\in 	[\eta-rh_j, \eta-h_j]$, then we may write it    in the form
	$\xi= \da ( \zeta  -u \xi_j)$ with $\|\zeta\|\leq \sa$,\  \ $\zeta\cdot \xi_j\ge \va_0\|\zeta\|$ and  $ \f {\va_0} {2r}\leq u\leq \f {\va_0} 2$, however,  
	\begin{align*}
	\|\zeta-u \xi_j\|^2& \leq \sa^2 +u^2 -2\sa u\va_0
	\leq (\sa- u\va_0)^2 +u^2 \leq 1-\f {u\va_0}2 +u^2 \leq 1.
	\end{align*}
	This shows the claim~\eqref{3-3-0}.

	Next,   we decompose $S$ as $S=S\cap (\sa B) =\bigcup_{j=1}^{2(d+1)} S_j$ with $S_j :=S\cap E_j$. Define $K_j =K_{j-1} \cup S_j$ for $j=1,\dots, 2(d+1)$. Then  $B_0\subset K_0\subset K_j$ for all $1\leq j\leq 2(d+1)$ and $K_{2(d+1)}=K_0\cup S$. Thus, applying~\eqref{5-2-eq-0-18} from Remark~\ref{rem-7-2}  iteratively to the pairs of the  sets $K=K_{j-1}$ and $J=S_j\subset E_j$ along  the directions $h=h_j$  for $j=1,\dots, 2(d+1)$,  and  taking into account~\eqref{3-3-0},  we obtain that  for any $q\in\Pi_{r-1}^{d+1}$, 
	\begin{align}\label{7-16-0}
	\|f-q\|_{L^p(K_j)}\leq C(p,r) \|\tr_{h_j}^r f\|_{L^p({ S_j-rh_j})}+ C(p,r) \|f-q\|_{L^p(K_{j-1})}.
	\end{align}	
	Note that   for each $\xi\in  S_j-rh_j$, we have that $\xi+rh_j\in  S_j\subset S$ and by~\eqref{3-3-0},  $[\xi, \xi+(r-1)h_j]\subset B\subset S$. Since $S$ is a  convex set, this implies that 
	$[\xi, \xi+rh_j]\subset S$  for each $\xi\in  S_j-rh_j$. Thus,   $\tr_{h_j}^r f (\xi)=\tr_{h_j}^r (f, S, \xi)$ for each $\xi\in  S_j-rh_j$. 		
	It then  follows from~\eqref{2-1}  and~\eqref{7-16-0}  that 
	\begin{align*}\label{13-7}
	E_{r-1}(f)_{L^p(K_0\cup S)}&\leq   C (p, d,r)C_0 \og^r(f, K_0\cup S,\mathcal{E})_p.
	\end{align*}	
	This gives the desired result in the special case of $B=B_0$.   
	
	Let us also point out that an iterative application of the  already proven result for   the above special case in fact implies some  more general results.  Indeed,  if  $B_0$ is  any open ball in $K_0$ and  $\ell\in\NN$, then 
	applying the already proven  result  iteratively to the pairs of sets $K_0\cup \sa^{j-1} B_0$ and  $S=\sa^j B_0$ for $j=1,\dots, \ell$, we obtain that 
	\begin{equation}\label{7-16-18}
	E_{r-1} (f) _{L^p (K_0\cup \sa^\ell B_0)} \leq C(p,r,d, \ell)C_0 \og^r(f, K_0\cup \sa^\ell B_0; \mathcal{E})_p.
	\end{equation} 
	More generally, if $a>\sa$, then  		
	using~\eqref{7-16-18} with $\ell\in \NN$ being such that $\sa^\ell \leq a <\sa^{\ell+1}$, and 
	applying  the result of the special case  to the pair of sets $K_0\cup \sa^\ell B_0$ and $S=a B_0$, we obtain that 
	\begin{equation}\label{7-17-18}
	E_{r-1} (f) _{L^p (K_0\cup a B_0)} \leq C(p,r,d, a) C_0\og^r(f, K_0\cup a B_0; \mathcal{E})_p.
	\end{equation}

	Next, we consider  a more  general case when  the center $\xi_0$  of the ball $B=B_{L\da} (\xi_0)$ satisfies that $\|\xi_0\|\leq \f \da 2$. In this case,  $ B_{\da/2} (\xi_0) \subset B_0\subset K_0$.  Thus, using~\eqref{7-17-18} with $B_{\da/2} (\xi_0)$ in place of the ball $B_0$, and with the parameter $a=2L$, we deduce 
	$$ E_{r-1} (f) _{L^p (K_0\cup B)} \leq C(p,r,d, L) \og^r(f, K_0\cup B; \mathcal{E})_p.$$
	The   Whitney  inequality on the set $K_0\cup S$ then  follows by applying  the already proven special case   to the initial set  $K_0\cup B$ (instead of $K_0$).

	Finally, we treat the remaining case; that is,  the case when   $\f \da 2 <\|\xi_0\|\leq L\da$.   
	Let    $\eta_0:=\f {\xi_0}{ \|\xi_0\|}$ be the direction of the center of the ball $B$.   Let  $n_0\leq \f {L\sa}{\sa-1}$ be a positive integer satisfying 
	$$ n_0 \f{(\sa-1)\da}\sa < \|\xi_0\|\leq (n_0+1) \f {(\sa-1) \da}\sa.$$	
	Let $\eta_j =j\f{(\sa-1)\da}\sa \eta_0$,  $j=0,1,\dots, n_0$ be the equally spaced points on the line segment $[0, \xi_0]$.  
	Define  $B_j =B_{\da/\sa}(\eta_j)$ for $j=1,\dots, n_0$ and 	recall that $B_0=B_\da (0)$.  It is easily seen that    $B_j\subset \sa B_{j-1} \subset B$ for $j=1,\dots, n_0$.
	Let   $K_j:=K_{j-1}\cup \sa B_j$ for $j=1,\dots, n_0$.  Then   $B_j\subset \sa B_{j-1}\subset K_{j-1}$ for $1\leq j\leq n_0$.  Thus, applying the result for  the already proven special case iteratively  to  the pairs of sets $(K_0, B_0):=(K_{j-1}, B_j)$,  we conclude that 
	\begin{equation}\label{7-18-18-0}
	E_{r-1} (f) _{L^p (K_{j})} \leq C(p,r,d, L) C_0\og^r(f, K_{j}; \mathcal{E})_p,\  \  j=1,\dots, n_0.
	\end{equation}
	Since $ \|\xi_0-\eta_{n_0} \|\leq (\sa-1)\da/\sa$, we have that $$B_{\da/\sa }(\xi_0) \subset B_{\da} (\eta_{n_0})= \sa B_{n_0} \subset K_{n_0}.$$
	Recall also that $K_0\subset K_{n_0}\subset K_0\cup B$. 
	Thus, using~\eqref{7-18-18-0} with $j=n_0$, and applying~\eqref{7-17-18} to  the pair of sets $(K_0, B_0):=(K_{n_0}, B_{\da/\sa} (\xi_0))$ and the parameter $a:=L\sa$, we obtain  
	\begin{equation*}\label{7-18-18}
	E_{r-1} (f) _{L^p (K_0\cup B)} \leq C(p,r,d, L) C_0\og^r(f, K_{0}\cup B; \mathcal{E})_p.
	\end{equation*}
	The desired result  then follows by applying the already proven  result for  special  case once again. This completes the proof. 
\end{proof}

\chapter{Polynomial partitions of the unity }

	\section{Polynomial partitions of the unity on domains of special type}
\label{sec:5}
The main purpose in this section is to prove a localized  polynomial partition of the  unity on a domain   $G\subset \RR^{d+1}$  of special type.  Without loss of generality, we may assume that $G$ is an upward $x_{d+1}$-domain given in~\eqref{2-7-special} with $\xi=0$, small base size $b>0$ and parameter $L=b^{-1}$. Namely,  	
\begin{align*}
G:=\{ (x, y):\  \  x\in [-b,b]^{d},\   \  g(x)-1\leq  y\leq  g(x)\},
\end{align*}
where $b\in (0,(2d)^{-1})$ is a sufficiently small constant and  $g$ is a $C^2$-function on $\RR^d$ satisfying that $\min_{x\in [-b, b]^d} g(x)\ge 4$.

Our construction of localized polynomial partition of the  unity relies on  a partition of the domain $G$, which we now  describe.  
Given a positive integer $n$, let   $\Ld^d_n:=\{ 0, 1,\dots, n-1\}^d\subset \ZZ^d$ be an index set. 
We shall use boldface letters $\mathbf{i}, \mathbf{j},\dots$ to denote indices in the set $\Ld_n^d$.  For each $\mathbf{i}=(i_1,\dots, i_{d})\in \Ld_n^d$, define 
\begin{equation}\label{partition-b}\Delta_{\bfi}:=[t_{i_1}, t_{i_1+1}]\times \dots \times [t_{i_{d}}, t_{i_{d}+1}]  \   \  \ \text{with}\  \    t_{i}=-b+\f {2i}n b.
\end{equation}
Then  $\{\Delta_{\bfi}\}_{\bfi\in\Ld_n^d}$ forms a   partition of the cube $[-b,b]^d$.  
Next, let    $N:=N_n:=2\ell_1 n$
and  $\al:= 1/(2\sin^2\f \pi{2\ell_1})$,  where  $\ell_1$ is  a fixed  large positive integer satisfying \begin{equation}\label{5-2-18}
\al\ge 5d\max_{x\in [-4b, 4b]^d} (|g(x)|+\max_{1\leq i, j\leq d} |\p_i\p_j g(x)|). 
\end{equation}
Let  $\{\al_j:=2\al \sin^2 (\f {j\pi}{2N})\}_{j=0}^N$ denote   the Chebyshev partition of the interval $[0, 2\al]$ of order $N$ such that  $\al_n=1$. Then  $\{\al_j\}_{j=0}^n$ forms a partition of the interval $[0,1]$.  
Finally, we  define a  partition of the domain $G$ as follows:  
\begin{align*}
G&=\Bl\{(x,y):\  \  x\in [-b,b]^d,\   \   g(x)-y\in [0,1]\Br\} =\bigcup_{\bfi\in\Ld_n^d} \bigcup_{j=0}^{n-1} I_{\mathbf{i},j},
\end{align*}
where 
$$I_{\mathbf{i},j}:=\Bl\{ (x, y):\  \  x\in \Delta_{\bfi},\  \   g(x)-y\in [\al_{j}, \al_{j+1}]\Br\}.$$
Note that $\Ld_n^d \times \{0,\dots, n-1\}=\Ld_n^{d+1}$.

With the above notation, we have

\begin{thm}\label{strips-0}
	For any   $m\ge2$, there exists a sequence of   polynomials $\bl\{q_{\bfi, j}:\  \   (\bfi, j) \in\Ld_n^{d+1}\br\}$ of degree at most  $ C(m, d) n$ on $\RR^{d+1}$ such that $$\sum_{(\bfi, j)\in\Ld_n^{d+1}} q_{\bfi,j}(x,y)=1\   \   \  \text{for all $(x,y)\in G$},$$
	and  for each $(x,y)\in I_{\mathbf{k}, l} $ with $(\mathbf{k},l)\in\Ld_n^{d+1}$, 
	\be\label{strips-ineq}
	| q_{\bfi,j}(x,y)|\le \frac{C_{m,d}}{\Bl(1+\max\{ \|\bfi-\mathbf{k}\|, |j-l|\}\Br)^m}.
	\ee
\end{thm}

Theorem~\ref{strips-0}  is motivated by~\cite[Lemma~2.4]{Dz-Ko}, but  some important details of the proof were omitted there. In this chapter, we shall give a complete and simpler proof of the theorem. 

In the sequel, we  often use Greek letters $\xi,\eta,\dots$ to denote points in $\RR^{d+1}$ and write $\xi\in\RR^{d+1}$ in the form $\xi=(\xi_x, \xi_y)$ with $\xi_x\in\RR^d$ and $\xi_y\in\RR$.

\begin{rem}\label{rem-5-2}
	For applications in later chapters, we  introduce    the following metric on the domain   $G$:  for  $\xi=(\xi_x, \xi_y)$ and  $\eta=(\eta_x, \eta_y)\in G$, 
	\begin{equation}\label{rhog}
	\wh{\rho}_G(\xi, \eta):=\max\Bl\{\|\xi_x-\eta_x\|,
	\Bl|\sqrt{g(\xi_x)-\xi_y}-\sqrt{g(\eta_x)-\eta_y}\Br|\Br\}.
	\end{equation}
	It can be easily seen that    if $ \xi\in I_{\bfi, j}$ and $ \eta\in I_{\mathbf {k}, \ell}$, then 
	\begin{equation}\label{Chapter-5-1}
	1+n\wh{\rho}_G(\xi,\eta) \sim 1+\max\{ \|\bfi-\mathbf{k}\|, |j-\ell|\}.
	\end{equation}
	This implies that  
	\begin{equation}\label{6-6-18}
	| q_{\bfi,j}(\xi)|\le \frac{C_{m,d}}{(1+n\wh{\rho}_G(\xi,\og_{\bfi,j}))^m},\   \   \ \forall  \xi\in G,\   \ \forall \og_{\ib, j} \in I_{\ib, j}.
	\end{equation}
	
\end{rem}
\begin{rem}\label{rem-6-3}
	If  $r\in\NN$ and $n\ge 10 r$, then the polynomials $q_{\ib,j}$ in Theorem~\ref{strips-0}  can be chosen to be of total degree $\le n/r$. Indeed, this can be obtained by invoking   Theorem~\ref{strips-0} with $n/r$ in place of $n$, relabeling the indices, and setting some of the polynomials to be zero. 
\end{rem}
For the proof of Theorem~\ref{strips-0}, we need two additional   lemmas, the first of which is well known.

\begin{lem}\label{chebyshev}\cite[Theorem~1.1]{Dz-Ko}
	Given any parameter $\ell>1$, there exists a sequence of     polynomials $\{u_j\}_{j=1}^n$ of degree at most  $ 2n$ on $\RR$  such that $\sum_{j=0}^{n-1}u_j(x)=1$ for all $x\in [-1,1]$ and
	$$| u_j(\cos\t)|\leq \frac{C_{\ell}}
	{( 1+n|\t -\f {j\pi}{n}|)^{\ell}},\   \  \t\in [0,\pi],\  \ j=0,\dots,n-1.$$
\end{lem}

The second lemma gives a    polynomial partition of the  unity   associated with the partition $\{\Delta_{\bfj}:\  \ \bfj\in\Ld_n^d\}$ of  the cube $[-b,b]^d$. 

\begin{lem}\label{uniform} Given any  parameter $\ell>1$,  there exists a sequence of    polynomials $\{v_{\mathbf{j}}^d\}_{\mathbf{j}\in\Ld_n^d} $ of total  degree $\le  2dn$ on $\RR^d$ such that for all $x\in [-b, b]^{d}$,  $\sum_{\mathbf{j}\in \Ld_n^d}v_{\mathbf{j}}^d(x)= 1$   and
	$$|v_{\mathbf{j}}^d(x)|\le \f{C_{\ell,d}} { (1+n\|x-x_{\mathbf{j}}\|)^\ell},\    \  \mathbf{j}\in\Ld_n^d,
	$$
	where $x_{\bf j}$ is an arbitrary point in $\Delta_{\bf j}$.
\end{lem}
This lemma is probably well known, but for completeness, we present a proof below.
\begin{proof}
	Without loss of generality, we may assume that  $d=1$ and $b=\f12$. 
	The  general case can be deduced easily  using tensor products of polynomials in one variable.
	Let $\{u_j\}_{j=0}^{ n-1}$ be a sequence of
	polynomials of degree at most $2n$ as given in  Lemma~\ref{chebyshev} with $2\ell$ in place of $\ell$. Noticing that    for  $u\in [-1,1]$ and $v\in [-\f 12, \f12]$, 
	\begin{equation}\label{4-1}
	|u-v|\leq |\arccos u-\arccos v |\leq \pi |u-v|,
	\end{equation}
	we obtain  
	\begin{equation}\label{4-2}
	| u_j(x)|\leq\f{ C_\ell}{ (1+n|x-\cos \f {j\pi}n|)^{2\ell}},\  \ x\in \Bigl[-\f12,\f12\Bigr].
	\end{equation}
	Next, we define a sequence of polynomials $\{v_j\}_{j=0}^{n-1}$ of degree at most $2n$ on $[-\f12, \f12]$ as follows: 
	$$ v_j(x)=\sum_{i:\  \   s_{j}<\cos \f {i\pi}n\leq s_{j+1}} u_i(x),$$
	where $0\leq i\leq n-1$,  $s_0=-2$, $s_n =2$,  $s_j=t_j =-\f12 +\f {j}n$ for $1\leq j\leq n-1$, and  we define  $v_j(x)=0$ if the sum is taken over the  emptyset.  Clearly,   $\sum_{j=0}^{n-1} v_j(x)=\sum_{i=0}^{n-1} u_i(x)=1$ for all $x\in [-\f12,\f12]$. 
	Furthermore, using~\eqref{4-2}, we have 
	\begin{align*}
	| v_j(x) |\leq \f {C_\ell}{ (1+n|x-s_j|)^{\ell}} \sum_{i=0}^{n-1} \f{ 1}{ (1+n|x-\cos \f {i\pi}n|)^{2\ell}}\leq  \f {C_\ell}{ (1+n|x-s_j|)^{\ell}}, 
	\end{align*}
	where the last step uses~\eqref{4-1}.
	This completes the proof.
\end{proof}

We are now in a position to prove Theorem~\ref{strips-0}.\\

\begin{proof}[Proof of Theorem~\ref{strips-0}]
	Set \[M:=d\max_{1\leq i,j\leq d}\max_{\|x\|\leq b} |\p_i\p_jg(x)|+1.\]
	For each  $\bfi\in\Ld_n^d$,  let  $x_{\bfi}\in \Delta_{\bfi}$ be   an arbitrarily fixed point in the cube $\Delta_{\bfi}$, and  define 
	$$f_{\bfi}( x):=g( x_{\bfi})+\nabla g( x_{\bfi}) \cdot ( x- x_{\bfi})+\f M2 \| x- x_{\bfi}\|^2.$$
	By   Taylor's theorem,  it is easily seen  that for each $x\in [-b,b]^d$,
	\begin{align}\label{5-8-aug}
	f_{\bfi}(x) -M\|x-x_{\bfi}\|^2 \leq g(x) \leq f_{\bfi}(x).
	\end{align}
	Since $0<b<(2d)^{-1}$, this implies that  for each $\bfi\in\Ld_n^d$,
	$$ G \subset\Bl \{ ( x,y):\  \ x\in [-b,b]^{d},\   \   0\leq f_{\bfi}( x)-y\leq M+1\Br\}.$$
	Recall that 
	$\{\al_j\}_{j=0}^N$
	is a Chebyshev partition of $[\al_0, \al_N]=[0, 2\al]$ of degree $ N=2\ell_1n$,  $\al_n=1$   and according to~\eqref{5-2-18}, $\al\ge 4M+1$. Thus, 
	\begin{align*}
	G\subset \bigcup_{\bfi\in\Ld_n^d} \bigcup_{j=0}^{N-1} \Bl \{ ( x,y):\  \ x\in \Delta_{\bfi},\   \   \al_{j}\leq f_{\bfi}( x)-y\leq \al_{j+1}\Br\}. 
	\end{align*}
	
	Next, using   Lemma~\ref{chebyshev}, we obtain   a sequence of  polynomials $\{u_j\}_{j=0}^{N-1}$of degree at most $4\ell_1 n$ on $[0, 2\al]$  such that   $\sum_{j=0}^{N-1} u_j(t) =1$ for all $t\in [0, 2\al]$,   and
	\begin{equation}\label{key-5-6}
	| u_j (t)| \leq \f {C_m}{ (1+n |\sqrt{t}-\sqrt{\al_j}|)^{4m}},\   \  t\in [0, 2M]\subset [0, \al].
	\end{equation}
	Similarly, using  Lemma~\ref{uniform},  we may  obtain  a sequence of   polynomials $\{v_{\mathbf{j}}\}_{\mathbf{i}\in\Ld_n^d} $ of total  degree $\le n$ on the cube $[-b, b]^d$  such that $\sum_{\mathbf{j}\in \Ld_n^d}v_{\mathbf{j}}( x)= 1$ for all $ x\in [-b, b]^{d}$,   and
	\begin{equation}\label{key-5-7}
	| v_{\mathbf{j}}( x)|\le \f{C_m} { (1+n\| x- x_{\mathbf{j}}\|)^{4m}},\   \ x\in [-b, b]^d.
	\end{equation}
	Define  a sequence $\{q^\ast_{\bfi,j}: \  \  \bfi\in\Ld_n^d,\   \  0\leq j\leq N-1\}$ of   auxiliary polynomials  as follows: 
	\begin{equation}\label{key-5-8-0}
	q^\ast_{\bfi,j}( x,y):=u_j(f_{\bfi}( x)-y)v_{\bfi}( x).
	\end{equation}
	It is easily seen from~\eqref{key-5-6} and~\eqref{key-5-7} that
	for each $( x, y)\in G$, 
	\begin{align}\label{5-7-chapter}
	|  q^\ast_{\bfi,j}(x, y)|\leq \f {C_m}{(1+n\|x- x_{\bfi}\|)^{4m} (1+n|\sqrt{f_{\bfi}({x})-y}-\sqrt{\al_j}|)^{4m}}.
	\end{align}
	
	We   claim that for each  $({x},y)\in G$,
	\begin{align}
	|  q^\ast_{\bfi,j}(x, y)|&\leq
	\f {C_m}{(1+n\|{x}- x_{\bfi}\|)^{2m} (1+n|\sqrt{g({x})-y}-\sqrt{\al_j}|)^{2m}}.\label{claim-4-5}
	\end{align}
	Note that~\eqref{claim-4-5}  follows directly from~\eqref{5-7-chapter} if $ 6M\| x- x_{\bfi}\|>  |\sqrt{g({x})-y}-\sqrt{\al_j}|$.
	Thus, for the proof of~\eqref{claim-4-5}, 
	 it suffices to prove that  the equivalence
	\begin{equation}\label{key-5-11}
	|\sqrt{f_{\bfi}({x})-y}-\sqrt{\al_j}|\sim |\sqrt{g({x})-y}-\sqrt{\al_j}|,
	\end{equation}
	holds 
	under the assumption  
	\begin{equation}\label{key-5-13-18} 6M\| x- x_{\bfi}\|\leq  |\sqrt{g({x})-y}-\sqrt{\al_j}|. \end{equation}
	Indeed, if  $\sqrt{f_{\bfi}({x})-y}+ \sqrt{g({x})-y}\leq 2M\|{x}-{x}_{\bfi}\|$, then~\eqref{key-5-13-18} implies 
	$$\sqrt{\al_j}\ge 4M \|{x}-{x}_{\bfi}\|\ge 2 \max\{\sqrt{f_{\bfi}({x})-y},  \sqrt{g({x})-y}\},$$ 
	and hence 
	$$|\sqrt{f_{\bfi}({x})-y}-\sqrt{\al_j}|\sim \sqrt{\al_j}\sim |\sqrt{g({x})-y}-\sqrt{\al_j}|.$$
	On the other hand, if  $\sqrt{f_{\bfi}({x})-y}+ \sqrt{g({x})-y}> 2M\|{x}-{x}_{\bfi}\|$, then  by~\eqref{key-5-13-18} and~\eqref{5-8-aug}, we have 
	\begin{align*}
	&\Bl|\sqrt{f_{\bfi}({x})-y}- \sqrt{g({x})-y}\Br|=\f{|f_{\bfi}({x})-g({x})|}{\sqrt{f_{\bfi}({x})-y}+ \sqrt{g({x})-y}}\\
	&\leq \f {M\|{x}-{x}_{\bfi}\|^2}{2M\|{x}-{x}_{\bfi}\|}= \f12 \|{x}-{x}_{\bfi}\|\leq \f 1{12M}|\sqrt{g({x})-y}-\sqrt{\al_j}|,
	\end{align*}
	which in turn  implies~\eqref{key-5-11}. This completes the proof of~\eqref{claim-4-5}.

	Finally,   we define  for  $\bfi\in\Ld_n^d$,
	$$q_{\bfi,j}(x, y)=\begin{cases}
	q^\ast_{\bfi,j}(x,y),\  \  \text{ if $0\leq j\leq n-2$},\\
	\sum_{k=n-1}^{N-1} q^\ast_{\bfi,k}(x,y),\   \  \text{if $j=n-1$.}
	\end{cases}$$
	Clearly, each $q_{\bfi, j}$ is a  polynomial of degree at most $Cn$. 
	Since  for any $( x, y)\in G$ the polynomial $u_j$ in the definition~\eqref{key-5-8-0} is evaluated at the  point $f_{\bfi}( x)-y$, which lies in the interval $[0, M+1]\subset [\al_0, \al_N]$, it follows  that  for any $( x, y)\in{G}$,
	$$\sum_{\bfi\in\Ld_n^d} \sum_{j=0}^{n-1} q_{\bfi,j}(x,y)=\sum_{\bfi\in\Ld_n^d} \sum_{j=0}^{N-1} q^\ast_{\bfi,j}(x)=\sum_{\bfi\in\Ld_n^d} v_{\bfi}^d ( x) \sum_{j=0}^{n-1} u_{j} (f_{\bfi}( x) -y)=1.$$
	To complete the proof, 
	by~\eqref{claim-4-5}, it remains to estimate $q_{\bfi, j}$ for $j=n-1$.   
	Note  that for $j\ge n$,
	$$\sqrt{\al_j}-\sqrt{g({x})-y}\ge \sqrt{\al_n}-\sqrt{g({x})-y}\ge 0.$$
	Thus,  using~\eqref{claim-4-5}, and recalling that $m\ge 2$,  we obtain  that  
	\begin{align*}
	|q_{\bfi,n-1}(x)|&\leq  \f {C_m}{(1+n\|{x}- x_{\bfi}\|)^{2m} (1+n|\sqrt{g({x})-y}-\sqrt{\al_n}|)^{m}} \\
	&\qquad\qquad\qquad \cdot
	\sum_{j=n}^{N}\f 1{(1+n|\sqrt{g({x})-y}-
		\sqrt{\al_j}|)^{m}}\\
	&\leq \f {C_m}{(1+n\|{x}- x_{\bfi}\|)^{2m} (1+n|\sqrt{g({x})-y}-\sqrt{\al_n}|)^{m}}.
	\end{align*}
	This completes the proof.
\end{proof}

\section{Polynomial partitions of the unity on general $C^2$-domains}
\label{unity:sec}
In this section, we shall extend Theorem~\ref{strips-0} to the $C^2$-domain $\Og$. We will use the metric $\rho_\Og$ defined by~\eqref{metric}. 
Our  goal is to show the following theorem: 
\begin{thm}\label{polyPartition}Given  any $m>1$ and any positive integer $n$,   there exist a finite subset $\Ld$ of $\Og$ and  a sequence $\{\vi_\og\}_{\og\in\Ld}$ of   polynomials of degree at most $C( m) n$ on the domain $\Og$  satisfying 
	\begin{enumerate}[\rm (i)]
		\item $ \rho_{\Og} (\og,\og') \ge \f 1n$ 	for any two distinct points $\og,\og'\in\Ld$;  
		\item for every $\xi\in \Og$,	$\sum_{\og \in\Ld}  \vi_\og (\xi)=1$  and 
		\item for any  $\xi\in \Og$ and $\og\in\Ld$, 
		$$  |\vi_\og(\xi)| \leq C_m (1+n\rho_\Og(\xi,\og))^{-m}.$$  
	\end{enumerate} 	
\end{thm}

Recall that   $\wh{\rho}_G$ denotes  the metric on a domain  $G$ of special type  as defined in~\eqref{rhog}. Of crucial importance in the proof of Theorem~\ref{polyPartition}   is  the following   lemma,  which states   that if $G\subset \Og$ is a domain of special type attached to $\Ga$, then restricted on $G$, the two  metrics  $\rho_\Og$ and   $\wh{\rho}_G$ are equivalent. 

\begin{lem}\label{metric-lem} If $G \subset \Og$ is a domain of special type attached to $\Ga$, then 
	\begin{equation}\label{6-1-metric}\wh{\rho}_G(\xi,\eta)\sim \rho_{\Og} (\xi,\eta),\    \    \  \xi, \eta\in G.\end{equation}
\end{lem}
\begin{proof}
	Without loss of generality, we may assume that $G$ takes the form: 
	\begin{align}\label{standard}
	G:=\{ (x, y):\  \  x\in (-b,b)^{d},\   \  g(x)-1<  y\leq  g(x)\},
	\end{align}
	where $b>0$ is a sufficiently small constant and  $g\in C^2(\RR^d)$ satisfies  $\min_{x\in [-4b, 4b]^d} g(x)\ge4$. Then 
	$$  \Ga':=\{ (x, g(x)):\   \  x\in [-2b, 2b]^{d}\}
	\subset \Ga,$$
	and 
	there exists an  open rectangular box $Q\subset \RR^{d+1}$ such that 
	$$G^\ast=\{ (x, y):\  \  x\in (-2b,2b)^{d},\   \  0< y\leq  g(x)\}= \Og\cap Q.$$
	This in particular  implies that there exists a small constant $\va\in (0, b)$ depending only on $\Og$ such that  for every  $\xi\in G$ with  $\dist(\xi, \Ga)<\va$,  $\dist(\xi,\Ga')=\dist(\xi, \Ga)$.

	Recall that    a point $\xi\in\RR^{d+1}$ is often written in the form $\xi=(\xi_x, \xi_y)$ with $\xi_x\in\RR^d$ and $\xi_y\in\RR$. 
	Our proof relies on the following lemma, which states that $\dist(\xi, \Ga')\sim g(\xi_x)-\xi_y$ for every $\xi=(\xi_x, \xi_y)\in G$. 
	
	\begin{lem}\label{lem-9-1} If $\xi=(\xi_x, \xi_y)\in  G$, then 
		$$c_\ast (g(\xi_x)-\xi_y)\leq \dist(\xi, \Ga')\leq g(\xi_x)-\xi_y,$$
		where 
		$ c_\ast =  \f 1 {8 \sqrt{ 1+\|\nabla g\|_\infty^2}},$
		and  the $L^\infty$-norm in $\|\nabla g\|_\infty$	 is taken over the cube $[-2b, 2b]^d$. 
	\end{lem}
	
	\begin{proof}
		Clearly, for each $\xi=(\xi_x, \xi_y)\in G$, 
		$$\dist(\xi, \Ga')\leq \|(\xi_x,\xi_y)-(\xi_x, g(\xi_x))\|=g(\xi_x)-\xi_y.$$
		To show the inverse inequality,  we note that  for  an arbitrary point   $(x, g(x))\in\Ga'$, 
		\begin{align}
		\|\xi-(x,g(x))\|^2&=\|\xi_x-x\|^2 +|\xi_y -g( x)|^2\label{6-1-18}\\
		&=\|\xi_x- x\|^2 +|\xi_y -g(\xi_x)|^2+
		|g(\xi_x)-g(x)|^2\notag\\
		&\   \   \   \  +2 (\xi_y-g(\xi_x))\cdot (g(\xi_x)-g(x)).\label{eq-9-2}
		\end{align}
		If  $\|x-\xi_x\|\ge   c_\ast (g(\xi_x)-\xi_y)$, then by~\eqref{6-1-18},
		$$ \|\xi-(x,g(x))\|\ge \| x-\xi_x\|\ge c_\ast (g(\xi_x)-\xi_y).$$
		If $\|x-\xi_x\|\leq  c_\ast (g(\xi_x)-\xi_y)$, then  
		\begin{align*}
		& \|\xi_x- x\|^2+
		|g(\xi_x)-g( x)|^2+2 (\xi_y-g(\xi_x))\cdot (g(\xi_x)-g( x))\\
		&\leq (1+\|\nabla g\|_\infty^2) \|\xi_x-x\|^2 +2\|\nabla g\|_\infty (g(\xi_x)-\xi_y)\| \xi_x- x\|\\
		&\leq \Bl[c_\ast^2 (1+\|\nabla g\|_\infty^2)+ 2 \|\nabla g\|_\infty c_\ast\Br] (g(\xi_x)-\xi_y)^2\\
		&\leq \f 12 (g(\xi_x)-\xi_y)^2,
		\end{align*}
		which, using~\eqref{eq-9-2},   implies that
		$$
		\|\xi-(x,g(x))\|^2
		\ge \f12 |\xi_y -g(\xi_x)|^2.$$
		Now taking the infimum over $(x, g(x))\in\Ga'$, we obtain the  inverse inequality:
		$$\dist (\xi,\Ga')\ge c_\ast (g({\xi}_x)-\xi_y).$$
	\end{proof}

	Now we continue with  the proof of Lemma~\ref{metric-lem}. Let   $\xi=(\xi_x, \xi_y), \eta=(\eta_x, \eta_y)\in G$.
	According to  Lemma~\ref{lem-9-1},  if  $\dist(\xi, \Ga)\ge \va$ or $\|\xi-\eta\|\ge \va$, then 
	$$ \rho_{\Og} (\xi, \eta) \sim \wh{\rho}_G (\xi, \eta)\sim \|\xi-\eta\|.$$
	Thus, without loss of generality we may assume that $\dist(\xi, \Ga), \dist(\eta,\Ga) \leq \va$ and $\|\xi-\eta\|\leq \va$. 
	We may also  assume that  $b\leq 2\va$ since otherwise  we may replace the function $g$ with   $g(\cdot +\xi_x)$ and consider the corresponding  domain of special type with base size   $2\va$.

	Since $\dist(\xi, \Ga) =\dist(\xi, \Ga')$, there exists $t_\xi\in [-2b, 2b]^d$ such that 
	$$s_\xi:= \dist(\xi, \Ga)=\|\xi-(t_\xi, g(t_\xi))\|.$$
	Since $\Ga$ is a $C^2$-surface, by a straightforward calculation we have that  
	$$ \xi=(\xi_x, \xi_y) = \Bl( x(t_\xi, s_\xi),\   \    g(t_\xi) -s_\xi A(t_\xi)\Br),$$
	where 
	\begin{equation}\label{def-9-3}x(t,s):=t+s A(t)\nabla  g(t), \    \    \    A(t)=(\sqrt{1+\|\nabla g(t)\|^2})^{-1},  \end{equation}
	for  $t\in [-2b,2b]^d$ and $s\in [0, \va]$. It then  follows that 
	\begin{equation}\label{1-1}
	g(\xi_x)-\xi_y =F(t_\xi, s_\xi),
	\end{equation}
	where  $F$ is a function   on  the rectangular box   $ [-2b, 2b]^{d}\times [0, 2\va]$ given by 
	\begin{equation}\label{def-9-3b}  F(t,s):=g\bl(x(t,s)\br)-g(t)+s A(t).\    \    \    \end{equation}
	A straightforward calculation shows that for $t\in [-2b, 2b]^d$ and $s\in [0, 2\va]$ with $0<b\leq 2\va$, 
	\begin{align}
	\f {\p F}{\p s}(t,s)&=\Bl[ 1+(\nabla g)(x(t,s))\cdot \nabla g(t)\Br]A(t)=\sqrt{ 1+\|\nabla g(0)\|^2} +O(\va),\label{9-3-0}\end{align}
	and \begin{align}
	\|\nabla_t F(t,s)\|
	&=\Bl\| \nabla g(x(t,s))-\nabla g(t)+s \nabla A(t)\notag\\
	&	+s\nabla g(x(t,s))\Bl[ (\nabla g(t))^{tr}\nabla A(t) + A(t) H_g(t)\Br]\Br\| =O(\va),\label{9-4-0}
	\end{align}
	where $\nabla_t $ denotes  the gradient operator $\nabla$  acting on the variable $t$, $\nabla g$ is treated   as a row vector, and 
	$H_g:=(\p_i\p_j g)_{1\leq i, j\leq d}$ denotes the Hessian matrix of $g$.

	Clearly, we may define $t_\eta$ and $s_\eta$ in a similar way  for the vector $\eta$. Then 	
	using Lemma~\ref{lem-9-1} and~\eqref{1-1}, we obtain  that
	\begin{align*}
	\Bl|\sqrt{g(\xi_x)-\xi_y}&-\sqrt{g(\eta_x)-\eta_y}\Br|
	\sim \f { \Bl|F(t_\xi, s_\xi)-F(t_\eta, s_\eta)\Br|}{\sqrt{s_\xi}+\sqrt{s_\eta}},\end{align*}
	which,  by the mean value theorem,~\eqref{9-3-0} and~\eqref{9-4-0}, can be written in the form 
	\begin{align} \sqrt{1+|\nabla g(0)|^2} |\sqrt{s_\xi}-\sqrt{s_\eta}| +O(\va)  |\sqrt{s_\xi}-\sqrt{s_\eta}|.\label{1-3}
	\end{align}
	
	On the other hand, 
	by~\eqref{def-9-3}, we have
	\begin{align*}
	& \|\xi_x-\eta_x\|=\|x(t_\xi, s_\xi) -x(t_\eta, s_\eta)\|
	=\Bl\| t_\xi-t_\eta+s_\xi A(t_\xi) \nabla g(t_\xi)-s_\eta A(t_\eta) \nabla g(t_\eta)\Br\|,\end{align*}
	which, using the mean value theorem, equals 
	\begin{align}
	&\Bl\| t_\xi-t_\eta+ A(0) \nabla g(0) (s_\xi-s_\eta)\Br\| + O(\|t_\xi\|) |s_\xi-s_\eta| +O( s_\eta) \|t_\xi-t_\eta\|\notag\\
	&= \|t_\xi-t_\eta\|+ O(\sqrt{\va}) |\sqrt{s_\xi}-\sqrt{s_\eta}|+O(\va) \|t_\xi-t_\eta\|.\label{1-4}
	\end{align}
	Since we may choose $\va>0$ as small as we wish, we obtain from~\eqref{1-3} and~\eqref{1-4} that 
	\begin{align*}
	&\wh{\rho}_G(\xi,\eta) =\max\Bl\{\Bl|\sqrt{g(\xi_x)-\xi_y}-\sqrt{g( {\eta_x})-\eta_y}\Br|,  \|\xi_x-\eta_x\|\Br\} \\
	\sim& |\sqrt{s_\xi}-\sqrt{s_\eta}|+\|t_\xi-t_\eta\|\sim |\sqrt{s_\xi}-\sqrt{s_\eta}|+\|\xi_x-\eta_x\|.
	\end{align*}
	Note that by the definition~\eqref{rhog}
	\begin{align*}
	|g(\xi_x)-g(\eta_x) -\xi_y +\eta_y|\leq C \wh{\rho}_G (\xi,\eta),
	\end{align*}	
	which in turn implies that 
	$$|\xi_y-\eta_y| \leq C \wh{\rho}_G (\xi,\eta).$$
	It follows that 
	\begin{align*}
	\wh{\rho}_G (\xi, \eta)& \sim |\sqrt{s_\xi}-\sqrt{s_\eta}|+\|\xi_x-\eta_x\|\sim |\sqrt{s_\xi}-\sqrt{s_\eta}|+\|\xi-\eta\|
	\sim \rho_\Og (\xi, \eta).
	\end{align*}
	This completes the proof.	
\end{proof}

\begin{rem}\label{rem-6-2}
	For $\xi\in\Og$ and $\da>0$, define$$ U(\xi,\da):=\{\eta\in\Og:\  \  \rho_{\Og}(\xi,\eta)\leq \da\}.$$
	By Lemma~\ref{metric-lem}  and Lemma~\ref{lem-9-1}, it is easily seen  that 
	$$|U(\xi, \f 1n)|\sim \f 1{n^{d+1}} \Bl( \f 1n +\sqrt{\dist(\xi, \Ga)}\Br),\   \   \   \xi\in\Og.$$
\end{rem}

We are now in a position to prove Theorem~\ref{polyPartition}.\\

\begin{proof}[Proof of Theorem~\ref{polyPartition}]
	
	For convenience, we say  a subset  $K\subset \Og$ admits a polynomial partition of the unity of degree $Cn$ with parameter $m>1$   if there exist a finite subset $\Ld \subset \Og$ and  a sequence $\{\vi_\og\}_{\og\in\Ld}$ of   polynomials of degree at most $C n$  such that $\rho_\Og(\og,\og') \ge \f 1n$ for any two distinct points $\og,\og'\in\Ld$,  $\sum_{\og \in\Ld} \vi_\og (x) =1$ for every $x\in K$ and 
	$ |\vi_\og (x)| \leq C (1+n\rho_{\Og} (x,\og))^{-m}$ for every $x\in K$ and $\og\in\Ld$, in which case    $\{ \vi_\og\}_{\og\in\Ld}$ is called  a polynomial partition of the unity of degree $Cn$ on   the set $K$.
	According to Theorem~\ref{strips-0}, Remark~~\ref{rem-5-2}, and Lemma~\ref{metric-lem},  if   $G\subset \Og$ is a domain  of special type attached to $\Ga$ or if $G=Q$ is a cube such that $4Q\subset \Og$, then for any $m>1$,  $G$ 
	admits a polynomial partition of the unity of degree $Cn$ with  parameter $m$.

	Our  proof relies on the decomposition in Lemma~\ref{LEM-4-2-18-0}. 
	Let $\{\Og_s\}_{s=1}^J$ be the sequence of subsets of $\Og$ given in Lemma~\ref{LEM-4-2-18-0}.  For $1\leq j\leq J$, let $H_j =\bigcup_{s=1}^j \Og_s$.  Assume that for some $1\leq j\leq J-1$, $H_j$ admits a polynomial partition $\{u_{\og_i}\}_{i=1}^{n_0}$   of the unity  of degree $Cn$ with parameter $m>1$. By induction and  Lemma~\ref{LEM-4-2-18-0}, it  will suffice to show that $H_{j+1}$ also  admits a polynomial partition of the unity of degree $Cn$ with   parameter $m>1$.
	For simplicity, we write $H=H_j$ and $K=\Og_{j+1}$. 
	Without loss of generality, we may assume that $K=S_{G,\ld_0}$ with $\ld_0\in (\f12, 1)$ and  $G\subset \Og$ a domain of special type attached to $\Ga$.  The case when $K=Q$ is a cube such that $4Q\subset \Og$ can be treated similarly, and in fact, is simpler.

	By Theorem~\ref{strips-0}, $G$ admits a   polynomial partition $\{u_{\og_j}\}_{j=n_0+1}^{n_0+n_1}$ of the  unity of degree $Cn$ with parameter $m>1$.
	Recall $H\cap G$ contains an open ball of radius $\ga_0\in (0,1)$.  Let  $L>1$ be  such that $\Og\subset B_L[0]$, and let 
	$\t:=\f {\ga_0}{20L} \in (0,1)$. 
	According to Lemma~\ref{lem-4-3},  there exists a polynomial $R_n$ of degree at most $Cn$ such that $0\leq R_n(\xi)\leq 1$ for $\xi\in B_L[0]$,  $1-R_n(\xi)\leq \t^n$ for $\xi \in K$ and $R_n(\xi)\leq \t^n$ for $x\in \Og\setminus  G$.
	We now define 
	$$ w_j(\xi) =\begin{cases} u_{\og_j}(\xi) (1-R_n(\xi)), &\   \  \text{if $1\leq j\leq n_0$},\\
	u_{\og_j}(\xi)R_n(\xi), &\   \  \text{if $n_0+1\leq j\leq n_0+n_1$}.\end{cases}$$
	Clearly, each $w_j$ is a  polynomial of degree at most $Cn$ on $\RR^{d+1}$. 
	Since  polynomials are analytic functions  and $H\cap G$ contains an open ball of radius $\ga_0$, it follows that  
	$$\sum_{j=1}^{n_0+n_1}
	w_j(\xi) =R_n(\xi)+1-R_n(\xi)=1,\   \  \forall \xi\in \RR^{d+1}.$$
	
	Next, we prove  that for each $1\leq j\leq n_0+n_1$, 
	\begin{equation}\label{key-6-2-1}
	| w_j(\xi)|\leq C (1+n\rho_{\Og}(\xi, \og_j))^{-m},\   \    \forall \xi\in H\cup K.
	\end{equation}
	Indeed,  if $1\leq j\leq n_0$, then for $\xi\in H$,
	$$|w_j(\xi)|\leq |u_{\og_j}(\xi)|\leq C (1+n\rho_{\Og}(\xi, \og_j))^{-m},
	$$
	whereas   for $\xi\in K\subset G$,
	\begin{align*}
	| w_j(\xi)| &\leq \t^n \|u_{\og_j}\|_{L^\infty (B_L[0])}\leq C\t ^n \Bl( \f {10L}{\ga_0}\Br)^n\|u_{\og_j}\|_{L^\infty(H\cap G)}\\
	\leq & C 2^{-n}\leq  C_m (1+n\rho_{\Og}(\xi, \og_j))^{-m} ,
	\end{align*} 
	where the second step uses Lemma~\ref{lem-4-1}.
	Similarly,  if $n_0<j\leq n_0+n_1$, then for $\xi\in G$,
	$$| w_j(\xi)|\leq |u_{\og_j}(\xi)|\leq C_m (1+n\rho_{\Og}(\xi, \og_j))^{-m},
	$$
	whereas  for $\xi\in H\setminus G$, 
	$$ | w_j(\xi)| \leq \t^n \|u_{\og_j}\|_{L^\infty (B_L[0])}\leq C\t ^n \Bl( \f {10L}{\ga_0}\Br)^n\leq C 2^{-n}\leq  C_m (1+n\rho_{\Og}(\xi, \og_j))^{-m}.$$ 
	Thus, in either case, we prove the estimate~\eqref{key-6-2-1}.
	
	Finally, 
	we write the set    $A:=\{\og_1,\dots, \og_{n_0+n_1}\}$ as  a disjoint union 
	$A=\bigcup_{\og\in\Ld} I_{\og}$, where  $\Ld$ is  a subset of $A$ satisfying that  
	$\min_{\og\neq  \og'\in\Ld} \rho_{\Og}(\og,\og')\ge \f 1n$, and
	$ I_\og\subset \{\og'\in A:\  \  \rho_{\Og} (\og, \og') \leq \f1n\}$  for each $\og\in\Ld$.  We then define
	$$\vi_\og(\xi): =\sum_{j:\  \    \og_j\in I_\og} w_j(\xi),\   \ \xi \in H\cup G,\  \  \og\in\Ld,$$
	where $1\leq j\leq n_0+n_1$.
	Clearly,  each $\vi_\og$ is a polynomial of degree at most $C n$ and 
	$$\sum_{\og\in\Ld} \vi_{\og} (\xi) =\sum_{j=1}^{n_0+n_1} w_j(\xi)=1,\   \   \  \forall \xi\in H\cup G.$$
	On the other hand, we recall that  
	$\rho_{\Og}(\og_i, \og_j) \ge \f 1n$ if $1\leq i\neq j\leq n_0$ or $n_0+1\leq i\neq j\leq n_0+n_1$.
	Thus,  using   Remark~\ref{rem-6-2} and  a volume comparison argument, 
	we have that  $\# I_\og \leq C(\Og, m)$ for each $\og \in \Ld$, where $\# I$ denotes the cardinality of a set $I$. It then follows from~\eqref{key-6-2-1}   that 
	$$|\vi_\og (\xi) |\leq C (1+n\rho_\Og (\xi,\og))^{-m},\   \  \xi\in H\cup G,\   \  \og \in\Ld.$$
	Thus, we have shown that  the set $H\cup K$ admits a polynomial partition of the unity of degree $cn$ with parameter $m$, completing     the induction. 
\end{proof}

\begin{rem}
	The above proof implies $\#\Lambda=O(n^{d+1})$; recall that $\Omega\subset\R^{d+1}$.
\end{rem}

%
%
%
%

\chapter{The  direct Jackson theorem}\label{ch:direct}

\section{Jackson inequality on domains of special type}\label{Sec:8}
We will first prove  the Jackson inequality,  Theorem~\ref{THM-4-1-18}, on  a domain $G$ of special type that is attached to $\Ga=\p \Og$. Without loss of generality,
we may assume that
\begin{align*}
G:=\{ (x, y):\  \  x\in (-b,b)^{d},\   \  g(x)-1\le  y\leq  g(x)\},
\end{align*}
where $b\in (0,(2d)^{-1})$ is the base size of $G$,  and  $g$ is a $C^2$-function on $\RR^d$ satisfying that $\min_{x\in [-4b, 4b]^d} g(x)\ge4$. We may choose the base size $b$ to be sufficiently small so that 
\begin{equation}\label{8-1-18}
\max_{x\in [-4b, 4b]^d}  \|\nabla g(x)\|\leq \f 1{200db}.
\end{equation}

We first   recall some  notations from Section~\ref{sec:5} and Section~\ref{modulus:def}.
Given  $n\in\NN$,  the    partition $\{\Delta_{\bfi}\}_{\bfi\in\Ld_n^d}$ of the cube $[-b,b]^d$ is  defined  by 
\begin{equation}\label{partition-a}\Delta_{\bfi}:=[t_{i_1}, t_{i_1+1}]\times \dots \times [t_{i_{d}}, t_{i_{d}+1}]  \   \  \ \text{with}\  \    t_{i}=(-1+\f {2i}n)b,
\end{equation}
where 
$\Ld_n^d:=\{ 0, 1,\dots, n-1\}^d\subset \ZZ^d$ is the  index set.  For simplicity, we also set  $t_i=-b$ for  $i<0$, and $t_i =b$ for $i>n$, and therefore, $\Delta_{\ib}$ is defined for all $\ib\in\ZZ^d$.
Next,  the sequence,
\begin{equation}\label{8-3-0-18}
\al_j:=2\al \sin^2 (\f {j\pi}{2N}),\   \  j=0,1,\dots, N:=2\ell_1 n, 
\end{equation}
forms a Chybeshev partition of the interval $[0, 2\al]$, where 
$\al:= 1/(2\sin^2\f \pi{2\ell_1})$,  and  $\ell_1$ is   a fixed  large  positive integer  for which~\eqref{5-2-18} is satisfied. 
Note that $\al_n=1$, and 
\begin{align}\label{8-4-18-0}
\f {4j\al} {N^2} \leq \al_j-\al_{j-1} \leq \f { \pi^2 j\al} {N^2},\   \  j=1,\dots, N.
\end{align}
Finally, a  partition of the domain $G$ is  defined  as   
\begin{align*}
G&=\Bl\{(x,y):\  \  x\in [-b,b]^d,\   \   g(x)-y\in [0,1]\Br\} =\bigcup_{(\bfi,j)\in\Ld_n^{d+1}} I_{\mathbf{i},j},
\end{align*}
where 
$$I_{\mathbf{i},j}:=\Bl\{ (x, y):\  \  x\in \Delta_{\bfi},\  \   g(x)-y\in [\al_{j}, \al_{j+1}]\Br\}.$$

Next,  we introduce a few  new notations for this chapter.  Without loss of generality, we assume that $n\ge 50$. Let  $10\leq m_0, m_1\leq n/5$ be two fixed large integer parameters  satisfying 
\begin{equation}\label{8-5-18}
m_1\ge \f {32\ell_1^2 m_0^2 b^2}{\al} \|\nabla^2 g\|_{L^\infty ([-b, b]^d)}.
\end{equation}
We define,  for $\bfi\in \Ld_n^d$,  
$$\Delta_{\bfi}^\ast =[t_{i_1-m_0}, t_{i_1+m_0}]\times [t_{i_2-m_0}, t_{i_2+m_0}]\times \dots\times [t_{i_{d}-m_0}, t_{i_{d}+m_0}],$$
and for $(\ib, j) \in\Ld_n^{d+1}$, 
$$I_{\bfi,j}^\ast:=\Bl\{ (x, y):\  \ x\in \Delta_{\bfi}^\ast,\   \  \al^\ast_{j-m_1}\leq g(x)-y\leq \al^\ast_{j+m_1}\Br\},$$
where
$\al_j^\ast =\al_j$ if $0\leq j\leq n$,  $\al_j^\ast =0$ if $j<0$ and $\al_j^\ast =1$ if $j>n$.  
Let $x_{\bfi}^\ast$ be   
an arbitrarily given  point    in the set $ \Delta_{\bfi}^\ast$.   Denote by   $\zeta_{k}(x_{\bfi}^\ast)$  the unit tangent vector to the boundary $\Ga$ at the point 
$({x}^\ast_{\bfi}, g({x}^\ast_{\bfi}) )$  that is  parallel to the   $x_ky$-plane and satisfies $\zeta_{k} (x_{\bfi}^\ast)\cdot e_k>0$ for $k=1,\dots, d$; that is,  
$\zeta_{ k} (x_{\bfi}^\ast):= \f { e_k + \p_k g( x_{\bfi}^\ast) e_{d+1} }{\sqrt{1+|\p_k g( x_{\bfi}^\ast)|^2}}.$
Set
$$\mathcal{E}(x_{\bfi}^\ast):=\{\zeta_{1} (x_{\bfi}^\ast), \dots, \zeta_{d} (x_{\bfi}^\ast)\},\   \  \ib\in\Ld_n^d.$$
By Taylor's theorem, we have 
\begin{equation}\label{8-9-18}
\Bl|g(x) -H_{\bfi}(x)  \Br| \leq M_0 n^{-2},\   \  \forall x\in \Delta_{\bfi}^\ast,
\end{equation}
where     $$H_{\bfi} (x):=g(x_{\bfi}^\ast)+ \nabla g(x_{\bfi}^\ast)\cdot (x-x_{\bfi}^\ast),\   \  x\in\RR^d,$$
and 
$M_0:=8 m_0^2 b^2 \|\nabla^2 g\|_{L^\infty( [-b,b]^d)}+C_d A_0.$
Here we recall that $A_0$ is the parameter in~\eqref{modulus-special}. 
Thus, setting 
\begin{align}\label{8-7-1-18}
S_{\bfi,j}:=\Bl\{ (x,y):  x\in\Delta_{\bfi}^\ast,    H_{\bfi} (x) -\al^\ast_{j+m_1} +\f {M_0} {n^2}\leq y\leq H_{\bfi}(x) -\al^\ast_{j-m_1} -\f {M_0} {n^2}\Br\}
\end{align}
and 
\begin{equation}\label{8-8-1}
S_{\bfi, j}^\ast:=\Bl\{ (x,y):\  \  x\in\Delta_{\bfi}^\ast,\   \  H_{\bfi} (x) -\al^\ast_{j+m_1} -\f {M_0} {n^2}\leq y\leq H_{\bfi}(x) -\al^\ast_{j-m_1} +\f {M_0} {n^2}\Br\},
\end{equation}
we have  
\begin{equation}\label{8-7-0}
S_{\bfi,j} \subset I_{\bfi,j}^\ast \subset S_{\bfi,j}^\ast,\   \  (\bfi, j)\in\Ld_n^{d+1}.
\end{equation}
On the other hand, it is easily seen from~\eqref{8-3-0-18},~\eqref{8-5-18} and~\eqref{8-4-18-0} that $S_{\bfi,j}\neq \emptyset$ and 
\begin{equation} \al^\ast_{j+m_1} -\al^\ast_{j-m_1} -\f {2M_0}{n^2} \sim \f {j+M_0}{n^2}.\end{equation}
Thus,  $S_{\bfi, j}$ and $S_{\bfi,j}^\ast$ are two nonempty  compact parallepipeds with the same set $\mathcal{E}(x^\ast_{\bfi})\cup \{e_{d+1}\}$ of edge directions 
and  comparable side lengths.

With the above notations, we introduce the following local modulus of smoothness on $G$:

\begin{defn}\label{def-8-1} For $0<p\leq \infty$, define the local modulus of smoothness of  order $r$  of $f\in L^p( G)$ by  
	$$ \og_{\text{loc}}^r (f, n^{-1})_{L^p( G)}:=
	\Bl[\sum_{(\bfi,j)\in\Ld_n^{d+1}} \Bl(\og^r (f, I_{\bfi,j}^\ast; e_{d+1})_p  ^p+\og^r (f, S_{\bfi,j}; \mathcal{E}(x^\ast_{\bfi}))_p  ^p\Br) \Br]^{1/p},$$
	with the usual change of the   $\ell^p$-norm over the set $(\bfi, j)\in\Ld_n^{d+1}$  for  $p=\infty$.  
\end{defn}

In this chapter, we shall prove  the following Jackson type estimate for the above  local modulus of smoothness, from which  Theorem~\ref{THM-4-1-18} will follow.

\begin{thm}\label{THM-WT-OMEGA} For $0<p\leq \infty$, and $f\in L^p(G)$,
	\[
	E_{n} (f)_{L^p(G)} \leq C  \omega_{\text{loc}}^r(f,  n^{-1})_{L^p( G)},
	\]
	where the constant  $C$  is independent of $f$ and $n$. 
\end{thm}

\begin{rem}\label{rem:loc mod pnt choice}
	Note that $\omega_{\text{loc}}^r(f,  n^{-1})_{L^p( G)}$ depends on the choice of $x_{\bfi}^\ast$, which is an arbitrary point in $\Delta_{\bfi}^\ast$. It follows from the proof that the constant $C$ in Theorem~\ref{THM-WT-OMEGA} is independent of the selection of the points $x_{\bfi}^\ast\in \Delta_{\bfi}^\ast$.
\end{rem}

We divide the rest of this chapter  into two parts that will be  given in the two  subsequent   sections respectively. In the first part,  we shall assume Theorem~\ref{THM-WT-OMEGA}, and show how it implies Theorem~\ref{THM-4-1-18}, while the second part is devoted to the proof of Theorem~\ref{THM-WT-OMEGA}.  
\section{Proof of Theorem~\ref{THM-4-1-18}}\label{subsection-8:1}
The aim in this  section is to show that Theorem~\ref{THM-4-1-18} can be deduced from  Theorem~~\ref{THM-WT-OMEGA}. 
Recall that  for each $\bfi\in\Ld_n^d$,  $\mathcal{E}(x^\ast_{\bfi})$ is the set of  unit  tangent vectors to $\p' G^\ast$ at the point $(x_{\bfi}^\ast, g(x_{\bfi}^\ast))$, where $x_{\bfi}^\ast\in \Delta_{\bfi}^\ast$.   Thus, by Definition~\ref{def-8-1},   Theorem~\ref{THM-WT-OMEGA}, and Remark~\ref{rem:loc mod pnt choice}, to show  Theorem~\ref{THM-4-1-18}, 
it suffices to prove  that 
\begin{equation}\label{8-3-18}
\Sigma_1:=\sum_{(\bfi,j)\in\Ld_n^{d+1}}  \og^r(f, I_{\bfi, j}^\ast; e_{d+1})^p_p\leq   C\og_{\Og,\vi}^r (f, \f 1n; e_{d+1})^p_{p},\  
\end{equation}
and for $k=1,\dots, d$,
\begin{equation}\label{8-4-18}
\Sigma_2 (k):=n^d \sum_{(\bfi, j)\in\Ld_n^{d+1}}  \int_{\Delta_{\bfi}^\ast} \og^r(f, S_{\bfi,j}; \zeta_k(x_{\bfi}^\ast))^p_p dx_{\bfi}^\ast\leq C   \wt{\og}_{G}^r (f, \f 1 n)^p_p 
\end{equation}
with the usual change of the $\ell^p$ norm in the case of $p=\infty$.

To prove the estimates~\eqref{8-3-18} and~\eqref{8-4-18}, we need to use  the   average modulus of smoothness  of order $r$ on a compact interval $I=[a_I, b_I]\subset \RR$ defined as   
$$w_r(f, t; I)_p  :=\Bl(\f1t \int_{t/4r}^t\Bl( \int_{I-rh} |\tr_h^r f(x)|^p  dx\Br)\, dh  \Br)^{1/p},\     \  0<p\leq \infty,$$
with the usual change when $p=\infty$.
The average modulus $w_r(f, t; I)_p$ turns out to be equivalent to the regular modulus $\og^r(f,t)_p:=\sup_{0<h\leq t} \|\tr_h^r f\|_{L^p(I-rh)}$,  as is well known.  

\begin{lem}\cite[p.~373, p.~185]{De-Lo} \label{lem-8-1} For $f\in L^p(I)$ and $0<p\leq \infty$,
	\begin{equation}\label{key-equiv-mod-0}
	C_1 w_r(f, t; I)_p \leq \og^r (f,t)_p\leq C_2 w_r(f,t; I)_p,\   \  0<t\leq |I|,
	\end{equation}
	where the constants $C_1, C_2>0$ depend only on $p$ and $r$.
\end{lem}

For simplicity, we will assume $p<\infty$. The proof below with slight modifications works equally well for the case $p=\infty$. 

We start with the proof of~\eqref{8-3-18}. 
Using~\eqref{8-4-18-0} and~\eqref{key-equiv-mod-0}, we have    
\begin{align*}\og^r(f, I_{\bfi,j}^\ast;  e_{d+1})_p^p &=\sup_{0<h<\f {c(j+1)}{n^2}} \int_{\Delta_{\bfi}^\ast}\Bl[ \int_{g(x)-\al_{j+m_1}}^{g(x)-\al_{j-m_1}} |\tr_{h e_{d+1}}^r (f, I_{\bfi, j}^\ast, (x,y))|^p dy\Br] \, dx \\
&\sim \f {n^2} {j+1}\int_{\f {c(j+1)}{4rn^2}}^{\f {c(j+1)}{n^2}} \int_{I_{\bfi,j}^\ast } |\tr_{he_{d+1}}^r (f, I_{\bfi,j}^\ast, \xi)|^p d\xi dh.
\end{align*}
By~\eqref{funct-vi}, we note  that for $\xi=(x,y)\in I_{\bfi, j}^\ast-\f {c(j+1)}{4n^2} e_{d+1}$, 
\begin{align*} \vi_\Og (e_{d+1}, \xi)&\sim \sqrt{g(x)-y}\sim \f {j+1}n,\   \  0\leq j\leq n.\end{align*}
Thus, 
performing the  change of variable 
$h=s\vi_\Og(e_{d+1}, \xi)$ for each  fixed $\xi\in I_{\bfi,j}^\ast-\f {c(j+1)}{4n^2} e_{d+1}$,  we obtain 
\begin{align*}
& \og^r(f, I_{\bfi,j}^\ast;  e_{d+1})_p^p 
\leq C   n \int_{I_{\bfi,j}^\ast}\Bl[ \int_{0}^{\f cn} |\tr_{s\vi_\Og (e_{d+1}, \xi)e_{d+1}}^r(f, I_{\bfi, j}^\ast, \xi)|^p \, ds\Br] d\xi.
\end{align*}
It then follows that 
\begin{align*}
\Sigma_1
&\leq Cn\sum_{j=0}^{n-1} \sum_{\bfi\in\Ld_n^d}  \int_0^{\f cn}\Bl[ \int_{I_{\bfi, j}^\ast} |\tr_{s\vi_\Og (e_{d+1}, \xi)e_{d+1}}^r(f, \Og,x)|^p d\xi\Br] ds\\
&\leq C n\int_0^{\f cn} \int_{\Og }|\tr_{u\vi_\Og (e_{d+1}, \xi)e_{d+1}}^r (f,\Og,\xi)|^p\, d\xi ds\leq C \og^r_{\Og,\vi}(f, n^{-1}; e_{d+1})_p^p.
\end{align*}
This proves the estimate~\eqref{8-3-18}.

The estimate~\eqref{8-4-18} can be proved in a similar way.  Indeed, by~\eqref{2-3-18} and~\eqref{key-equiv-mod-0},   it is easily seen  that 
\begin{align*}
\og^r(f, S_{\bfi, j}; \zeta_k(x_{\bfi}^\ast))_p^p\sim n \int_0^{\f cn} \|\tr_{h \zeta_k(x_{\bfi}^\ast)}^r (f, S_{\bfi, j}) \|_{L^p(S_{\bfi,j})}^p\, dh.
\end{align*}
It follows that 
\begin{align*}
\Sigma_2(k) &\leq C n^{d+1} \int_0^{\f cn} \Bl[ \sum_{(\bfi, j)\in\Ld_n^{d+1}}  \int_{\Delta_{\bfi}^\ast} 
\|\tr_{h \zeta_k(x_{\bfi}^\ast)}^r (f, S_{\bfi, j}) \|_{L^p(S_{\bfi,j})}^p\, dx_{\bfi}^\ast\Br]\, dh\\
& \leq C n^{d}\sup_{0<h\leq \f cn}   \sum_{(\bfi, j)\in\Ld_n^{d+1}}   
\int_{S_{\bfi, j}} \int_{\|u-\xi_x\|\leq \f c n} |\tr_{h \zeta_k(u)}^r (f, S_{\bfi, j},\xi)|^p\, du\, d\xi\\
& \leq C  n^{d}\sup_{0<h\leq \f cn}   
\int_{G^n} \int_{\|u-\xi_x\|\leq \f c n} |\tr_{h \zeta_k(u)}^r (f, G,\xi)|^p\, du\, d\xi\leq C \wt{\og}_G^r(f, \f cn)_p^p,
\end{align*}
where $G^n:=\{\xi\in G:\  \  \dist(\xi, \p' G) \ge \f {A_0}{n^2}\}$. 
This proves~\eqref{8-4-18}.

\section{Proof of Theorem~\ref{THM-WT-OMEGA} }\label{subsection-8:2}
This section  is devoted to the proof of Theorem~~\ref{THM-WT-OMEGA}. The proof relies on several lemmas.

\begin{lem}\label{thm-2-1} Let  $(\bfi,j)\in \Ld_n^{d+1}$. Then for $0<p\leq \infty$ and $r\in\NN$,  
	$$E_{(d+1)(r-1)}(f)_{L^p(I_{\bfi,j}^\ast)}\leq C(p, r, d, G) \Bl[\og^r (f, I_{\bfi,j}^\ast; e_{d+1})_p+\og^r (f, S_{\bfi,j}; \mathcal{E}(x^\ast_\bfi))_p\Br].$$
\end{lem}

\begin{proof}Lemma~\ref{thm-2-1} follows directly from~\eqref{8-7-0} and Theorem~~\ref{cor-7-3}.
\end{proof}

\begin{lem}\label{lem-5-1} Given   $0<p\leq \infty$ and  $r\in\NN$, there  exist positive  constants $C=C(p, r)$ and $s_1=s_1(p,r)$ depending only on $p$ and $r$  such that  for any  integers $0\leq k, j\leq N/2$  and any $P\in\Pi_r^1$,
	\begin{equation}\label{5-2a}
	\|P\|_{L^p[\al_{j}, \al_{j+1}]}\leq C(p,r) (1+|j-k|)^{s_1}\|P\|_{L^p[\al_{k},\al_{k+1}]}.
	\end{equation}
\end{lem}
\begin{proof} First, we prove that 
	\begin{equation}\label{5-1}
	\|P\|_{L^p(I_{2t} (x))}  \leq L_{p,r}\|P\|_{L^p(I_t (x))},  \  \   \forall  P\in \Pi_r^1,\  \ \forall x\in [0, 2\al],\  \ \forall t\in (0, 1],
	\end{equation}
	where 
	$$I_t(x):=\Bl\{ y\in [0, 2\al]:\  \ |\sqrt{x}-\sqrt{y}|\leq \sqrt{2\al} t\Br\}.$$
	To see this, we note that   with   $\rho_t(x)=2\al t^2+t\sqrt{2\al  x}$, 
	\begin{equation}\label{8-17-0} [x-\f18 \rho_t(x), x+\f18 \rho_t(x)]\cap [0, 2\al]\subset I_t(x)\subset I_{2t}(x) \subset [x-4\rho_t(x), x+4\rho_t(x)],\end{equation}
	where  the first relation  can be deduced by considering the cases $0\leq x\leq \al t^2$ and $\al t^2<x\leq 4\al$ separatedly. 
	By Lemma~\ref{lem-4-1}, this implies  that  with $I_t=I_t(x)$ and  $J=[x-4\rho_t(x), x+4\rho_t(x)]$,  
	\begin{align*}
	\|P\|_{L^p(I_{2t})} &\leq \|P\|_{L^p(J)} \leq C_{p,r} \|P\|_{L^p (\f 1{32} J \cap [0, 2\al])} \leq C_{p,r}\|P\|_{L^p(I_t)},
	\end{align*}
	which  proves~\eqref{5-1}.

	Next, we note that  the doubling property~\eqref{5-1} implies that for any $x, x'\in [0, 2\al]$ and any $t\in (0, 1]$, 
	\begin{equation}\label{8-14-18-00}
	\|P\|_{L^p(I_{t} (x))}  \leq L_{p,r} \Bl( 1+ \f {|\sqrt{x}-\sqrt{x'}|}{\sqrt{2\al} t}\Br)^{s_1}\|P\|_{L^p(I_t (x'))},  \  \   \forall  P\in \Pi_r^1,
	\end{equation}
	where $s_1=(\log L_{p,r})/\log 2$.

	Finally, for each $1\leq k\leq N/2$, we may write  $ [\al_k, \al_{k+1}] =I_{t_k} (x_k)$
	with  $t_k:=\f {\sqrt{\al_{k+1}}-\sqrt{\al_k}}{2\sqrt{2\al}}$ and  $x_k :=\f {(\sqrt{\al_k} +\sqrt{\al_{k+1}})^2} 4$. Note also that    
	by~\eqref{8-4-18-0},  
	\begin{align}\label{8-15-18-00}
	\f {\sqrt{2}|k-j|}{2N}\leq 	\f{|\sqrt{\al_j}-\sqrt{ \al_k}|}{\sqrt{2\al}}\leq \f {\pi|k-j|} {2N},\   \ 0\leq k, j\leq N/2.
	\end{align}
	It then follows by~\eqref{8-14-18-00} and~\eqref{8-15-18-00} that 
	\begin{align*}
	\|P\|_{L^p[\al_{j}, \al_{j+1}]}&\leq \|P\|_{L^p (I_{\pi/(4N)}(x_j))} \leq L_{p,r} \Bl ( 1+ \f {4N|\sqrt{x_j}-\sqrt{x_k}|}{\sqrt{2\al} \pi}\Br)^{s_1}
	\|P\|_{L^p (I_{\pi/(4N)}(x_k))} \\
	&\leq L_{p,r}^4 ( 1+|k-j|)^{s_1} \|P\|_{L^p[\al_{k}, \al_{k+1}]}.
	\end{align*}
\end{proof}

For $x=(x_1,\dots, x_d)\in\R^d$, we set  $\|x\|_\infty :=\max_{1\leq j\leq d} |x_j|$. 

\begin{lem}\label{lem-8-4} Given   $0<p\leq \infty$ and  $r\in\NN$, there  exist positive  constants $C=C(p, r, d)$ and $s_2=s_2(p,r,d)$ depending only on $p$, $r$ and $d$ such that  for any  $\ib, \kb\in\Ld_n^d$  and $Q\in\Pi_r^d$,
	\begin{equation}\label{5-2b}
	\|Q\|_{L^p(\Delta_{\bfi})}\leq C(p,r,d) (1+\|\bfi-\kb\|_\infty)^{s_2}\|Q\|_{L^p(\Delta_{\kb})}.
	\end{equation}
\end{lem}

\begin{proof} The proof of Lemma~\ref{lem-8-4} is similar to that of Lemma~\ref{lem-5-1}, and in fact, is simpler.  It is a direct consequence of Lemma~\ref{lem-4-1}. 
\end{proof}

\begin{lem}\label{lem-5-2} Given $0<p\leq \infty$ and $r\in\NN$, there exists a positive number $\ell=\ell(p,r,d)$ such that 	
	for any  $(\bfi, j), (\mathbf{k}, l)\in \Ld_n^{d+1}$ and any $Q\in\Pi_r^{d+1}$, 
	\begin{equation}\label{desire}
	\|Q\|_{L^p(I_{\bfi, j})} \leq C \Bl(1+\max\{\|\bfi-\mathbf{k}\|_\infty,  |j-l|\}\Br)^{\ell}\|Q\|_{L^p(I_{\mathbf{k}, l})},
	\end{equation}
	where the constant $C$ depends only on $p, d, r$ and $\|\nabla^2 g\|_\infty$.
\end{lem}
\begin{proof} For simplicity, we shall prove Lemma~\ref{lem-5-2} for the case of $0<p<\infty$ only. The proof below with slight modifications works for  $p=\infty$.

	Writing 
	$$\|Q\|^p_{L^p(I_{\bfi, j})} = \int_{\Delta_{\bfi}}\Bl[\int_{\al_{j-1}}
	^{\al_{j}} |Q({x}, g({x})-u)|^p\, du\Br] d{x},$$
	and using  Lemma~\ref{lem-5-1}, we  obtain 
	\begin{align}\label{8-18-0}
	\|Q\|^p_{L^p(I_{\bfi, j})}\leq C(p,r) (1+|j-l|)^{s_1} \Bl[ \int_{\Delta_{\bfi}}\int_{g({x})-\al_{l}}
	^{g({x})-\al_{l-1}} |Q({x}, y)|^p\, dy d{x}\Br].
	\end{align}
	Using   Taylor's theorem, we have that
	\begin{equation}\label{5-3}
	|g({x})-t_{\bfi} ({x})|\leq \f {A}{2b^2}\|{x}-{x}_{\bfi}\|_\infty^2,\    \    \ \forall x\in [-b,b]^d,
	\end{equation}
	where   ${x}_{\bf i}$ is the center of the cube $\Delta_{\bfi}$, 
	$t_{\bfi}({x}):=g({x}_{\bfi})+\nabla g({x}_{\bfi})\cdot ({x}-{x}_{\bfi})$, and 
	$A:=2b^2d^2\|\nabla^2 g\|_{L^\infty[-b,b]^d}/2$.
	Thus,  the double  integral in the square brackets on the right hand side of~\eqref{8-18-0} is bounded above by 
	\begin{align*}
	&   \int_{\al_{\ell-1}}^{\al_\ell+\f {A}{n^2}}\Bl[  \int_{\Delta_{\bfi}}\Bl|Q\Bl({x}, t_{\bfi}({x})
	-u+\f A {2n^2}\Br)\Br|^p\, d{x}\Br] du=: I.
	\end{align*}
	However, applying  Lemma~\ref{lem-8-4} to this last  inner integral in the square brackets, we obtain 
	\begin{align}\label{8-20-1}
	I&\leq C(p,r,d)     (1+\|\bfi-\mathbf{k}\|_\infty)^{s_2} \int_{\Delta_{\mathbf k}}\Bl[\int_{t_{\bfi}({x})
		-\al_\ell-\f A {2n^2}}^{t_{\bfi}({x})
		-\al_{\ell-1}+\f A {2n^2}} |Q({x}, u)|^p\, du\Br] d {x}.
	\end{align}
	By~\eqref{5-3},  this last   integral in the square brackets on the right hand side of~\eqref{8-20-1} is bounded above by 
	\begin{align*}
	& \int_{g({x})
		-\al_\ell-\f {4A(1+\|\mathbf{k}-\bfi\|_\infty^2)}{n^2}}^{g({x})
		-\al_{\ell-1}+\f {4A(1+\|\mathbf{k}-\bfi\|_\infty^2)}{n^2}} |Q({x}, u)|^p\, du 
	=\int^{
		\al_\ell+\f {4A(1+\|\mathbf{k}-\bfi\|_\infty^2)}{n^2}}_{\al_{\ell-1}-\f {4A(1+\|\mathbf{k}-\bfi\|_\infty^2)}{n^2}} |Q({x}, g({x})-y)|^p\, dy, \end{align*}
	which,  using  Lemma~\ref{lem-4-1} and the fact that $\al_\ell-\al_{\ell-1}\ge c n^{-2}$,  is controlled above  by 
	\begin{align*}
	C(p,r)  \Bl(A (1+\|\mathbf {k}-\bfi\|_\infty)\Br)^{2rp+4}  \int^{\al_\ell}_{
		\al_{\ell-1}} |Q({x}, g({x})-y)|^p\, dy.
	\end{align*}

	Putting the above together, we prove  that 
	\begin{align*}
	\|Q\|^p_{L^p(I_{\bfi, j})}\leq C(p,r,d) (1+|j-l|)^{s_1}   (1+\|\bfi-\mathbf{k}\|_\infty)^{s_2+2rp+4}\|Q\|^p_{L^p(I_{\kb, l})}.
	\end{align*}
	This leads to the desired estimate~\eqref{desire} with $\ell= (s_1+s_2+2rp+4)/p$.
\end{proof}

Now we are in the position to prove Theorem~\ref{THM-WT-OMEGA}. 
\begin{proof}[Proof of Theorem~\ref{THM-WT-OMEGA}]  We shall prove the result for the case of   $0<p<\infty$ only.
	The proof below with slight modifications works equally well for the case $p=\infty$.
	
	For simplicity, we  use the Greek letters $\ga, \b,\dots$ to denote  indices in the set $\Ld_n^{d+1}$. 
	By Lemma~\ref{thm-2-1}, for each  $\ga:=(\ib, j)\in\Ld_n^{d+1}$  there exists a polynomial  $s_{\ga}\in\Pi_{(d+1)(r-1)}^{d+1}$  such that 
	\begin{align}\label{5-4}
	\|f-s_{\ga}\|_{L^p(I_{\ga}^\ast)} \leq C(p, r, d) W^r(f, I_{\ga}^\ast)_p,
	\end{align}
	where 
	$$  W^r(f, I_{\ga}^\ast)_p:=\og^r (f, I_{\ga}^\ast; e_{d+1})_p+\og^r (f, S_{\ga}; \mathcal{E}(x^\ast_\bfi))_p.$$
	Let $\{q_{\ga}:\  \ \ga\in\Ld_n^{d+1}\}\subset \Pi_{\lfloor n/(r(d+1))\rfloor}^{d+1}$ be the polynomial  partition of the unity as  given in Theorem~\ref{strips-0} and  Remark~\ref{rem-6-3} with a  large parameter $m>2d+2$, to be specified later. Define 
	$$P_n(\xi):=\sum_{\ga\in\Ld_n^{d+1}} s_\ga (\xi) q_\ga(\xi)\in\Pi_{n}^{d+1}.$$
	Clearly, it is sufficient  to prove that 
	\begin{equation}\label{8-22-00}
	\|f-P_n\|_{L^p(G)}\leq C \og^r_{\text{loc}} (f, \f1n)_p.
	\end{equation}

	To show~\eqref{8-22-00},  we write, for each  $\b\in \Ld_n^{d+1}$,
	\begin{align*}
	f(\xi)-P_n(\xi)&=f(\xi)-s_\b(\xi)+\sum_{\ga\in \Ld_n^{d+1}} (s_\b(\xi)-s_\ga (\xi))q_\ga (\xi).
	\end{align*}
	It follows by Theorem~\ref{strips-0}  that 
	\begin{align*}
	\|f-P_n\|_{L^p(I_\b)}^p &\leq C_p \|f-s_\b\|_{L^p(I_\b)}^p +C_p \sum_{\ga\in\Ld_n^{d+1}} \|s_\b-s_\ga\|^p_{L^p(I_\b)} (1+\|\b-\ga\|_\infty)^{-mp_1},
	\end{align*}
	where $p_1:=\min\{p,1\}$.
	Using~\eqref{5-4}, we then  reduce  to showing that   \begin{align}\label{8-23-00}
	\Sigma_n'&:=\sum_{\b\in\Ld_n^{d+1}}\sum_{\ga\in\Ld_n^{d+1}} \|s_\b-s_\ga\|^p_{L^p(I_\b)} (1+\|\b-\ga\|_\infty)^{-mp_1}\leq C \og_{\text{loc}}^r(f,\f1n)_p^p.
	\end{align}
	
	To show~\eqref{8-23-00},  we  claim that  there exists a positive number $s_3=s_3(p,d,r)$ such that for any  $\ga,\b\in\Ld_n^{d+1}$,
	\begin{equation}\label{claim-8-23} \|s_\ga-s_\b\|^p_{L^p(I_\ga)} \leq C (1+\|\ga-\b\|_\infty)^{s_3 p} \sum_{\eta\in \mathcal{I}_{k_0} ( \ga)} W^r(f, I^\ast_\eta)^p_p,\end{equation}
	where  $k_0:=1+\|\b-\ga\|_\infty$, and
	$$ \mathcal{I}_t(\ga):=\{ \eta\in\Ld_n^{d+1}:\  \ \|\ga-\eta\|_\infty\leq t\}\   \ \text{for $ \ga\in\Ld_n^{d+1}$ and $t>0$}.$$
	For the moment, we assume~\eqref{claim-8-23} and proceed with the proof of~\eqref{8-23-00}. Indeed, 
	we have 
	\begin{align*}
	\Sigma_n'  &\leq C\sum_{\b\in\Ld_n^{d+1}}\sum_{k=1}^\infty  k^{-mp_1}\sum_{\ga\in \mathcal{I}_k(\b)\setminus \mathcal{I}_{k-1}(\b)}\|s_\b-s_\ga\|_{L^p(I_\b)}^p,\end{align*}
	which,      using~\eqref{claim-8-23}, is bounded above by 
	\begin{align*}
	&  C\sum_{k=1}^\infty k^{-mp_1+s_3p+2d+2} \sum_{\eta\in  \Ld_n^{d+1}}W^r(f, I_\eta^\ast)_p^p.
	\end{align*}  Choosing  the parameter $m$ to be bigger than $s_3p/p_1 + (2d+4)/p_1$, we then  prove~\eqref{8-23-00}.

	It remains to prove the claim~\eqref{claim-8-23}.  A crucial ingredient in the proof   is to construct a sequence $\{\ga_1,\dots, \ga_{N_0}\}$ of distinct indices in $\Ld_n^{d+1}$ with the properties that  $N_0\leq  C ( 1+\|\ga-\b\|_\infty)^2$,
	$\ga_1=\ga$, $\ga_{N_0}=\b$, and for $j=0,\dots, N_0-1$,
	\begin{align}\label{8-25-00}
	I_{\ga_j} \subset I_{\ga_{j+1}}^\ast\   \  \text{and}\   \   \|\ga_j-\ga\|\leq 1+\|\ga-\b\|.
	\end{align}
	Indeed, once such a sequence is constructed, then 
	we have 
	\begin{align*}
	\|s_\ga-s_\b\|_{L^p(I_\ga)}^p&\leq N_0^{\max\{p,1\}-1}  \sum_{j=1}^{N_0-1}\|s_{\ga_j}-s_{\ga_{j+1}}\|^p_{L^p(I_{\ga})}, \end{align*}
	which, using~\eqref{8-25-00} and  Lemma~\ref{lem-5-2} with $\ell=\ell(p,r,d)>0$, is estimated above by
	\begin{align*}
	&\leq  C N_0^{\max\{p,1\}-1} (1+\|\ga-\b\|_\infty)^{\ell p}  \sum_{j=1}^{N_0-1}\|s_{\ga_j}-s_{\ga_{j+1}}\|^p_{L^p(I_{\ga_j})}.\end{align*}
	However, using~\eqref{8-25-00} and~\eqref{5-4}, we have that 
	\begin{align*}
	\|s_{\ga_j}-s_{\ga_{j+1}}\|^p_{L^p(I_{\ga_j})}\leq &C_p \Bl[ \|f-s_{\ga_j}\|_{L^p(I_{\ga_j})}^p +\|f-s_{\ga_{j+1}}\|_{L^p(I^\ast_{\ga_{j+1}})}^p\Br]\\
	\leq& C(p,r,d)\Bl[  W^r(f, I_{\ga_j}^\ast)_p^p+ W^r(f, I_{\ga_{j+1}}^\ast)_p^p\Br].
	\end{align*}
	Putting the above together, we prove the claim~\eqref{claim-8-23} with $s_3:=\ell+2\max\{1, \f 1p\}$.

	Finally, we construct the sequence $\{\ga_1,\dots, \ga_{N_0}\}$ as follows.  Assume that  $\ga=(\mathbf{k}, l)$, and  $\b=(\mathbf{k}', l')$. Without loss of generality, we may assume that   $l\leq l'$. (The case $l>l'$ can be treated similarly.)
	Recall that 
	$ \Delta_{\ib}:=\Bl\{x\in \RR^d:\  \ \|x-x_{\ib}\|_\infty\leq \f b n\Br\},$
	where  ${x}_{\bfi}$ is  the center of the cube $\Delta_{\bfi}$.
	Let $\{z_j\}_{j=0}^{n_0+1}$ be a sequence of points on the line segment $[x_{\kb}, x_{\kb'}]$ satisfying that  $z_0 =x_{\kb}$, $z_{n_0+1} =x_{\kb'}$, $\|z_j-z_{j+1}\|_\infty =\f {3b} n $ for $j=0,1,\dots, n_0-1$  and $\f {3b} n \leq \|z_{n_0}-z_{n_0+1}\|_\infty<\f {6b}n$, where $n_0+1\leq  \f 23 \|\kb-\kb'\|_\infty$.  Let $\ib_j \in\Ld_n^d$ be such that $z_j\in\Delta_{\ib_j}$ for $0\leq j\leq n_0+1$. 
	Since $\f {3b} n \leq \|z_j-z_{j+1}\|_\infty \leq  \f {6b} n$, the cubes $\Delta_{\ib_j}$ are distinct and  moreover
	\begin{equation}\label{8-25-0}
	\Delta_{\ib_j} \subset 9 \Delta_{\ib_{j+1}},\  \  j=0,1,\dots, n_0.
	\end{equation}
	In particular, this implies that $\ib_0=\kb$ and $\ib_{n_0+1} =\kb'$.
	It can also be easily seen from the construction  that for $j=0,\dots, n_0+1$, 
	\begin{equation}\label{8-27}
	\|\ib_j -\kb\|_\infty \leq \|\kb-\kb'\|_\infty+1.
	\end{equation}
	Next,  we order   the indices $(\ib_j, k)$, $0\leq j\leq n_0+1$, $l\leq k\leq l'$ as follows: 
	\begin{align*}
	(\ib_0, l), (\ib_0, l+1),\dots, (\ib_0, l'),
	(\ib_1, l'), (\ib_1, l'-1),\dots, (\ib_1, l), (\ib_2, l),\dots, (\ib_{n_0+1}, l').      
	\end{align*}
	We  denote the resulting  sequence by  
	$ \{\ga_1, \ga_2,\dots, \ga_{N_0}\},$
	where
	$$N_0\leq (1+|l-l'|) (n_0+2) \leq ( 1+\|\ga-\b\|_\infty)^2.$$ 
	Clearly, $\ga_1=\ga$, and  $\ga_{N_0}=\b$. Moreover, by~\eqref{8-27}, we have $\|\ga_j-\ga\|_\infty\leq 1+\|\ga-\b\|_\infty$ for $j=1,\dots, N_0$, whereas by~\eqref{8-25-0}, 
	$I_{\ga_j} \subset I_{\ga_{j+1}}^\ast $ for  $j=0,\dots, N_0-1$.
	This completes the proof. 
\end{proof}

\chapter{Camparison with average  moduli}\label{ch:IvanovModuli}
In this chapter, we shall prove that the   moduli of smoothness,   defined in~\eqref{eqn:defmodulus}  can be controlled above by  Ivanov's moduli of smoothness, defined  in~\eqref{eqn:ivanov}.  By Remark~\ref{rem-3-2}, it is enough to show 

\begin{thm}\label{thm-9-1}
	There exist 	a parameter $A_0>1$  and a constant $A>1$ such that for any $0<q\leq p\leq \infty$,
	$$\og_\Og^r(f, \f 1n; A_0)_p \leq C \tau_r (f, \f A n)_{p,q},$$
	where the constant $C$ is independent of $f$ and $n$. 
\end{thm}

As a result, we may establish the Jackson inequality for Ivanov's moduli of smoothness for  any dimension   $d\ge 1$ and  the full range of $0<q\leq p\leq \infty$.

\begin{cor}  	
	If $f\in L^p(\Og)$, $0< q\leq p \leq \infty$ and $r\in\NN$, then
	$$ E_n (f)_p \leq C_{r, \Og} \tau_r (f, \f A n)_{p,q}.$$
\end{cor}

Recall that for  $S\subset \R^d$,
$$S_{rh}:=\Bl\{\xi\in S:\  \ [\xi, \xi+rh]\subset S\Br\},   \   r>0, h\in\R^d.$$
The proof of Theorem~\ref{thm-9-1}  relies on the following lemma, which  generalizes Lemma~7.4 of~\cite{Di-Pr08}.
\begin{lem}\label{lem-9-1:Dec}
	Let     $r\in\NN$,   $h\in\R^d$ and $\da_0\in (0,1)$.  Assume that  $(S,E)$ is a pair of  subsets   of $\R^d$  satisfying  that  for  each $\xi\in S_{rh}$,  there exists   a convex subset $E^\xi$ of $E$ such that  $|E^\xi|\ge \da_0 |E|$
	and   $[\xi, \xi+rh]\subset E^\xi$.
	Then for any $0<q\leq p <\infty$ and $f\in L^p(E)$, we have 
	\begin{equation}\label{9-1-18}
	\|\tr_h^r (f, S, \cdot)\|_{L^p(S)} \leq  C(q, d, r)\Bl(\int_S \Bl(\f 1 {\da_0 |E|}  \int_E \bl| \tr_{(\eta-\xi)/r} ^r (f, E, \xi)\br|^q\, d\xi\Br)^{\f pq}\, d\eta\Br)^{\f1p},
	\end{equation}
	where the constant $C(q,d, r)$ is independent of $S$, $E$ and $q$ if $q\ge 1$.
\end{lem}
Lemma~\ref{lem-9-1:Dec} was  proved in \cite[Lemma 7.4]{Di-Pr08}
in the case  when  $p=q$ and   $E=S$ is convex.   For the general case, it can be obtained by  modifying the proof there.  
\begin{proof}   
	The proof is based  on the following combinatorial identity, which was  proved in \cite[Lemma 7.3]{Di-Pr08}:   if $\xi, \eta\in\R^d$ and $f$ is defined on the convex hull of the set $\{\xi, \xi+rh, \eta\}$, then 
	\begin{align}\label{9-2-0}
	\tr^r_h f(\xi) =&\sum_{j=0}^{r-1} (-1)^j \binom r j \tr^r f\Bl[\xi+jh, \  \f jr (\xi+rh)+(1-\f jr)\eta\Br]\\
	&	-\sum_{j=1}^r (-1)^j \binom r j \tr^r f \Bl[ (1-\f jr) \xi+\f jr \eta,\  \  \xi+rh \Br],\notag
	\end{align}
	where  we used the notation   
	$\tr^r f[u, v] :=\tr_{(v-u)/r} ^r f(u)$ for $u,v\in\RR^d$. 
	
	Since $E^\xi$ is a convex set containing the line segment $[\xi, \xi+rh]$ for each $\xi\in S_{rh}$, we obtain from~\eqref{9-2-0} that   for  $\xi\in S_{rh}$,
	\begin{align*}
	|\tr_h^r f(\xi) |\leq&  C_r \max_{0\leq j\leq r-1}
	\Bl(\f 1 {|E^\xi|}	\int_{E^\xi}	\Bl|\tr^rf\bl[\xi+jh, \  \f jr (\xi+rh)+(1-\f jr)\eta\br]\Br|^q\, d\eta\Br)^{\f1q}\\
	&+  C_r \max_{1\leq j\leq r}
	\Bl(\f 1 {|E^\xi|}	\int_{E^\xi}\Bl|\tr^r f \bl[ (1-\f jr) \xi+\f jr \eta,\  \  \xi+rh \br]\Br|^q\, d\eta\Br)^{\f1q}.
	\end{align*}
	Taking the $L^p$-norm  over the set  $S_{rh}$ on both sides of this last inequality,   we obtain
	\begin{align*}
	\Bl(\int_{S_{rh}}	|\tr_h^r f(\xi) |^p\, d\xi\Br)^{\f1p}\leq&  C_{r}\da_0^{-\f1q} \Bl[\max_{0\leq j\leq r-1} I_j(h) + \max_{1\leq j\leq r}K_j(h)\Br],
	\end{align*}
	where 
	\begin{align*}
	I_j(h):&=\Bl(\int_{S_{rh}}\Bl(\f 1{|E|}	\int_{E^\xi}	\Bl|\tr^r f\bl[\xi+jh, \  \f jr (\xi+rh)+(1-\f jr)\eta\br]\Br|^q\, d\eta\Br)^{\f pq}\, d\xi\Br)^{\f1p},\\
	K_j(h):&=\Bl(\int_{S_{rh}}\Bl(\f 1{|E|}	\int_{E^\xi}\Bl|\tr^r f \bl[ (1-\f jr) \xi+\f jr \eta,\  \  \xi+rh \br]\Br|^q\, d\eta\Br)^{\f pq}\, d\xi\Br)^{\f1p}.	
	\end{align*}
	For the term   $I_j(h)$ with  $0\leq j\leq r-1$, we have
	\begin{align*}
	I_j(h)&=\Bl(\int_{S_{rh}+jh}\Bl(\f 1 {|E|}	\int_{E^{u-jh}}	\Bl|\tr^r f\bl[u, \  \f jr (u+(r-j)h)+(1-\f jr)\eta\br]\Br|^q\, d\eta\Br)^{\f pq}\, du\Br)^{\f1p}\\
	&\leq r^{d/q}
	\Bl(\int_{S_{rh}+jh}\Bl(\f 1 {|E|}	\int_{E^{u-jh}}	\Bl|\tr^rf[u, \  v]\Br|^q\, dv\Br)^{\f pq}\, du\Br)^{\f1p},
	\end{align*}
	where we used the change of variables  $u=\xi+jh$ in the first step,  the change of variables  $v=\f jr (u+(r-j) h)+(1-\f jr) \eta$  and the fact that each set  $E^\xi$ is convex  in the second step.
	Since 
	$[u,v] \subset E^{u-jh} \subset E$  whenever $u\in S_{rh}+jh$ and $v\in E^{u-jh}$ and since  $\tr^r f [u,v]=\tr^r f[v,u]$, it follows that 
	$$ I_j(h) \leq r^{d/q} 	\Bl(\int_{S}\Bl(\f 1 {|E|}	\int_{E}	\Bl|\tr^r_{(u-v)/r} (f, E, v)\Br|^q\, dv\Br)^{\f pq}\, du\Br)^{\f1p}.$$
	
	The  terms $K_j(h)$,  $1\leq j\leq r$ can be estimated in a similar way. In fact,  making  the change of variables $u=\xi+rh$ and $v=(1-\f jr)(u-rh) +\f jr \eta$, we   obtain 
	\begin{align*}
	K_j(h)&=\Bl(\int_{S_{-rh}} 
	\Bl(\f 1 {|E|} 	\int_{E^{u-rh}}\Bl|\tr^r f \bl[ (1-\f jr) (u-rh)+\f jr \eta,\  \  u \br]\Br|^q\, d\eta\Br)^{\f pq}\, du\Br)^{\f1p}\\
	&\leq r^{d/q} \Bl(\int_{S_{-rh}} 
	\Bl(\f 1 {|E|} 	\int_{E^{u-rh}}\Bl|\tr^r f [v,\  \  u ]\Br|^q\, dv\Br)^{\f pq}\, du\Br)^{\f1p}\\
	&\leq r^{d/q} 	\Bl(\int_{S}\Bl(\f 1 {|E|}	\int_{E}	\Bl|\tr^r_{(u-v)/r} (f, E, v)\Br|^q\, dv\Br)^{\f pq}\, du\Br)^{\f1p}.
	\end{align*}
	
	Putting the above together,  we complete  the proof. 	
\end{proof}

We are now in a position to prove Theorem~\ref{thm-9-1}.

\begin{proof}[Proof of  Theorem~\ref{thm-9-1}]
	We shall  prove Theorem~\ref{thm-9-1} for  $p<\infty$ only.
	The case $p=\infty$ can be deduced by letting $p\to\infty$. In fact,  all the general constants below are independent of $p$ as $p\to\infty$.

	
	
	By Lemma~\ref{lem-2-1-18}, there exists $\da_0\in (0,1)$ such that 
	$\Og\setminus \Og(\da_0) \subset \bigcup_{j=1}^{m_0} G_j,$
	where 
	$$\Og(\da_0):=\{\xi\in\Og:\  \  \dist(\xi, \Ga) > \da_0\}.$$
	We claim that 
	for any $ 0<t< \f {\da_0}{ 8\diam (\Og) +8}$,
	\begin{align}\label{9-5-2}
	\sup_{\|h\|\leq  t} \Bl\|\tr_{h\vi_{\Og} (h, \cdot)}^r (f, \Og, \cdot)\Br\|_{L^p(\Og(\da_0))}\leq C_{q,d} \tau_r (f, A_1t)_{p,q}.
	\end{align} Indeed,  using Fubini's theorem and Lemma~\ref{lem-8-1},  we have 
	\begin{align*}
	\sup_{\|h\|\leq  t} \Bl\|\tr_{h\vi_{\Og} (h, \cdot)}^r f\Br\|_{L^p(\Og(\da_0))}\leq C_d  \sup_{\|h\|\leq t} \Bl\|\tr_{h}^r f\Br\|_{L^p(\Og(\da_0/2))}.
	\end{align*}
	Let $\{\og_1,\dots, \og_{N}\}$ be a subset of $\Og(\da_0/2)$ such that $\min_{1\leq i\neq j\leq N} \|\og_i-\og_j\|\ge t$ and 
	$\Og(\da_0/2) \subset \bigcup_{j=1}^{N} B_j$,  where $B_j:=B_{t} (\og_j)$.
	Using Lemma~~\ref{thm-9-1}, we then  have 
	\begin{align*}
	&\sup_{\|h\|\leq t} \Bl\|\tr_{h}^r f\Br\|^p_{L^p(\Og(\da_0/2))}\leq C_{p} \sum_{j=1}^{N} \sup_{\|h\|\leq t} \Bl\|\tr_{h}^r(f, 2B_j, \cdot)\Br\|^p_{L^p(B_j)}\\
	&\leq C_{q}  \sum_{j=1}^{N}\int_{2B_j} \Bl(\f 1{t^{d+1}} \int_{ B_{4t} (\xi)} |\tr_{(\eta-\xi)/r}^r f (\xi)|^q \, d\xi\Br)^{\f pq} \, d\eta\leq C_{q,d} \tau_r (f, A_1 t)_{p,q}^p.
	\end{align*}
	This proves the claim~\eqref{9-5-2}.

	Now using~\eqref{9-5-2} and Definition~\ref{def:modulus}, we   reduce to showing that  for each $x_i$-domain $G\subset \Og$   attached to $\Ga$, and a sufficiently large parameter $A_0$, 
	\begin{equation}\label{9-3}
	\wt{\og}^r_G (f, \f1n; A_0)_{L^p(G)}\leq C  \tau_r (f, \f {A_1} n)_{p,q}
	\end{equation}
	and  
	\begin{equation}\label{9-4}
	\sup_{0<s\leq  \f 1n} \|\tr_{s \vi_{\Og} (e_i,\cdot)e_i }^r (f, G,\cdot)\|_{L^p(G)}\leq C  \tau_r (f, \f {A_1} n)_{p,q}.
	\end{equation}

	Without loss of generality, we may assume that $e_i=e_{d+1}$, $G$ takes the form~\eqref{standard} with small base size  $b\in(0,1)$, and $n\ge N_0$, where $N_0$ is a large positive integer depending only on the set $\Og$.
	We  follow  the same notations as   in  Chapter~~\ref{Sec:8} with sufficiently large  parameters $m_0$ and $m_1$. Thus,  	
	$\{I_{\ib, j}:\  \  (\ib, j) \in\Ld_{n}^{d+1}\}$ is a partition of $G$, and  $S_{\ib,j}\subset I_{\ib,j}$ is  the compact parallepiped as defined in\eqref{8-7-1-18}.

	We start with the proof of~\eqref{9-3}.
	Given a  parameter $\ell>1$, we   define 
	\begin{align*}
	S_{\bfi,j}^ \diamond:=\Bl\{ (x,y):\  \  & x\in(\ell \Delta_{\bfi}^\ast)\cap [-2b, 2b]^d,\   \  H_{\bfi} (x) -\al^\ast_{j+m_1} +\f {M_0-\ell} {n^2}\leq y\leq\\
	&\leq  H_{\bfi}(x) -\al^\ast_{j-m_1} -\f {M_0-\ell} {n^2}\Br\},
	\end{align*}
	where $\ell \Delta_{\ib}^\ast$ denotes the dilation of the cube  $\Delta_{\ib}^\ast$ from its center $x_{\ib}$.  
	We  choose   the parameter $\ell$ sufficiently large  so that 
	\begin{enumerate}[\rm (i)]
		\item for any $\xi=(\xi_x,\xi_y)\in I_{\ib,j}$ and  $u\in B_{n^{-1} } (\xi_x)\subset \R^d$,  	$ \Bl[\xi, \xi +\f rn \zeta_k(u)\Br]\subset      S_{\ib,j}^ \diamond$ for all $1\leq k\leq d$; 
		\item  there exists a constant $c_0>0$ such that 
		$I_{\ib, j} \subset S_{\ib,j}^{\diamond}\subset G^\ast$ whenever  $\ib\in\Ld_n^{d}$ and  $j\ge c_0 \ell$.
	\end{enumerate}
	Furthermore, we may also choose the parameter $A_0$  large enough  so that 
	with $\Ld_{n,\ell}^{d+1}:=\{(\ib,j)\in \Ld_n^{d+1}:\   \  c_0\ell\leq j\leq n\}$, 
	$$   G_n:=\Bl\{\xi\in G:\  \  \dist(\xi, \p' G) \ge  \f {A_0} {n^2}\Br\}\subset \bigcup_{(\ib,j)\in\Ld_{n,\ell}^{d+1}} I_{\ib,j}.$$
	With the above notation, we  have that  for any $0<s\leq \f 1n$ and $k=1,\dots, d$, 	
	\begin{align*}
	&n^d\int_{G_n} \int_{\|u-\xi_x\|\leq \f1n} |\tr_{s \zeta_k(u)}^r (f, G^\ast,\xi)|^p  \, du d\xi\leq C_d \sum_{(\ib,j)\in\Ld_{n,\ell}^{d+1}} \sup_{\zeta \in\SS^d} \int_{S^{\diamond}_{\ib,j}}  |\tr_{s \zeta}^r (f, S_{\ib,j}^\diamond,\xi)|^p  d\xi,\end{align*}
	which, using  Lemma~\ref{lem-9-1:Dec}, is estimated above  by 
	\begin{equation}\label{9-5}
	C_{q,d,r} \sum_{(\ib,j)\in\Ld_{n,\ell}^{d+1}}
	\int_{S_{\ib,j}^\diamond}\Bl( \f 1 {|S_{\ib,j}^{\diamond}|}  \int_{S_{\ib,j}^\diamond} \bl| \tr_{(\eta-\xi)/r} ^r f( \xi)\br|^q\, d\xi\Br)^{\f pq}\, d\eta.
	\end{equation}
	However, By  Lemma~\ref{metric-lem},  there exists a constant $A_1>1$ such that for each $(\ib, j)\in\Ld_{n,\ell}^{d+1}$, 
	$$    U(\eta_{\ib,j}, \f 1{n A_1})\subset  S_{\ib, j}^{\diamond}\subset  U(\eta_{\ib,j}, \f {A_1}{2n}) \    \  \text{for some  $\eta_{\ib,j}\in S_{\ib,j}^\diamond$}.$$
	Thus,  by  Remark~\ref{rem-6-2},  the sum in~\eqref{9-5} is controlled above by  a constant multiple of 
	\begin{align*}
	\int_{\Og}\Bl( \f 1 {|U(\xi, \f {A_1} n)|}  \int_{U(\xi, \f {A_1} n)} \bl| \tr_{(\eta-\xi)/r} ^r (f,\Og, \xi)\br|^q\, d\xi\Br)^{\f pq}\, d\eta= \tau_r (f, \f {A_1} n)_{p,q}^p.
	\end{align*}
	This completes the proof of~\eqref{9-3}.

	It remains to prove~\eqref{9-4}.  First, by the $C^2$ assumption of the domain $\Og$ (see, e.g.~\cite{Wa}), there exists a constant $r_0\in (0,1)$ such that for each  $\xi=(\xi_x, \xi_y)\in G$,  there exists a closed ball   $B_\xi\subset G^\ast$  of  radius $r_0\in (0, 1)$ that touches the boundary  $\Ga$ at the point $\ga(\xi):=(\xi_x, g(\xi_x))$. 
	Given   a large parameter  $A$, we 	
	define
	\begin{equation}\label{9-6}
	E_\xi:=\Bl\{ \eta \in  B_\xi:\  \  \dist(\eta, T_\xi)\leq \f {A} {n^2}\Br\},\   \  \xi\in G,
	\end{equation}
	where  $T_\xi$ denotes  the tangent plane to $\Ga$ at the point $\ga(\xi)$. Clearly,   $E_\xi\subset G^\ast$ is convex, 
	\begin{equation}\label{9-6-0}
	U(\ga(\xi), \f {c_1} {n}) \subset E_\xi\subset U(\ga(\xi), \f {c_2}n),
	\end{equation}
	where the constants  $c_1, c_2>0$ depend only on $G$ and the parameter $A$. 
	Next, recall that  $S_{\ib,j}^\ast$ is  the compact parallepiped  defined in~\eqref{8-8-1}. By definition, there exists a positive integer $j_0$ depending only on $G$ such that $S_{\ib,j}^\ast\subset G^\ast$ whenever $j_0<j\leq n$.  Furthermore,    according to  Lemma~\ref{metric-lem}, we have that
	\begin{equation}\label{9-8-1}
	\sup_{\xi\in S^\ast_{\ib, j}} \|\xi-\ga(\xi)\|\leq \f {c_3} {n^2},\   \ \text{ for $0\leq j\leq j_0$, }
	\end{equation}
	and    
	\begin{equation}\label{9-7}U(\eta_{\ib,j}, \f {c_4} n) \subset I^\ast_{\ib,j} \subset S_{\ib,j}^\ast\cap G^\ast  \subset U(\eta_{\ib,j}, \f {c_5} n),\    \  \forall (\ib,j) \in\Ld_n^{d+1},\end{equation}
	for some point   $\eta_{\ib,j} \in I_{\ib,j}$,  
	where $c_3, c_4, c_5$ are positive constants depending only on the set $G$. 
	By~\eqref{9-8-1}, we may choose the parameter $A$ in~\eqref{9-6} large enough so that if $0\leq j\leq j_0$ and  $\xi\in I^\ast_{\ib,j}$, then  $[\xi, \ga(\xi)]\subset E_\xi$. 
	Note that if $\xi\in I_{\ib,j}^\ast$ with  $0\leq j\leq j_0$, then by~\eqref{9-7} and~\eqref{9-8-1}, 
	$$\rho_{\Og} (\eta_{\ib,j}, \ga(\xi)) \leq \f {c_6}n,$$
	where  $c_6>0$ is a constant depending only on $G$.    
	Now we define, for   $(\ib,j) \in\Ld_n^{d+1}$, 
	$$ E_{\ib, j} =\begin{cases}
	S_{\ib,j}^\ast, \   \  & \text{ if $j_0<j\leq n$}, \\
	U(\eta_{\ib,j}, \f {c_2+c_6} n), \  \ & \text{ if $0\leq j\leq j_0$}.
	\end{cases}$$
	Thus,  $E_{\ib, j}\subset G^\ast$, and  by~\eqref{9-6-0},~\eqref{9-8-1} and~\eqref{9-7},  we have  that for $\leq j\leq j_0$.
	\begin{equation}\label{9-10}
	\bigcup_{\xi\in I^\ast_{\ib,j}} E_\xi\subset  \bigcup_{\xi\in I^\ast_{\ib,j}} U(\ga(\xi), \f {c_2} n)\subset E_{\ib,j}.
	\end{equation} 
	Thus, setting $e=e_{d+1}$, and using Lemma~\ref{lem-8-1},  we have 
	\begin{align*}
	\sup_{0<s\leq  \f 1n}& \|\tr_{s \vi_{\Og} (e_i,\cdot)e_i }^r (f, G,\cdot)\|_{L^p(G)}^p\leq  C n\int_0^{\f 1n} \int_{G }|\tr_{s\vi_G (e, \xi)e}^r (f, G,\xi)|^p\, d\xi ds\\
	&\leq C\sum_{(\ib,j) \in\Ld_n^{d+1}} \sup_{0<s\leq \f {cj^2} {n^3}}  \int_{I_{\bfi, j}^\ast} |\tr_{se}^r(f, I_{\ib, j}^\ast, \xi)|^p d\xi.
	\end{align*}
	However, by~\eqref{9-6-0},~\eqref{9-10} and Lemma~\ref{lem-9-1:Dec},  this last sum can be estimated above by a constant multiple of 
	\begin{align*}
	& \sum_{(\ib,j) \in\Ld_n^{d+1}}   \int_{I_{\bfi, j}^\ast} \Bl( \f 1{|E_{\ib, j}|}\int_{E_{\ib,j}} |\tr_{(\eta-\xi)/r}^r(f, \Og, \xi)|^q d\eta\Br)^{\f pq}\, d\xi\\
	&\leq C \sum_{(\ib,j) \in\Ld_n^{d+1}}  \int_{I_{\bfi, j}^\ast} \Bl( \f 1{|U(\xi, \f {A_1} n)|}\int_{U(\xi, \f {A_1} n)} |\tr_{(\eta-\xi)/r}^r(f, \Og, \xi)|^q d\eta\Br)^{\f pq}\, d\xi
	\leq C \tau_r (f, \f {A_1} n)_{p,q}^p,
	\end{align*}
	where $A_1:= 2(c_2+c_5+c_6)$. This completes the proof.
\end{proof}

\chapter{Bernstein's inequality}

\section{Bernstein inequality on domains of special type in $\RR^2$}\label{sec:12}

Our main goal in this section is to establish a Bernstein type inequality related to   tangential derivatives on a domain $G\subset \RR^2$ of special type.
Without loss of generality, we may assume that the domain  $G$   takes  the  form,
\begin{equation}\label{10-2} G:=\Bl\{ (x, y):\  \ x\in (-a,a),\   \  g(x)-a< y\leq g(x)\Br\},\end{equation}
for some constant  $a\ge 1$  and function  $g\in C^2[-2a,2a]$ satisfying  that $g(x)\ge 4a$ for all $x\in [-2a, 2a]$, (since otherwise, we  may  rescale the domain $G$ or consider a reflection of $G$). 
Then for   a parameter $\mu\in (0,2]$, 
\begin{align}
G(\mu):&=\{ (x, y):\  \ x\in (-\mu a,\mu a),\   \  g(x)-\mu a<y\leq g(x)\},\     \label{10-8}\\
\p' G (\mu) :&=\{ (x, g(x)):\  \   x\in (-\mu a, \mu a)\},\notag
\end{align} 
and $\p' G=\p' G(1)$. 

For $(x,y)\in G(2)$,  define  
\begin{equation}
\da(x,y):=g(x)-y\   \   \text{and}\  \   \
\vi_n(x,y) :=\sqrt{\da(x,y)} +\f 1n,\   \  n=1,2,\dots.
\end{equation}
According to Lemma~\ref{metric-lem}, we have 
\begin{equation}\label{10-6}
\da(x,y)=g(x)-y \sim \dist (\xi, \p' G),\   \    \  \forall \xi=(x,y)\in G,
\end{equation}

Next, for each   fixed  $x_0\in [-a, a]$, we define an $\ell$-th order  differential operator ${\mathcal{D}}_{x_0}^\ell$  by 
\begin{equation}
{\mathcal{D}}_{x_0}^\ell: =\Bigl (\p_1 +g'(x_0)\p_2\Bigr)^{\ell}= \sum_{i=0}^\ell \binom{\ell}{i} (g'(x_0))^{i} \p_1^{\ell-i} \p_2^i.
\end{equation}
Thus,  up to a constant, ${\mathcal{D}}_{x_0}^\ell$ is the $\ell$-th order directional derivative in the direction of $(1, g'(x_0))$ (i.e., the  tangential direction   to  $\p' G$ at the point $(x_0, g(x_0))$). We also define the operator ${\mathcal{D}}^\ell$ by  
$$ {\mathcal{D}}^\ell f(x,y):=({\mathcal{D}}_x^\ell f)(x,y),\   \ \  \  (x,y)\in G,\  \    f\in C^1(G).$$
Note that  the operators ${\mathcal{D}}^\ell$ and $\p_2$ are commutative, as can be easily seen from the definition.

Finally, let $M$  denote a constant satisfying  that $M>10$ and 
\begin{equation}\label{10-7}
\|g''\|_{L^\infty([-2 a,2 a])}\le M\  \  \text{and}\  \  |g' (0)|\le M.
\end{equation}

In this  section, we shall prove  

\begin{thm}\label{THM:2D BERN} 
	If   $0<p\leq \infty$, $\ld \in (1,2)$ and      $f\in \Pi_n^2$, then 
	\begin{equation}\label{10-7-18}
	\|\vi_n^{i}{\mathcal{D}}^r\partial_2^{i+j} f\|_{L^p(G)}\le  c n^{r+i+2j}\|f\|_{L^p(G_\ast)},\  \  r,i,j=0,1,\dots, 
	\end{equation}
	where $G_\ast =G(\ld)$, and  $c$ is a positive constant depending  only on $M$, $\ld$, $r, i,j$,  and $p$.
\end{thm}

A slightly stronger inequality can be deduced from   Theorem~\ref{THM:2D BERN}:  

\begin{cor}\label{cor-10-2} Let $\ld \in (1,2)$ and $\mu>1$ be two given parameters.  Let  $\vi^\ast_n(x,y):=\min\{\mu \vi_n(x,y), (\ld-1) a\}$ for $(x,y)\in G$. If  $0<p\leq \infty$  and       $f\in\Pi_n^2$,  then 
	\begin{align}
	\Bl\|\vi_n(x,y)^{i} &\max_{|t|\leq \vi^\ast_n(x,y)} \Bigl|  {\mathcal{D}}_{x+t}^{r}\partial_2^{i+j}f(x,y)\Br|\Br\|_{L^p(G; dxdy)}\\
	&\le  c\mu^r  n^{r+2j+i}\|f\|_{L^p(G_\ast)},\  \  r,i,j=0,1,\dots,\notag
	\end{align}
	where $G_\ast=G_\ast(\ld)$. 
\end{cor}

For the moment, we  take Theorem~\ref{THM:2D BERN} for granted and show how  it implies  Corollary~\ref{cor-10-2}. 

\begin{proof}[Proof of  Corollary~\ref{cor-10-2} (assuming  
	Theorem~\ref{THM:2D BERN})]
	For any fixed   $x_0, t\in [-a,a]$, we have 
	\begin{equation}\label{10-8-eq} {\mathcal{D}}_{x_0+t} ={\mathcal{D}}_{x_0}+\Bl( g'(x_0+t)-g'(x_0)\Br) \p_2= {\mathcal{D}}_{x_0} +|t|O(x_0,t)  \p_2,\end{equation}
	where $\sup_{x, t\in [-a,a]} |O(x,t)|\leq M$.
	Since  the operators $\p_2$ and ${\mathcal{D}}_{x_0}^i$ are commutative,   we obtain from~\eqref{10-8-eq} that if $(x_0,y_0)\in G$  and $|t|\leq \vi^\ast_n(x_0,y_0)$, then 
	\begin{align*}
	\Bl|\vi_n(x_0,y_0)^i {\mathcal{D}}_{x_0+t} ^r \p_2 ^{i+j}  f(x_0,y_0)\Br| &\leq C \mu^r \max_{r_1+r_2=r} \vi_n(x_0,y_0)^{i+r_2} |{\mathcal{D}}_{x_0}^{r_1} \p_2^{i+j+r_2} f(x_0,y_0)|.
	\end{align*}
	It then follows from Theorem~\ref{THM:2D BERN} that 
	\begin{align*}
	\Bl\|&\vi_n(x,y)^{i} \max_{|t|\leq \vi^\ast_n(x,y)} \Bigl| {\mathcal{D}}_{x+t}^{r}\partial_2^{i+j}f(x,y)\Br|\Br\|_{L^p(G; dxdy)}\\
	&\leq C \mu^r \max_{r_1+r_2=r} \Bl\| \vi_n^{i+r_2} {\mathcal{D}}^{r_1} \p_2 ^{i+r_2+j} f\|_{L^p(G)}\leq C\mu^r n^{i+r+2j}\|f\|_{L^p(G_\ast)}.
	\end{align*}
	This completes the proof of Corollary~\ref{cor-10-2}.
	\end{proof} 

The rest of this section is devoted to the proof of Theorem~\ref{THM:2D BERN}. 
We start with  the following  simple  lemma,  which   is a direct  consequence of the univariate Bernstein inequality for algebraic polynomials:
\begin{lem}\label{lem:univ BM} 
	Assume that $f\in C(G)$ satisfies that  $f(x, \cdot)\in\Pi_n^1$  for each fixed  $x\in [-a,a]$. Then for $0<p\leq \infty$ and  $\ld\in (1,2)$, 
	\[
	\|\vi_n^i\partial_2^{i+j} f\|_{L^p(G)}\le  c n^{i+2j}\|f\|_{L^p(G_*)},\   \  i,j=0,1,\dots, 
	\]
	where $G_\ast =\overline{G_\ast(\ld)}$, and 
	$c$ is a  positive constant depending only on $M$, $\mu$,  $i+j$ and $p$. 
\end{lem}
\begin{proof}
	If $P$ is an algebraic  polynomial of one variable of degree $\le n$, then by the   univariate Bernstein inequality  (\cite[Lemma 2.2]{DJL}  or  \cite[p. 265]{De-Lo}),  we have that for any $b>0$ and $\al>1$, 
	\begin{equation}\label{markov-bern}
	\Bl	\|(\sqrt{b^{-1}t} +n^{-1})^iP^{(i+j)}(t)\Br\|_{L^p([0,b], dt)} \le  C_\al  n^{i+2j}  b^{-(i+j)}
	\|P\|_{L^p([0,\al b])}.
	\end{equation}
	Lemma~\ref{lem:univ BM} then follows  by integration over vertical line segments. 	
\end{proof}

To illustrate the idea, we shall  first prove  Theorem~\ref{THM:2D BERN} for  the  case of $r=1$ in Section~\ref{subsection:10-1},  which  is relatively simpler, but already  contains the  crucial  ideas required for the proof for $r>1$.   The proof of Theorem~\ref{THM:2D BERN} for  $r>1$ is technically more involved and will be  given in Section~\ref{subsection:10-2}.

\section{Proof of Theorem~\ref{THM:2D BERN} for $r=1$} \label{subsection:10-1}
We start with  some necessary notations.  Write  $\ld:=1+\f 1{M_\ld}$. Without loss of generality, we may assume that $M=M_\ld$, where  $M >10$ is a constant satisfying~\eqref{10-7}.  We shall keep this assumption for the rest of this section. 
Given a parameter  $A>0$, we   define 
$$ Q_A(z,t):=g(z)+g'(z) t -\f A2t^2,  \   \  z\in [-\ld a, \ld a],\  \ t\in\RR. $$ 
By  Taylor's theorem,  if $z, z+t\in [-\ld a, \ld a]$, then 
\begin{equation}\label{eqn:taylor} 
g(z+t)-Q_A(z,t) =\int_z^{z+t} [ A +g''(u)] (z+t-u)\, du.\end{equation}
Of crucial importance in our proof is the fact that   every point $(x,y)\in G$    can be represented uniquely in the form $(z+t, Q_A(z,t))$.
To be more precise,   set 
$$ E_A:=\Bl\{(z,t)\in\RR^2:\   \  z, z+t\in [-\lambda a,\lambda a],\   \   |t|\leq a_0:=\sqrt{\frac{2a\lambda}{A+M}}\Br \},$$
and 
define the  mapping $\Phi_A: E_A\to \RR^2$ by
\begin{equation}\label{1-2:Phi}
\Phi_A(z, t)=(x,y): =\big(z+t, Q_A(z,t)\big),\   \     (z,t) \in E_A.
\end{equation}
Also, we  denote by $\Phi_A^+$ and $\Phi_A^{-}$ the restrictions of $\Phi_A$ on the sets $E_A^+$ and $E_A^{-}$ respectively, where 
$E_A^{+} =\{(z,t)\in E_A:\  \ t\ge 0\}$ and $E_A^{-} =\{(z,t)\in E_A:\ \ t\leq 0\}$.

We collect some   useful facts related to   the mapping $\Phi_A: E_A\to\RR^2$ in the following lemma,  which  will play an important role in the proof of Theorem~\ref{THM:2D BERN}:  
\begin{lem}\label{lem-1-1:Berns} 	 Given a parameter $A\ge \overline{A}:=10M^2+M+1$,  the   following statements hold with $G_\ast=\overline{G_\ast(\ld)}$: 
	\begin{enumerate}[\rm (i)]
		\item   $
		G\subset \Phi_A(E_A^{+}) \subset \Phi_A (E_A) \subset G_\ast$. Moreover,  both the  mappings $\Phi_A^+: E_A^{+}\to G_\ast$ and $\Phi_A^{-} : E_A^{-} \to G_\ast$ are   injective.
		\item If  $(z,t)\in E_A$, then    
		\begin{equation}\label{1-3:jacobian}
		\bigl| \det  \left(J_{\Phi_A}(z,t)\right) \bigr|\equiv 
		\frac{\p(x,y)}{\p(z,t)} = 
		(A+g''(z)) |t|.
		\end{equation}
		\item Every  $(x,y)\in G$ can be represented uniquely in the form  
		$(x,y)=\Phi_A(z,t)$ with  $(z,t)\in E^{+}_A$  and $0\leq t\leq a_1$, where  $$ a_1:=a_1(M,a, A):=\sqrt{\frac{2a}{A-M}}<a_0. $$

		\item 	Let    $u_A$  be the  function on the set  $\Phi_A(E_A^{+})$ such that for  every  $(x,y)=\Phi_A(z,t)$ with  $(z, t)\in E_A^{+}$,
		\begin{equation}\label{10-12}
		u_A (x,y) =g'(z+t) -g'(z) +A t=:w_A(z,t).
		\end{equation}
		Then  for every $(x,y)\in \Phi_A (E_A^{+})$, 
		\begin{equation}\label{1-6:bern}
		\frac{(A-M)\sqrt{2}}{\sqrt{A+M}} \, \sqrt{\delta(x,y)} \le u_A(x,y)\le \frac{(A+M)\sqrt{2}}{\sqrt{A-M}} \, \sqrt{\delta(x,y)}.
		\end{equation}
	\end{enumerate}
\end{lem}

Lemma~\ref{lem-1-1:Berns} with  several different parameters $A$  is  required in our proof of Theorem~\ref{THM:2D BERN} for $r>1$.   However,    for the proof in  the case of $r=1$,  it will be enough to have this lemma  for  one  fixed parameter $A\ge \bar{A}$ only. 
\begin{proof} 
	
	(i) First, the relation $G\subset \Phi_A(E_A^{+})$ follows directly from the statement (iii), which will proven later. Next, we prove  that $\Phi_A (E_A) \subset G_\ast.$ Indeed, if  $(x,y) =\Phi_A(z,t) $ with $(z, t) \in E_A$, then  $x=z+t\in [-a\lambda,a\lambda]$ and by~\eqref{eqn:taylor},  
	\begin{equation} \label{1-4:taylor} 
	g(x)-y=g(z+t)-Q_A(z,t) =\int_z^{z+t} (g''(u)+A) (z+t-u)\, du.
	\end{equation} 
	Since $|t|\leq a_0$, this implies  that 
	$0\leq g(x)-y\leq (M+A) a_0^2/2 = a\lambda$, which shows that  $(x, y)\in G_\ast$.
	
	Finally, we show that both of the mappings $\Phi_A^{+}$ and $\Phi_A^{-}$ are injective. 	  Assume that $\Phi_A(z_1, t_1)=\Phi_A(z_2, t_2)$ for some  $(z_1, t_1), (z_2, t_2) \in E_A$ with $t_1t_2\ge 0$ and $t_2\ge t_1$.  Then $z_1+t_1=z_2+t_2=:\bar{x}$, $Q_A(z_1,t_1)=Q_A(z_2,t_2)$, and hence, using~\eqref{1-4:taylor}, we obtain that for $i=1,2$,
	\begin{align*}
	g(\bar{x}) -Q_A(z_i,t_i) =g(z_i+t_i)-Q_A(z_i,t_i) =\int_0^{t_i} [ g''(\bar{x}-v) +A] v \, dv,
	\end{align*}
	which implies 
	\begin{equation}\label{1-5:inj}
	\int_{t_1}^{t_2} \Bl( g''(\bar{x} -v)+A\Br) v\, dv =0.
	\end{equation}
	Since $g''(\bar{x} -v)+A\ge A-M>0$ and $v$ doesn't change sign on the interval $[t_1, t_2]$,~\eqref{1-5:inj} implies that  $t_1=t_2$, which  in turn implies that $z_1=z_2$. Thus, both  $\Phi_A^+$ and $\Phi_A^{-}$ are  injective.

	(ii) Equation~\eqref{1-3:jacobian}  can be verified by straightforward calculations. 
	
	(iii)  It suffices to show that the existence of the  representation $(x,y)=\Phi_A (z,t)\in G$, as the uniqueness follows directly  from the already proven fact that the mapping  $\Phi_A^{+}$ is injective.  
	Note first that  $-a\leq x\leq a$, and  $g(x)-a\leq y\leq g(x)$. 
	Since   $a\ge 1$ and $A\ge \bar {A} >2M^2+M$, it follows that   for every  $t\in [0, a_1]$, 
	$-a\lambda<-a-a_1\leq x-t \leq a,$
	Thus, 
	using~\eqref{eqn:taylor}, we obtain that for any $t\in [0, a_1]$, 
	\[
	h(t):= g(x) -Q_A(x-t,t)=\int_{x-t} ^x (g''(u) +A) (x-u)\, du.
	\]
	Clearly, $h$ is a continuous function on $[0, a_1]$ satisfying that $h(0)=0$ and 
	\[
	h(a_1) \ge (A-M) \int_0^{a_1} v\, dv =  a.
	\]
	Since $0\leq g(x)-y\leq a$, it follows  by the Intermediate Value theorem  that there exists $t_0\in [0, a_1]$ such that 
	$h(t_0) =g(x)-y$, which in turn implies that  $y=Q_A(x-t_0,t_0)$. This shows that  $(x,y) =\Phi_A(x-t_0, t_0)$. Since $A\ge {\overline{A}}>M+\frac{2M}{\lambda-1}=2M^2+M$, we have 	$a_1<a_0$.

	(iv) Let $(x,y)=\Phi_A (z,t)$ with $(z,t)\in E_A^{+}$. 
	By~\eqref{10-12} and the mean value theorem, we have 
	\begin{equation}\label{10-16}
	(A-M)t\le u_A(x,y)\le (A+M)t.
	\end{equation}
	On the other hand, however,  using~\eqref{eqn:taylor}, we have that 
	\begin{equation*}
	g(x) -y =g(z+t) -Q_A(z,t) =\int_z^{z+t} ( g''(u) +A) (z+t-u) \, du,
	\end{equation*}
	which in particular  implies 
	\begin{equation}\label{10-17}
	(A-M)\tfrac{t^2}{2}\le g(x)-y\le (A+M)\tfrac{t^2}{2}.
	\end{equation}
	Finally, combining~\eqref{10-16} with~\eqref{10-17}, we deduce the desired estimates 
	\eqref{1-6:bern}.
\end{proof}

We are now in a position to prove  Theorem~\ref{THM:2D BERN} for $r=1$.

\begin{proof}[Proof of  Theorem~\ref{THM:2D BERN} for $r=1$]
	If   $f\in\Pi_n^2$ and $x\in [-2a, 2a]$, then  ${\mathcal{D}}_x f(x,\cdot)$ is a polynomial  of degree at most $n$ of a single variable.  
	Since  the operators $\mathcal{D}$ and $\p_2$ are commutative,  it follows from Lemma~\ref{lem:univ BM} that 	
	$$\|\vi_n^{i}\mathcal{D}\partial_2^{i+j} f\|_{L^p(G)}=\|\vi_n^{i}\partial_2^{i+j} \mathcal{D} f\|_{L^p(G)}\leq C_\al n^{i+2j} \|\mathcal{D} f\|_{L^p(G_\ast(\sqrt{\ld}))}.$$
	Thus,  it is sufficient  to show that for any $\ld\in (1,2)$, 	\begin{equation}\label{10-19:es}
	\|\mathcal{D}f\|_{L^p(G)} \le C n \|f\|_{L^p(G_\ast)},\  \  \forall f\in\Pi_n^2, 
	\end{equation} 
	where $G_\ast=G_\ast(\ld)$. Recall that we may assume, without loss of generality,   $M=\f 1{\ld-1}$ satisfies~\eqref{10-7}.  
	Also, in our proof below, we shall always  assume that $p<\infty$.  The  case $p=\infty$ can be treated similarly, and in fact, is simpler.

	To prove~\eqref{10-19:es}, we set 
	$$	I:=\|\mathcal{D}f\|_{L^p(G)}^p=	\iint_{G} \Bl|{\mathcal{D}}_x f(x,y)|^p\, dxdy,$$
	and let  $A\ge \bar{A}$ be a fixed parameter. Using   Lemma~\ref{lem-1-1:Berns} (i) and (ii), and  performing the  change of variables $x=z+t$ and $y=Q_A(z, t)$,  we   obtain 
	\begin{align*}
	I& =\iint_{(\Phi_A^{+})^{-1} (G)} \Bl| {\mathcal{D}}_{z+t} f(\Phi_A(z,t))\Br|^p (A +g''(z)) t \, dzdt.
	\end{align*}
	A straightforward calculation shows that  for each $(z,t)\in E_A$,
	\begin{equation}\label{eqn:part_der} 
	{\mathcal{D}}_{z+t} f( \Phi_A (z,t))=\f d{dt} \Bl[ f( \Phi_A(z,t))\Br] +w_A(z,t)\p_2 f (\Phi_A(z,t)),
	\end{equation}
	where $w_A(z,t)= g'(z+t)-g'(z) +At$. 
	Thus,     $I\leq 2^p  (I_1+I_2)$, where 
	\begin{align}
	I_1&:=\iint_{(\Phi_A^+)^{-1} (G)} \Bl| \f d {dt} \Bl[ f(\Phi_A(z,t))\Br]\Br|^p (A +g''(z)) t\, dzdt,\label{I1:def}\\
	I_2&:=  \iint_{(\Phi_A^+)^{-1} (G)} |w_A(z,t)|^p |\p_2 f(\Phi_A(z,t))|^p (A +g''(z)) t \, dzdt.\notag
	\end{align}
	
	For  the  double integral $I_2$,   performing the change of variables $(x,y)=\Phi_A^{+} (z,t)$, and 
	using    Lemma~\ref{lem-1-1:Berns} (i),  (ii) and (iv), we obtain   
	\begin{align}\label{1-7:Bern}
	I_2& =\iint_G |u_A(x,y)|^p |\p_2 f(x,y)|^p\, dxdy\leq C  \iint_G  |\sqrt{\da(x,y)}\p_2 f(x,y)|^p\, dxdy\notag\\
	&=\int_{-a}^a \int_{g(x)-a}^{g(x)} |\sqrt{\da(x,y)}\p_2 f(x,y)|^p\, dydx.\notag
	\end{align}
	It then follows by 
	the Markov-Bernstein-type  inequality~\eqref{markov-bern} that 
	\begin{align*}
	I_2\leq 	C n^p \int_{-a}^a  \int_{g(x)-\ld a}^{g(x)} |f(x,y)|^p\, dydx\leq C n^p \|f\|_{L^p(G_\ast)}^p.
	\end{align*}
	
	Thus, it remains to prove  that \begin{equation}\label{10-23-0}
	I_1	\leq C_p n^p \|f\|_{L^p(G_\ast)}^p.
	\end{equation}
	Since $A\ge  \bar{A}$, we have $a_0 <\f {a} {2M}$. Thus, 
	using  Lemma~\ref{lem-1-1:Berns} (iii),  we obtain  that in the $zt$-plane, 
	\begin{equation}\label{10-24-0}
	(\Phi_{A}^{+})^{-1} (G) \subset [-a-a_1,a]\times [0, a_1]\subset [-a-a_0, a] \times [-a_0, a_0] \subset E_{A}.
	\end{equation}	
	It follows that 
	\begin{align*}
	I_1		&\leq C(A) \int_{- a-a_1 }^{a} \Bl[ \int_{0}^{a_1}\Bl| \f d{dt} \Bl[ f(z+t, Q_A(z, t))\Br] \Br|^p |t|\, dt\Br]\, dz.
	\end{align*}
	Using Bernstein's inequality with doubling weights for algebraic polynomials (see \cite[Theorem 7.3]{MT2} for $p\ge 1$, and \cite[Theorem 3.1]{Er} for $0<p<1$),  we obtain 
	\begin{align*}	I_1	&\leq C(M,p) n^p  \int_{-a-a_1 }^{a} \Bl[ \int_{-a_0}^{a_0}\Bl| f(z+t, Q_A(z, t)) \Br|^p |t|\, dt\Br]\, dz\\
	&	\leq C(M,p)  n^p \iint_{E_A} |f(z+t, Q_A(z,t))|^p|t|\, dzdt. \end{align*}
	Splitting this last double integral into two parts $\iint_{E_A^{+}}+\iint_{E_A^{-}}$, and applying  the change of variables $(x,y)=\Phi_A(z,t)$ to each of them  separately, we obtain that 
	\begin{align*}
	I_1&	\leq   C(M,p)n^p \iint_{E_A^{+}\cup E_A^{-}} |f(\Phi_A(z,t))|^p(A +g''(z)) |t|\, dzdt \\
	&\leq C(M,p) n^p \Bl[ \iint_{\Phi(E_A^{+})}  |f(x,y)|^p dxdy+\iint_{\Phi(E_A^{-})}  |f(x,y)|^p dxdy \Br]\\ &\leq C(M,p) n^p \|f\|_{L^p(G_\ast)}^p,	 
	\end{align*}
	where we used Lemma~\ref{lem-1-1:Berns}~ (i), (ii) in the second step,  and  the fact that $\Phi_A (E_A) \subset G_\ast$ (i.e.,  Lemma~\ref{lem-1-1:Berns}~ (i)) in the third step. 
\end{proof} 

\section{Proof of Theorem~\ref{THM:2D BERN} for $r>1$}\label{subsection:10-2}  The proof for $r>1$   uses induction and   Lemma~\ref{lem-1-1:Berns} for   several distinct  parameters $A$. Recall that $\Phi_A (z,t)=(z+t, Q_A(z,t))$ for $(z,t)\in E_A$, where $$ Q_A(z,t):=g(z)+g'(z) t -\f A2t^2,  \   \  z\in [-\ld a, \ld a],\  \ t\in\RR. $$ 

The  proof of Theorem~\ref{THM:2D BERN} for $r>1$ relies on the following two lemmas:
\begin{lem}\label{lem:bern repres} 
	Let  $(x_0,y_0)\in G$ and let $(z_0, t_0) \in E_A^{+} $ be such that $(x_0, y_0)
	=\Phi_A (z_0, t_0)$, where $A\ge \bar{A}$ is an arbitrarily given parameter. 
	Let $f\in C^\infty (\RR^2)$ and  define  $F(t):=f(\Phi_A(z_0, t))$ for  $t\in\R$. Then for any $r\in\NN$,
	\begin{align}\label{eqn:f_s1_s2_s3}
	\frac{d^rF(t_0)}{dt^r}
	= S_1(A,x_0,y_0)+S_2(A,x_0,y_0)+S_3(A,x_0,y_0),
	\end{align}
	where
	\begin{align}
	S_1(A,x_0,y_0) &:=  \sum_{j=0}^r \binom{r} j (-u_A(x_0, y_0))^j ({\mathcal{D}}^{r-j} \partial_2^{j} f)(x_0,y_0),\label{s1}\\
	S_2(A,x_0,y_0) &:=  \sum_{\sub{i+j+2k=r\\
			k\ge1}} \frac{r!}{i!j!k!2^k}(-u_A(x_0, y_0))^j(-A)^k ({\mathcal{D}}^{i} \partial_2^{j+k} f) (x_0,y_0),\label{s2} \\
	S_3(A,x_0,y_0) &:=  \sum_{2(i+j)=r} \frac{r!}{i!j!2^{r/2}} (-A)^j ({\mathcal{D}}^{i} \partial_2^j f) (x_0,y_0).\label{s3}
	\end{align}
\end{lem}

\begin{lem}\label{lem:A_0 A_r} Given a positive integer $r$, there exist  constants $A_0, A_1, \dots, A_r$ depending only on $M$, $\ld$  and $r$ such that  $\bar{A}\leq A_0<\dots<A_r$ and  for any $(x_0,y_0)\in G$, and  $f\in C^r(\RR^2)$,
	\begin{equation}\label{eqn:a_v_bound}
	\max_{0\le v\le r}\left|(\delta(x_0,y_0))^{v/2}({\mathcal{D}}^{r-v}\partial_2^vf)(x_0,y_0)\right| \le c \max_{0\le v\le r} \left|S_1(A_v,x_0,y_0)\right|,
	\end{equation}
	where the constant $c$ depends only on $M$, $\ld$, $r$ and  $a$, and $	S_1(A,x_0,y_0)$ is given in~\eqref{s1}. 
\end{lem}

For the proof of Lemma~\ref{lem:bern repres}, 
we need the following additional  lemma for  computing  higher order derivatives of certain composite functions: 
\begin{lem}\label{lem:mult faa di bruno}
	Let $f\in C^\infty (\R^2)$ and  define  $F(t):= f(z+t,Q(t))$ for  $t\in\R$, where $z\in\RR$ and $Q$  is a univariate  polynomial of degree at most $2$.  Then for any $r\in\NN$,
	\begin{align*}
	\frac{d^rF}{dt^r} =  \sum_{i+j+2k=r} \frac{r!}{i!j!k!2^k}(Q')^j(Q'')^k (\partial_1^{i} \partial_2^{j+k}  f) + \sum_{2(i+j)=r} \frac{r!}{i!j!2^{r/2}} (Q'')^j (\partial_1^{i} \partial_2^j  f), 
	\end{align*}
	where the summations are taken over non-negative integers $i$, $j$, $k$, and the partial derivatives are evaluated at $(z+t,Q(t))$.
\end{lem}
\begin{proof}
	This is a special  case of the multivariate F\'aa di Bruno formula from~\cite[p.~ 505]{Co-Sa}.
\end{proof}

The proofs of Lemma~\ref{lem:bern repres} and Lemma~\ref{lem:A_0 A_r} are given as follows:

\begin{proof}[Proof of Lemma~\ref{lem:bern repres}]
	Define  $\wt f(u,v):=f(u,g'(x_0)u+v),\   \   \   (u,v)\in\RR^2,$
	and 	rewrite the function $F$ as 
	$$
	F(t)=f\bl(z_0+t, Q_A(z_0, t)\br) = \wt f\bl(z_0+t, Q(t)\br),$$
	where 
	$$ Q(t):= Q_A (z_0, t)-g'(x_0)(z_0+t)= g(z_0) +g'(z_0) t -\f A 2 t^2 -g'(z_0+t_0) (z_0+t).$$
	Clearly, for any $(u,v)\in\RR^2$, 
	$$\partial^i_1 \wt f(u,v) = {\mathcal{D}}^i_{x_0}f\bl(u, g'(x_0) u+v\br)\   \  \text{and}\   \  \p_2\wt f(u,v)=\partial_2 f\bl(u, g'(x_0) u +v\br).$$
	Thus, to complete the proof of~\eqref{eqn:f_s1_s2_s3}, we 
	just need to apply Lemma~\ref{lem:mult faa di bruno} to the function $F(t)$,  observing   that 
	$$ Q'(t_0) = g'(z_0) -A t_0 -g'(z_0+t_0) =-w_A(z_0, t_0)=-u_A(x_0, y_0),$$
	and 
	$Q''(t)=-A$.  
\end{proof}

\begin{proof}[Proof of  Lemma~\ref{lem:A_0 A_r}]  Without loss of generality, we may assume that $\da(x_0, y_0) >0$, since otherwise~\eqref{eqn:a_v_bound} holds trivially.	
	First, using Lemma~\ref{lem-1-1:Berns}~(iv), we can  find a strictly increasing sequence of  constants $A_0, A_1,\dots, A_r\ge \bar{A}$ depending only on $M$, $\ld$  and $r$  such that 
	for $i=0,1,\dots, r-1$, 
	$$ 0<c_{M} \leq  \f{u_{A_i} (x_0,y_0)}{\sqrt{\da(x_0,y_0)}}<2\f{ u_{A_i} (x_0,y_0)}{\sqrt{\da(x_0,y_0)}}<\f{u_{A_{i+1}}(x_0,y_0)}{\sqrt{\da(x_0,y_0)}} \leq C_{M}.$$	
	Next, setting 
	\begin{align*}
	B_j = -\f { u_{A_j}(x_0, y_0)} {\sqrt{\da(x_0, y_0)}}\   \ \text{and}\   \  u_j:=  \binom{r} j\da(x_0, y_0)^{\f j2} {\mathcal{D}}^{r-j} \p_2 ^j f (x_0, y_0)
	\end{align*} 
	for $j=0,1,\dots, r$, and using~\eqref{s1}, 
	we obtain that 
	$S_1 ( A_i, x_0, y_0) =\sum_{j=0}^r (B_i)^j u_j$ for  $ i=0,1,\dots, r$; that is, 
	\begin{align}\label{10-26-0}
	\mathbf{S} =\mathcal{M} \mathbf{u},
	\end{align}
	where 
	\begin{align*}
	\mathbf{S}&= \Bl (S_1 ( A_0, x_0, y_0), \dots, S_1 ( A_r, x_0, y_0)\Br)^t,\   \    \  
	\mathbf{u}= (u_0, u_1,\dots, u_r)^t,
	\end{align*} 
	and $\mathcal{M}$ is  the  $(r+1) \times (r+1)$ Vandermonde matrix 	with $(i,j)$-entry $\mathcal{M}_{i,j}:=(B_i)^j$. 
	Since 
	$$ 0< c_{M} \leq \min_{0\leq i\neq j \leq r} |B_j-B_i|\leq C_{M},$$
	the stated estimate~\eqref{eqn:a_v_bound} follows from~\eqref{10-26-0} and  Cramer's rule.
\end{proof}

Now we are in a position to prove   Theorem~\ref{THM:2D BERN}
for $r>1$. 
\begin{proof}[Proof of Theorem~\ref{THM:2D BERN} for $r>1$.]
	We use induction on $r\in\NN$. 
	The case  $r=0$ is given in Lemma~\ref{lem:univ BM}, whereas the case $r=1$ has already been proven in the last section.  Now suppose the statement  has been  established for all the  derivatives ${\mathcal{D}}^\ell \p_2^{i+j}$ with  $\ell=0,1,\dots, r-1$ and  $i,j=0,1,\dots.$  We then aim to prove that for any $f\in \Pi_n^2$ and $\ld\in (1,2)$,
	\begin{equation*}\label{10-7-18b}
	\|\vi_n^{i}{\mathcal{D}}^r\partial_2^{i+j} f\|_{L^p(G)}\le  c n^{r+i+2j}\|f\|_{L^p(G_\ast(\ld))},\  \  i,j=0,1,\dots.
	\end{equation*}
	Since the operators ${\mathcal{D}}^i$ and $\p_2^j$ are commutative,  as in   the case of $r=1$,  it is sufficient to show  that for every $f\in\Pi_n^2$ and any $\ld\in (1,2)$,
	\begin{align}\label{10-28}
	\|{\mathcal{D}}^r f\|_{L^p(G)}\le  c n^{r}\|f\|_{L^p(G_\ast(\ld))}.
	\end{align}
	Here and throughout the proof, the  constant $c$ depends only  on $p$, $a$, $\lambda$, $M$ and $r$. 
	
	To show~\eqref{10-28}, we start with the case of $p=\infty$, which is simpler. By  Lemma~\ref{lem:A_0 A_r}, it is enough to show that 
	\begin{equation}\label{10-32-0}\max_{0\leq v\leq r} \max_{(x_0, y_0) \in G} |S_1(A_v,x_0,y_0)|\le c n^r\|f\|_{L^\infty(G_*)}. \end{equation}
	
	Fix temporarily $0\leq v\leq r$ and  $(x_0,y_0)\in G$. By Lemma~\ref{lem-1-1:Berns}~(iii),  there exists  $(z_{v,0},t_{v,0})\equiv (z_0, t_0)\in E_{A_v}^+$  such that  $0\leq t_0\leq a_1$ and $\Phi_{A_v}(z_{ 0},t_{0})=(x_0,y_0)$.  Applying  Lemma~\ref{lem:bern repres} to the function  $F(t):=f(\Phi_{A_v}(z_0,t))$, and 
	recalling (Lemma~\ref{lem-1-1:Berns}~(iv))
	\begin{align}\label{eqn:bern delta bound}
	|u_{A_v}(x_0,y_0)|\le c \sqrt{\delta(x_0,y_0)},
	\end{align}
	we  obtain
	\begin{align*}
	|S_1(A_v, x_0, y_0)|&\leq |F^{(r)} (t_0)| + |S_2(A_v, x_0, y_0)|+|S_3(A_v, x_0, y_0)|\\
	&\leq |F^{(r)} (t_0)| +C\max_{ i+j+2k =r, k\ge 1} (\da(x_0, y_0))^{j/2} |{\mathcal{D}}^i \p_2^{j+k} f (x_0, y_0)|\\
	& + C \max_{2(i+j)=r} |{\mathcal{D}}^i \p_2^j f(x_0, y_0)|.
	\end{align*} 
	It then follows by induction hypothesis that 
	\begin{align*}
	|S_1(A_v, x_0, y_0)|\leq |F^{(r)} (t_0)| + C n^r\|f\|_{L^\infty(G_\ast)}.
	\end{align*}
	On the other hand, since the function $F(t):=f(\Phi_{A_v}(z_0,t))$ 
	is an algebraic polynomial of degree $\le 2n$,   by  the univariate Bernstein  inequality, we obtain
	\[
	|F^{(r)}(t_0)|\le \|F^{(r)}\|_{C([-a_1,a_1])}\le c(a_1,a_0) n^r \|F\|_{C([-a_0,a_0])}\le c n^r \|f\|_{L^\infty(G_*)},
	\]
	where the last step uses Lemma~\ref{lem-1-1:Berns}  (i) and the fact that $a_0\leq \f a {2M}$.
	Summarizing the above, we derive~\eqref{10-32-0} and hence complete the proof of  
	\eqref{10-28} for $p=\infty$.

	Next, we  prove~\eqref{10-28} for  $0< p<\infty$. By Lemma~\ref{lem:A_0 A_r}, it is enough to show that for each  $0\le v\le r$,
	\begin{equation}\label{eqn:s1_est}
	\iint_G |S_1 (A_v,x,y)|^p\, dxdy \le c n^{rp} \|f\|^p_{L^p(G_*)}.
	\end{equation}
	Fix $0\leq v\leq r$ and set 
	\[
	I_\ell :=\iint_G |S_\ell (A_v,x,y)|^p\, dxdy, \quad \ell=1,2,3.
	\]	
	Using~\eqref{s2},~\eqref{s3} and the induction hypothesis, we have 
	\begin{align}
	I_2& \leq C \max_{\sub{i+j+2k =r\\
			\  k\ge 1}} \iint_G \Bl|(\da(x,y))^{j/2} |{\mathcal{D}}^i \p_2^{j+k} f(x,y)\br|^p\, dxdy \leq C n^{rp} \|f\|_{L^p(G_\ast)}^p,\label{i2}\\
	I_3&\leq C \max_{2(i+j)=r} \iint_G |{\mathcal{D}}^i \p_2^j f(x,y)|^p\, dxdy \leq C  n^{rp} \|f\|_{L^p(G_\ast)}^p.\label{i3}
	\end{align}
	On the other hand,  performing the change of variables $x=z+t$ and $y=Q_{A_v}(z,t)$ and using Lemma~\ref{lem-1-1:Berns} (i), (ii),
	we obtain 
	\[
	I_\ell:= \iint_{(\Phi_{A_v}^+)^{-1} (G)} |S_\ell(A_v,z+t, Q_{A_v}(z,t))|^p (A_v +g''(z)) t \, dzdt, \quad \ell=1,2,3.
	\]
	Thus, setting 
	\[
	I: = \iint_{(\Phi_{A_v}^+)^{-1} (G)} \Bl| \f {d^r} {dt^r} \Bl[ f(z+t, Q_{A_v}(z,t))\Br]\Br|^p (A_v +g''(z)) t\, dzdt,
	\]
	and using~\eqref{i2},~\eqref{i3} and Lemma~\eqref{lem:bern repres},
	we obtain 
	$$I_1\le c_p(I+I_2+I_3)\leq c_pI + c_p n^{rp} \|f\|_{L^p(G_\ast)}^p.$$
	
	Thus, for the proof~\eqref{eqn:s1_est},  we reduce to  showing  that \begin{align}
	I: &= \iint_{(\Phi_{A_v}^+)^{-1} (G)} \Bl| \f {d^r} {dt^r} \Bl[ f(z+t, Q_{A_v}(z,t))\Br]\Br|^p (A_v +g''(z)) t\, dzdt\notag\\
	&\leq C   n^{rp} \|f\|_{L^p(G_\ast)}^p.\label{10-37}
	\end{align}
	The proof is very similar to that of~\eqref{10-23-0}  given in the last subsection for $r=1$. 
	Indeed,  using~\eqref{10-24-0}, Lemma~\ref{lem-1-1:Berns} and the   univariate weighted  Bernstein inequality, we have 
	\begin{align*}
	I	&\leq c \int_{-a-a_1 }^{a} \Bl[ \int_{0}^{a_1}\Bl| \f {d^r} {dt^r} \Bl[ f(z+t, Q_{A_v}(z,t))\Br] \Br|^p |t|\, dt\Br]\, dz\\
	& \leq c n^{rp} \int_{-a-a_1 }^{a} \Bl[ \int_{-a_0}^{a_0}\Bl| f(\Phi_{A_v}(z,t)) \Br|^p |t|\, dt\Br]\, dz	\leq c  n^{rp} \iint_{E_{A_v}} |f(\Phi_{A_v}(z,t))|^p|t|\, dzdt\\
	&\leq c n^{rp} \Bl[ \iint_{\Phi(E_{A_v}^{+})}  |f(x,y)|^p dxdy+\iint_{\Phi(E_{A_v}^{-})}  |f(x,y)|^p dxdy \Br]\leq c n^{rp} \|f\|_{L^p(G_\ast)}^p.	 
	\end{align*}
	This proves~\eqref{10-37} and hence completes the proof of the theorem.
\end{proof}

\section{Bernstein type inequality: the higher-dimensional case}
\label{sec:13}
In this section, we shall show how  to extend the Bernstein inequality (i.e., Theorem~\ref{THM:2D BERN}) to higher-dimensional domains. 
We start with a  domain $G\subset \RR^{d+1}$  of special type.
Without loss of generality, we may assume that   $G$   takes  the  form,
\begin{equation}\label{11-2-2} G:=\Bl\{ (x, y):\  \ x\in (-a,a)^d,\   \  g(x)-La< y\leq g(x)\Br\},\end{equation}
for some constant  $a\ge 1$  and function  $g\in C^2([-2a,2a]^d)$ satisfying  that $g(x)\ge 4a$ for all $x\in [-2a, 2a]^d$. 
For $(x,y)\in G_\ast(2)$,  set 
\begin{equation*}
\da(x,y):=g(x)-y\   \   \text{and}\  \   \
\vi_n(x,y) :=\sqrt{\da(x,y)} +\f 1n,\   \  n=1,2,\dots.
\end{equation*}
Given   $x_0\in [-2a, 2a]^d$ and $\ell\in\NN$,  we define  
\begin{equation}\label{11-3}
{\mathcal{D}}_{j, d+1, x_0}^\ell: =\bigl (\p_j  +\p_j g(x_0)\p_{d+1}\bigr)^{\ell}=\bigl( \xi_j (x_0)\cdot \nabla\bigr)^\ell,\   \ j=1,2,\dots, d, 
\end{equation}
where 
\begin{equation}\label{11-4-0}
\xi_j (x_0) :=e_j + \p_j g(x_0) e_{d+1}.
\end{equation}
Note that $\xi_j (x_0)$ is the  tangent vector    to  $\p' G$ at the point  $(x_0, g(x_0))$ that is parallel to   the $x_jx_{d+1}$-coordinate plane, while  ${\mathcal{D}}_{j, d+1, x_0}^\ell$ is a constant multiple of the $\ell$-th order directional derivative in the direction of $\xi_j(x_0)$. We will also consider mixed directional derivatives, defined as 
\begin{align}
{\mathcal{D}}_{\tan, x_0} ^{\pmb{\al}}  ={\mathcal{D}}_{1, d+1, x_0}^{\al_1} {\mathcal{D}}_{2, d+1, x_0}^{\al_2} \dots {\mathcal{D}}_{d, d+1, x_0}^{\al_d},\label{11-4}
\end{align} 
where  $\pmb{\al}=(\al_1,\dots, \al_d)\in \ZZ_+^d$, $x_0\in [-2a, 2a]^d$, $(x,y)\in G$ and  $  f\in C^1(G).$
We also define for  $(x,y)\in G$ and $j=1,\dots, d$,
$${\mathcal{D}}_{\tan} ^{\pmb{\al}}  f(x,y):=({\mathcal{D}}_{\tan, x}^{\pmb{\al}}  f)(x,y), \   \     \mathcal{D}_{j, d+1} ^r f(x,y) =\mathcal{D}_{j, d+1, x}^r f(x,y). $$
Note that  the operators ${\mathcal{D}}_{\tan}^{\pmb{\al}}$ and $\p_{d+1}$ are commutative. 
Finally, let $M>10$  be  a constant satisfying  that 
\begin{equation}\label{11-5-2}
\max_{1\leq i\leq j\leq d} \|\p_i\p_j g\|_{L^\infty([-2 a,2a]^d)}\le M\  \  \text{and}\  \  \|\nabla g (0)\|\le M.
\end{equation}

In this section, we shall prove the following higher dimensional extension of Theorem~\ref{THM:2D BERN}. 

\begin{thm}\label{thm:2d bern-2} 
	If   $0<p\leq \infty$, $\ld \in (1,2]$ and      $f\in \Pi_n^{d+1}$, then  for any $\pmb{\al}\in\ZZ_+^d$, 
	\begin{equation}\label{11-6-18a}
	\|\vi_n^{i}{\mathcal{D}}_{\tan}^{\pmb\al}\partial_{d+1}^{i+j} f\|_{L^p(G)}\le  c n^{|\pmb \al|+i+2j}\|f\|_{L^p(G_\ast)},\  \  i,j=0,1,\dots, 
	\end{equation}
	where $G_\ast =\overline{G_\ast(\ld)}$, and  $c$ is a positive constant depending  only on $M$, $d$, $\ld$, $\pmb{\al}, i,j$,  and $p$.
\end{thm}

As a result, we can also extend  Corollary~\ref{cor-10-2} as follows:

\begin{cor}\label{cor-11-2} Let $\ld \in (1,2]$ and $\mu>1$ be two given parameters.  If  $0<p\leq \infty$  and       $f\in\Pi_n^{d+1}$,  then for $\pmb{\al}\in\ZZ_+^d$, and  $ i,j=0,1,\dots$,
	\begin{align*}
	\Bl\|\vi_n(\xi)^{i} &\max_{ u\in \Xi_{n,\mu,\ld}(\xi)}\Bigl|  {\mathcal{D}}_{\tan, u}^{\pmb{\al}}\partial_{d+1}^{i+j}f(\xi)\Br|\Br\|_{L^p(G; d\xi)} \le  c_\mu n^{|\pmb{\al}|+2j+i}\|f\|_{L^p(G_\ast(\ld))},	\end{align*}
	where 
	$$\Xi_{n, \mu,\ld} (\xi):= \Bl\{ u\in [-\ld a, \ld a]^d:\  \  \|u-\xi_x\|\leq \mu \vi_n(\xi)\Br\}.$$
\end{cor}

The proof of Corollary~\ref{cor-11-2} is very similar to that of Corollary~\ref{cor-10-2}, so we omit the dedails.

The rest of this section is devoted to the proof of Theorem~\ref{thm:2d bern-2}.  
To illustrate our idea, we start with the simpler case when $\pmb{\al}=r e_1$ and $i=j=0$.  In this case,  we can restate Theorem~\ref{thm:2d bern-2}   as follows: 
\begin{lem}\label{lem-11} 
	If   $0<p\leq \infty$, $\ld \in (1,2]$ and      $f\in \Pi_n^{d+1}$, then  
	\begin{equation}\label{11-6-18b}
	\|{\mathcal{D}}_{1, d+1}^{r} f\|_{L^p(G)}\le  c n^{r}\|f\|_{L^p(G_\ast)},\  \  r=0,1,\dots, 
	\end{equation}
	where $G_\ast =\overline{G_\ast(\ld)}$, and  $c$ is a positive constant depending  only on $M$, $d$, $\ld$, $r$,  and $p$.
\end{lem}

\begin{proof} This lemma can be deduced from  Theorem~\ref{THM:2D BERN} and Fubini's theorem. 
	To simplify our  notations,  we shall   give   the proof for the case of $d=2$  (i.e., 	 $G\subset \R^3$) and $0<p<\infty$ only. The proof  for $d>3$ is almost identical, while the case $p=\infty$ is simpler and can be treated similarly. 
	
	First, by Fubini's theorem, we may write 
	\begin{align}\label{2-2-0:bern}
	\iiint\limits_G &  \Bl|\mathcal{D}_{1, 3} ^r f(x_1, x_2, y)\Br|^pdx_1dx_2dy =\int_{-a}^a  I(x_2) \, dx_2,	
	\end{align}
	where 
	$$ I(x_2) :=  \int_{-a}^a  \int_{g(x_1, x_2)-a}^{g(x_1, x_2)} \Bl|\mathcal{D}_{1, 3}^r f(x_1, x_2, y)\Br|^p\, dy \, dx_1.$$
	For each fixed $x_2\in [-a, a]$, applying Theorem~\ref{THM:2D BERN}  to the function $g(\cdot, x_2)$ and the polynomial $f(\cdot, x_2, \cdot)$, we obtain
	\begin{align*}
	I(x_2) & \leq C n^r  \int_{-\ld a}^{\ld a} \int_{g(x_1, x_2) -\ld a}^{g(x_1, x_2)} |f (x_1,  x_2, y)|^p\, dy dx_1.
	\end{align*}
	Integrating this last inequality over $x_2\in [-a,a]$, and using~\eqref{2-2-0:bern},  we deduce  the desired estimate.
\end{proof}

The above proof of Lemma~\ref{lem-11}  already contains some    ideas  required  for the proof of Theorem~\ref{thm:2d bern-2} in the general case. Indeed, a similar method  can be applied to certain directional derivatives (instead of the derivative $\mathcal{D}_{1, d+1}^r$). To be more precise, given  $\xi\in\SS^{d-1}$,  we define the $\ell$-th order differential  operator $\wt {D}^\ell _{\xi}$  
by 
\begin{equation}\label{directional}
\wt {D}^\ell _{\xi} f(x,y)=\Bl(\f d{dt}\Br)^\ell \Bl[ f(x+t \xi, y+t \p_\xi g(x))\Br]\Bl|_{t=0},\   \  (x,y)\in G,\   \  f\in C^\ell (\RR^2).
\end{equation}
That is, 
$$\wt {D}^\ell _{\xi} f(x_0,y_0) = \Bl(\bigl(\xi, \p_\xi g(x_0)\bigr)\cdot \nabla \Br)^\ell f(x_0, y_0),\  \  (x_0, y_0)\in G, $$
where  $\p_\xi g(x_0)$ denotes the directional derivative of $g$ at the point $x_0$ in the direction of $\xi$.
Note  that  $(\xi, \p_\xi g(x))$
is the tangent vector   to $\p' G$ at the point $(x, g(x))$  that is   parallel to   the plane spanned by the vectors $(\xi,0)$ and $e_{d+1}$.

We can now  extend Lemma~\ref{lem-11} to   the case of more general  directional derivatives $\wt{D}_\xi^\ell$ as follows: 

\begin{lem}\label{lem-11-2} 
	If   $0<p\leq \infty$, $\ld \in (1,2]$ and      $f\in \Pi_n^{d+1}$, then  
	\begin{equation}\label{11-6-18c}
	\max_{\xi\in\SS^{d-1}} 	\| \wt{D}^r_\xi f\|_{L^p(G)}\le  c n^{r}\|f\|_{L^p(G_\ast)},\  \  r=0,1,\dots, 
	\end{equation}
	where $G_\ast =G_\ast(\ld)$, and  $c$ is a positive constant depending  only on $M$, $d$, $\ld$, $r$,  and $p$.
\end{lem}

\begin{proof}  The proof follows along the same line as that of Lemma~\ref{lem-11}. 
	We shall  give the proof  for $p<\infty$ only, since the case $p=\infty$ is simpler and can be treated similarly.
	
	Let  $\xi\in\SS^{d-1}$ be a fixed direction, and denote by $E_{\mu,\xi}$
	the orthogonal  projection of the cube $[-\mu a, \mu a]^d$ onto the space $\xi^{\bot}:=\{\eta\in\RR^d:\  \    \eta\cdot \xi=0\}$:
	$$ E_{\mu,\xi}: =\Bl\{ x-(x\cdot \xi) \xi:\  \  x\in [-\mu a, \mu a]^d\Br\},\   \  \mu\ge 1.$$
	Given each     $\eta\in E_{\mu, \xi}$, the set 
	$\{t\in\RR:\   \  \eta+t\xi \in [-\mu a, \mu a]^d\}$
	is a compact interval in $\RR$, which we denote by $[a_{\mu,\xi}(\eta), b_{\mu,\xi}(\eta)]$. By Fubini's theorem, we then  have that for  $1<\mu <\ld$,
	\begin{align*}
	\int_G &| \wt {D}^r _{\xi} f(x,y)|^p dxdy= \int_{E_{1,\xi}} \int_{a_{1, \xi}(\eta)}^{b_{1, \xi}(\eta)}\int_{g(\eta +s \xi)-a}^{g(\eta+s\xi)} |\wt {D}^r _{\xi} f (\eta+s\xi,y)|^p \,  dy\, ds d\eta\\
	&\leq  \int_{E_{1,\xi}}\Bl[ \int_{a_{\mu, \xi}(\eta)}^{b_{\mu, \xi}(\eta)}\int_{g(\eta +s \xi)-a}^{g(\eta+s\xi)} \Bl|  \f {d^r} {dt^r} \Bl[ f \bl(\eta+(s+t)\xi,y+t\p_\xi g(\eta+s\xi)\br)\Br]\Br|_{t=0} \Br|^p \,  dy\, ds\Br] d\eta\\
	&=:I.
	\end{align*}
	For simplicity,
	we	define,   for each fixed   $\eta\in E_{1, \xi}$,
	\begin{align*}
	g_\eta (s) &:=g (\eta+s\xi),\   \   s\in [ a_{\mu,\xi} (\eta), b_{\mu, \xi} (\eta)],\\
	f_{\eta}(s, y)&: =f( \eta+s\xi, y),\   \    s\in [ a_{\mu,\xi} (\eta), b_{\mu, \xi} (\eta)],\   \  g_{\eta} (s)-a\leq y\leq g_{\eta}(s).
	\end{align*} 
	We then write the integral $I$ as 
	\begin{align*}
	I
	&= \int_{E_{1,\xi}} \int_{a_{\mu, \xi}(\eta)}^{b_{\mu, \xi}(\eta)}\int_{g_\eta(s) -a}^{g_{\eta}(s)} \Bl|  (\p_1 +g_\eta'(s)\p_2)^r f_\eta (s,y) \Br|^p \,  dy\, ds d\eta.
	\end{align*}
	Since $1<\mu <\ld$, it is easily seen  that  for each  $\eta\in E_{1, \xi}$, 
	$b_{\mu, \xi}(\eta)-a_{\mu, \xi}(\eta) \ge 2(\mu-1)a>0$, and 
	$$ 
	\Bl [a_{\mu,\xi} (\eta)-(\ld-\mu)a, b_{\mu,\xi} (\eta) +(\ld-\mu) a\Br]\subset [a_{\ld,\xi} (\eta), b_{\ld,\xi} (\eta)].$$
	It then follows from Theorem~\ref{THM:2D BERN} that 
	\begin{align*}
	I
	&\leq C n^{pr}  \int_{E_{1,\xi}}\Bl[ \int_{a_{\ld, \xi} (\eta)}^{b_{\ld, \xi}(\eta) }\int_{g_\eta(s)-\ld a }^{g_\eta(s)} |  f_\eta(s,y)|^p \,  dy\, ds\Br] d\eta
	\leq C n^{r p} \int_{G_\ast} |f(x,y)|^p\, dxdy.
	\end{align*}

	%
	%
\end{proof}

We also need a combinatorial identity   on mixed directional derivatives.   For convenience, we extend the definition   $\p_\xi:=\xi\cdot\nabla$ of  directional derivatives  to   all $\xi\in\RR^n$. Thus, $\p_{s\xi}=s\p_\xi$ for any $s\in\RR$ and $\xi\in\RR^n$.   As  a   consequence of the Kemperman lemma on mixed differences (see~\cite[(3.7)]{CD} or~\cite[Lemma 4.11, p.338]{BS}), we have 
\begin{lem}\label{lem-11-3}
	Let $\xi_1,\dots, \xi_r$ be arbitrary vectors in $\RR^n$.
	Then 
	\begin{equation}\label{11-12}
	\p_{\xi_1}\p_{\xi_2}\dots\p_{\xi_r}=\sum_{ S \subset \{1,2,\dots, r\}} (-1)^{|S|}  \p_{\xi_S}^r,
	\end{equation}
	where  
	the sum extends over all subsets $S$ of $\{1,2,\dots, r\}$, $|S|$ is the cardinality of $S$ and  $\xi_S =-\sum_{j\in  S} j^{-1} \xi_j$. 
\end{lem}
\begin{rem}
	In the case when all the vectors $\xi_1,\dots, \xi_r$ are unit, the following interesting inequality with sharp constant was proved in \cite{CD}:
	$$|\p_{\xi_1}\dots\p_{\xi_r} f(x)|\leq \max_{\xi\in \SS^{n-1}} |\p_\xi^r f(x)|,\   \  x\in\RR^n,\   \  \xi_1,\dots,\xi_r\in\SS^{n-1},\  \  \ f\in C^r(\RR^n). $$
	Note that in Lemma~\ref{lem-11-3}  the vectors $\xi_1,\dots, \xi_r$ are not necessarily unit, and the sum on the right hand side of~\eqref{11-12} uses at most $2^r$ directions.   
	This is very important in our later applications. 	
	
\end{rem}

\begin{proof} The proof relies on the following  identity   on finite differences, whose proof can be found in  \cite[Lemma 4.11, p.338]{BS}:
	\begin{equation}\label{kemperman} \tr_{h_1} \tr_{h_2}\dots \tr_{h_r}=\sum_{ S \subset \{1,2,\dots, r\}} (-1)^{|S|} T_{h_S^\ast} \tr_{h_S}^r,\   \   \  \forall h_1,\dots, h_r\in \RR^n, \end{equation}
	where   $$h_S^\ast = \sum_{j\in S} h_j\   \   \  \text{and}\   \   \  h_S =-\sum_{j\in S} j^{-1} h_j.$$
	Here   we recall that  for $x, h\in\RR^n$ and $f:\RR^n\to \RR$, 
	$$ T_h f(x)=f(x+h),\   \   \tr_h f(x)=(T_h-I)f(x)=f(x+h)-f(x).$$ 
	
	Note that  for  $f\in C^r(\RR^n)$ and  $x, \xi_1,\dots,\xi_r\in\RR^n$, 
	\begin{align*}
	\tr_{t\xi_1} \tr_{t\xi_2} \dots  \tr_{t\xi_r} f(x)=\int_{[0,t]^r} (\p_{\xi_1}\dots\p_{\xi_r} f)\bl(x+ \sum_{j=1}^r u_j\xi_j\br)\,  du_1\dots du_r,  \  \  \forall t>0,
	\end{align*}
	which, in particular, implies 
	\begin{align*}
	\lim_{t\to 0+}  	t^{-r}	\bl(\tr_{t\xi_1} \dots \tr_{t\xi_r} f\br)(x)=\Bl(\p_{\xi_1}\dots \p_{\xi_r}f\Br)(x),\  \  f\in C^r(\RR^n).
	\end{align*}
	On the other hand, however,  using~\eqref{kemperman} with $h_j=t\xi_j$, we obtain
	\begin{align*}
	t^{-r}\Bl(	\prod_{j=1}^r \tr_{t\xi_j} \Br)f(x)= t^{-r}\sum_{ S\subset \{1,2,\dots, r\}} (-1)^{|S|} \tr_{t\xi_S}^r f\Bl(x+t\sum_{j\in S} \xi_j\Br).
	\end{align*}
	\eqref{11-12} then follows by letting $t\to 0+$. 
\end{proof}

Now we are in a position to prove Theorem~\ref{thm:2d bern-2}:\\

\begin{proof}[Proof of Theorem~\ref{thm:2d bern-2}]Since the operators $\mathcal{D}_{\tan}^{\pmb \al} $ and $\p_{d+1}$ are commutative, by the univariate Bernstein inequality (or the higher-dimensional analogue of Lemma~~\ref{lem:univ BM}),  it is sufficient to prove that 
	\begin{equation}\label{11-14}
	\|{\mathcal{D}}_{\tan}^{\pmb\al} f\|_{L^p(G)}\le  c n^{|\pmb \al|}\|f\|_{L^p(G_\ast)},\  \  \   \  f\in\Pi_n^{d+1}.
	\end{equation}

	Let  $(x_0, y_0)$ be an arbitrarily fixed  point in $G$.  Then  by~\eqref{11-3} and~\eqref{11-4}, we have  that for any ${\pmb\al} \in\ZZ_+^d$, 
	$$ \mathcal{D}_{\tan}^{\pmb{\al}} f(x_0, y_0) =  \p_{\xi_1(x_0)}^{\al_1} \dots \p_{\xi_d(x_0)}^{\al_d} f(x_0, y_0),$$
	where  the vectors $\xi_j(x_0)$ are given in~\eqref{11-4-0}.
	Let $\ell=|\pmb{\al}|$ and let $\{\eta_j(x_0)\}_{j=1}^\ell$ be the  sequence of vectors such that 
	$$\eta_{\al_{j-1}+1} (x_0)=\dots=\eta_{\al_{j-1} +\al_j}(x_0) =\xi_j(x_0),\   \ j=1,2,\dots, d,$$ where $\al_0=0$.
	Then using~\eqref{11-12}, we have that 
	\begin{align}\label{11-15-0}
	& \p_{\xi_1(x_0)}^{\al_1} \dots \p_{\xi_d(x_0)}^{\al_d} =\prod_{j=1}^\ell \p_{\eta_j(x_0)} 
	= \sum_{S\subset \{1,2,\dots, \ell\}} (-1)^{|S|} \p_{\eta_S(x_0)}^\ell, 
	\end{align}
	where 
	$\eta_S(x_0) = -\sum_{j\in S} j^{-1}\eta_j(x_0).$
	As can be easily seen from~\eqref{11-4-0},  for each $S\subset \{1,2,\dots, \ell\}$, the vector $\eta_S (x_0)$ can be written in the form
	$$ \eta_S (x_0)=c_S\Bl(\xi_S, \p_{\xi_S} g(x_0)\Br),\   \  c_S>0,\  \   \xi_S\in\SS^{d-1},$$
	where  $c_S$ and $\xi_S$ depend  only on the set $S$ and  $\pmb{\al}$ (but independent of $x_0$).   
	Thus, by~\eqref{11-15-0}, it follows that  
	\begin{align}
	\Bl| \p_{\xi_1(x_0)}^{\al_1} \dots \p_{\xi_d(x_0)}^{\al_d} f(x_0, y_0)\Br|\leq C_{\pmb\al } \max _{S\subset \{1,2,\dots, \ell\}} |\wt{D}_{\xi_S}^{\ell } f(x_0, y_0)|, \label{11-16}
	\end{align}
	where  the operator $\wt{D}_\xi^\ell f(x_0, y_0)$  is given in~\eqref{directional}; that is, $$\wt{D}_\xi ^\ell f(x_0, y_0)=\p_{(\xi, \p_\xi g(x_0))}^\ell 
	f(x_0, y_0),\   \    \xi\in \SS^{d-1}.$$
	
	Since the maximum in~\eqref{11-16} is taken over a set of $2^\ell$ elements, and all the vectors $\xi_S$ are independent of $(x_0, y_0)$, 
	it follows by Lemma~\ref{lem-11-2}  that 
	\begin{align*}
	\|{\mathcal{D}}_{\tan}^{\pmb\al} f\|_{L^p(G)}&\le C \max_{S\subset \{1,\dots, \ell\}} \|
	\wt{D}_{\xi_S}^{|\pmb \al|} f\|_{L^p(G)} \leq 
	c n^{|\pmb \al|}\|f\|_{L^p(G_\ast)}.
	\end{align*}
	This completes the proof. 
\end{proof}

\begin{rem} Note that in the above proof, it is very important that the vectors $\xi_S$ in the estimate~\eqref{11-16}   are independent of $(x_0, y_0)\in G$.  This is  the reason why we do not normalize the vectors $\xi_j(x_0)$   in~\eqref{11-4-0}.
\end{rem}


\section{Bernstein type inequality on general $C^2$ domains}\label{sec:14}

In this last section, we established a tangential   Bernstein type inequality  on  domains of special type (see Theorem~\ref{thm:2d bern-2}). The decomposition  Lemma (i.e.,   Lemma~\ref{lem-2-1-18}) proved   in Chapter 2  allows us to extend  this inequality to   a more general $C^2$-domain. 
To be more precise, let  $\Og\subset \RR^{d+1}$ be  the closure of a bounded, open and connected set with  $C^2$ boundary $\Ga=\p \Og$. 
Given $\eta\in\Ga$, we denote by $\mathbf{n}_\eta$  the unit outer normal vector   to  $\Ga$ at  $\eta$, and $\mathcal{S}_\eta$  the set of all unit tangent vectors to $\Ga$ at $\eta$; that is, $\mathcal{S}_\eta:=\{ \pmb{\tau}\in\SS^d:\  \ \tau\cdot \mathbf{n}_\eta=0\}$.
Given nonnegative integers $l_1, l_2$ and a parameter $\mu\ge 1$, we define the maximal operators $\mathcal{M}_{n,\mu}^{l_1, l_2} $, $n=1,2,\dots$ by
$$ \mathcal{M}_{n,\mu}^{l_1, l_2} f(\xi):=\max_\eta \max_{\pmb{\tau}_\eta\in\mathcal{S}_\eta}
\bl|  \p_{\pmb{\tau}_\eta}  ^{l_1} \p_{\nb_\eta}^{l_2} f(\xi)\br|,\  \  \xi\in\Og,\   \  f\in C^\infty(\Og)$$
with    the left most maximum being  taken over all $\eta\in\Ga$ such that 
\begin{equation}\label{12-1-0}
\|\eta-\xi\|\leq  \mu \Bl(\sqrt{\dist(\xi, \Ga)} +n^{-1}\Br)=: \mu \vi_{n,\Ga} (\xi). 
\end{equation}

The    main goal in this section is to show the following extension of Theorem~\ref{thm:2d bern-2}:

\begin{thm}\label{thm-10-1}  If     $f\in \Pi_n^{d+1}$,   $\mu>1$ and $0<p\leq \infty$, then 
	\begin{equation}\label{bernstein-tan}
	\Bl\|\vi_{n, \Ga}^j \mathcal{M}_{n,\mu}^{r, j+l} f \Br\|_{L^p(\Og)}\leq C_\mu n^{r+j+2l} \|f\|_{L^p(\Og)},\  \ r,j,l=0,1,\dots, 
	\end{equation}  
	where the function $\vi_{n,\Ga} $ is given in~\eqref{12-1-0}, and  the constant $C_\mu$ is independent of $f$ and $n$.
\end{thm}
\begin{proof} Let $ f\in\Pi_n^{d+1}$, and    let  $\Og^\da:=\{\xi\in\Og:\  \  \dist(\xi, \Ga) > \da\}$ for $\da\in (0,1)$.  By Fubini's theorem and the univariate Bernstein inequality~\eqref{markov-bern},  we have   that for any $\da\in (0,1)$,
	\begin{equation}\label{10-1}
	\| \p^{\pmb{\b}} f \|_{L^p(\Og^\da)}\leq C_{{\pmb \b}, \da}  n^{|\pmb{\b}|} \|f\|_{L^p(\Og)},\   \  \forall \pmb {\b}\in \ZZ_{+}^{d+1}.
	\end{equation}
	which in turn implies that  
	\begin{equation*}\label{Bernstein-tan-1}
	\Bl\|\mathcal{M}_{n,\mu}^{r, l} f \Br\|_{L^p(\Og^\da)}\leq C n^{r+l} \|f\|_{L^p(\Og)},\  \  r, l=0,1,\dots. 
	\end{equation*}  
	Thus, to show~\eqref{bernstein-tan},  by Lemma~\ref{lem-2-1-18},  it suffices to prove  that    for each domain $G\subset \Og$ of special type  attached to $\Ga$, 
	\begin{equation}\label{bernstein-tan-2}
	\Bl\|\vi_{n, \Ga}^j \mathcal{M}_{n,\mu}^{r, j+l} f\Br\|_{L^p(G_\ast(c_0))}\leq C_{\mu, G} n^{r+j+2l} \|f\|_{L^p(G_\ast)}, r,j,l=0,1,\dots, 
	\end{equation} 
	where $G_\ast=G_\ast(2)$, and $c_0\in (0,1)$ is a small parameter depending only on $G$.
	Without loss of generality,  we may assume that $G$ is an $x_{d+1}$-domain of   the  form~\eqref{11-2-2} with $a\ge 1$ and  $g\in C^2([-2a, 2a]^d)$ satisfying    $\min _{x\in [-2a, 2a]^d} g(x)\ge 4a,$ 
	such that $G_\ast(2) \subset \Og$ and $\p' G_\ast(2)\subset \Ga$.  Then according to Lemma~\ref{metric-lem},  we have 
	\begin{equation}\label{12-5}
	\da(x,y):=g(x)-y \sim \dist (\xi, \Ga),\   \    \  \forall \xi=(x,y)\in G. 
	\end{equation}

	Next, given  $\xi=(\xi_x, \xi_y)\in G$, we  set  $\vi_n(\xi): =\sqrt{g(\xi_x) -\xi_y} +\f 1n$, and define 
	\begin{align*}
	\Ga_{n, \mu} (\xi): &=\{ \eta\in \Ga:\  \ \|\eta-\xi\|\leq \mu \vi_{n,\Ga} (\xi)\}.
	\end{align*}
	Clearly, we may choose the parameter  $c_0\in (0,1)$ small enough so that $\Ga_{n, \mu} (\xi) \subset \p' G_\ast$ for every $\xi\in G_\ast (c_0)$.
	Thus, using~\eqref{12-5}, we have that 
	$$ \Ga_{n,\mu} (\xi) \subset \Bl\{ (u, g(u)):\  \   u\in \Xi_{n, \mu} (\xi)\Br\}\subset \p' G_\ast,\  \ \forall \xi\in G_\ast(c_0),$$ 
	where 
	$$\Xi_{n, \mu} (\xi):= \Bl\{ u\in [-2a, 2a]^d:\   \   \|u-\xi_x\|\leq c_1\mu \vi_n(\xi)\Br\},$$ 
	and the constant $c_1$ depends only on the set $G$.
	
	Finally, we recall that 
	$$ \mathcal{D}_{\tan, x_0}^{\pmb{\al}} f(x, y) =  \p_{\xi_1(x_0)}^{\al_1} \dots \p_{\xi_d(x_0)}^{\al_d} f(x, y),\   \ {\pmb\al} \in\ZZ_+^d, \  \  x_0\in [-2a, 2a]^d, $$
	where  the vectors $\xi_j(x_0)$ are given in~\eqref{11-4-0}. Clearly,  for each $\xi=(a_1, a_2\dots, a_{d+1})\in\RR^{d+1}$ and $x_0\in [-2a, 2a]^d$, we have 
	$$ \xi=\sum_{j=1}^d a_j \xi_j(x_0) + \Bl( a_{d+1} -\sum_{j=1}^d a_j \p_j g(x_0) \Br) e_{d+1}.$$ This implies that  for each $\xi\in G_\ast (c_0)$, 
	\begin{align*}
	\mathcal{M}_{n,\mu} ^{ r, j+l} f(\xi) &=\max_{\eta\in\Ga_{n,\mu} (\xi)} \max_{\pmb{\tau}_\eta \in \mathcal{S}_\eta} 
	\bl|  \p_{\pmb{\tau}_\eta}  ^{r} \p_{\nb_\eta}^{j+l} f(\xi) \br|\\
	&\leq C  \max_{u\in\Xi_{n,c\mu} (\xi)} \max_{\sub{|\pmb{\al}|+i =r+j+l\\
			0\leq 	i\leq j+l}} 
	\bl|  \mathcal{D} _{\tan, u} ^{\pmb \al} \p_{d+1}^{i} f(\xi) \br|. 
	\end{align*}
	Thus, using Corollary~\ref{cor-11-2} and considering the  cases $j\leq i\leq j+l$ and $0\leq i<j$ separately, we obtain 
	\begin{align*}
	&\Bl\| \vi_n(\xi)^j \max_{u\in\Xi_{n,\mu} (\xi)} 
	\bl|  \mathcal{D} _{\tan, u} ^{\pmb \al} \p_{d+1}^{i} f(\xi) \br| \Br\|_{L^p(G_\ast(c_0), d\xi)} \leq   C n^{r+j+2l} \|f\|_{L^p(G^\ast)}.
	\end{align*}
	Combining the above estimates, we deduce~\eqref{bernstein-tan-2},  and hence  complete the proof of the theorem.
\end{proof}


\chapter{Inverse inequality for $1\leq p\leq \infty$}\label{sec:15}

The main purpose in this section is to show Theorem~\ref{inverse-thm}, the inverse theorem. By Theorem~\ref{thm-9-1-00}, $\og^r_\Og(f, t)_p\leq C_{p,q} \tau_r(f, Ct)_{p,q}$ for $1\leq q\leq p\leq \infty$,  where $\tau_r(f,t)_{p,q}$ is the $(q,p)$-averaged modulus of smoothness given in~\eqref{eqn:ivanov}. Thus, it is sufficient to prove 
\begin{thm} \label{thm-15-1}If $r\in\NN$, $1\leq q\leq  p\leq \infty$ and $f\in L^p(\Og)$, then 
	$$\tau_r (f, n^{-1})_{p,q} \leq C_{r} n^{-r} \sum_{s=0}^n (s+1)^{r-1} E_s (f)_p.$$
	
\end{thm}
Here we recall that $L^p(\Og)$ denotes  the space $L^p(\Og)$ for $p<\infty$  and  the space $C(\Og)$  for  $p=\infty$.

The proof of Theorem~\ref{thm-15-1} relies on two lemmas. To state these   lemmas, we recall that 
for $t>0$, $\xi\in\Og$ and $f\in L^p(\Og)$, 
$$U( \xi, t):= \{\eta\in \Og:\  \  \rho_\Og(\xi,\eta) \leq t\},$$  and 
$$ w_r (f, \xi,  t)_q : =\begin{cases}
\displaystyle \Bl( \f 1 {|U(\xi,t)|} \int_{U( \xi,t)} |\tr_{(\eta-\xi)/r} ^r (f,\Og,\xi)|^q \, d\eta\Br)^{\f1q},\  \  & \text{if $1\leq q <\infty$};\\
\sup_{\eta\in U( \xi,t)} |\tr_{(\eta-\xi)/r}^r (f,\Og,\xi)|,\   \ &\text {if $q=\infty$}.\end{cases}$$
\begin{lem} \label{lem-15-1}  Let $G\subset \Og$ be a domain of special type attached to $\Ga$. 
	If  $r\in\NN$, $1\leq q\leq  p\leq \infty$ and $f\in L^p(\Og)$, then 
	\begin{equation}\label{inverse:15-1} \Bl\| w_r (f, \cdot,  n^{-1})_q \Br\|_{L^p(G)}\leq C_{r,p} n^{-r} \sum_{s=0}^n (s+1)^{r-1} E_s (f)_{L^p(\Og)}.\end{equation}
\end{lem}
\begin{proof} By monotonicity, it is enough to consider the case $q=p$. 
	It is  easily seen from the definition that 
	\begin{equation}\label{15-1}\|w_r(f,\cdot, t)_p\|_p\leq C_{p,r}\|f\|_p.\end{equation}
	Without loss of generality, we may assume that
	$$ G:=\{(x,y):\  \  x\in (-b, b)^d,\   \   g(x)-1<y\leq g(x)\},$$
	where $b>0$ and $g\in C^2(\RR^d)$.  We may aslo assume that $n\ge N_0$, where $N_0$ is a sufficiently large positive integer depending only on $\Og$, since otherwise~\eqref{inverse:15-1} follows directly from  the inequality $\|w_r(f,\cdot, t)_p\|_p\leq C E_0 (f)_p$, which can be obtained from~\eqref{15-1}.
	
	For $0\leq k\leq n$, let $P_k\in\Pi_k^{d+1}$ be such that 
	$\|f-P_k\|_{L^p(\Og)} =E_k(f)_{L^p(\Og)}$. Let $m\in\NN$ be such that $2^{m-1} \leq n <2^m$. Then by~\eqref{15-1}, we have 
	\begin{align*}
	\Bl\| w_r (f, \cdot,  n^{-1})_p \Br\|_{L^p(G)}&\leq \Bl\| w_r (f-P_{2^m}, \cdot,  n^{-1})_p\Br\|_{L^p(G)}+\Bl\| w_r (P_{2^m}, \cdot,  n^{-1})_p \Br\|_{L^p(G)}\\
	&\leq  C \|f-P_{2^m}\|_{L^p(\Og)} +  \sum_{j=0}^{m-1} \Bl\| w_r (P_{2^{j+1}}-P_{2^j}, \cdot,  n^{-1})_p\Br\|_{L^p(G)}.
	\end{align*}
	Thus, for the proof of~\eqref{inverse:15-1},  it suffices to show that for each $P\in\Pi_k^{d+1}$,
	\begin{equation}\label{15-2}
	\Bl\| w_r (P, \cdot, n^{-1})_p \Br\|_{L^p(G)}\leq C n^{-r} k^r \|P\|_{L^p(\Og)}.
	\end{equation}

	To show~\eqref{15-2}, we first  recall  the following  partition  
	of the domain $\overline{G}$ constructed in Section~\ref{sec:5}: $
	\overline{G}=\bigcup_{\bfi\in\Ld_n^d} \bigcup_{j=0}^{n-1} I_{\mathbf{i},j}$,   where   
	\begin{equation}\label{15-4-0}
	I_{\mathbf{i},j}:=\Bl\{ (x, y):\  \  x\in \Delta_{\bfi},\  \   g(x)-y\in [\al_{j}, \al_{j+1}]\Br\}
	\end{equation}
	and 
	\begin{align*}
	\bfi&=(i_1,\dots, i_d)\in \Ld^d_n:=\{ 0, 1,\dots, n-1\}^d\subset \ZZ^d,\\
	\Delta_{\bfi}:&=[t_{i_1}, t_{i_1+1}]\times \dots [t_{i_{d}}, t_{i_{d}+1}]  \   \  \ \text{with}\  \    t_{i}=-b+\f {2i}n b,\\
	\al_j:&= \sin^2 (\f {j\pi}{2\ell_1 n})/(\sin^2\f \pi{2\ell_1}), \   \   j=0,1,\dots, \ell_1 n,\\
	&\text{ with $\ell_1>1$  being  a large  integer parameter}.
	\end{align*}
	As in  Section~\ref{Sec:8}, we also define for any   two given integer parameters $m_0, m_1>1$,  
	\begin{align*}
	\Delta_{\bfi}^\ast&=\Delta_{\bfi,m_0}^\ast:=[t_{i_1-m_0}, t_{i_1+m_0}]\times [t_{i_2-m_0}, t_{i_2+m_0}]\times \dots\times [t_{i_{d}-m_0}, t_{i_{d}+m_0}],\\
	I_{\bfi,j}^\ast:&=I_{\bfi,j, m_0,m_1}^\ast:=\Bl\{ (x, y):\  \ x\in \Delta_{\bfi}^\ast,\   \  \al^\ast_{j-m_1}\leq g(x)-y\leq \al^\ast_{j+m_1}\Br\},
	\end{align*}
	where
	$\al_j^\ast =\al_j$ if $0\leq j\leq n$,  $\al_j^\ast =0$ if $j<0$ and $\al_j^\ast =2$ if $j>n$.  By Lemma~\ref{metric-lem}, we may choose the parameters $m_0, m_1$ large enough so that 
	\begin{equation}\label{15-4}U(\xi, n^{-1}) \subset I_{\bfi, j}^\ast\   \   \text{whenever $\xi\in I_{\bfi,j}$}.\end{equation}
	Note  that for $(\bfi, j)\in\Ld_n^{d+1}$ and  $(x,y) \in I_{\bfi,j}^\ast$, 
	\begin{align}
	\al_j& \sim \f {j^2}{n^2} \sim \da(x,y):=g(x)-y,\    \   j\ge 1, \notag\\
	\al_{j+1}-\al_j &\leq C \f {j+1}{n^2} \leq  \f Cn \vi_n(x,y):=\f C n \bl( \f 1n+\sqrt{\da(x,y)}\br).\label{15-5}
	\end{align}

	Now we turn to the proof of~\eqref{15-2}. 
	Let $P\in \Pi_k^d$ and $1\leq p<\infty$. Then  using Remark~\ref{rem-6-2} ,  Lemma~\ref{metric-lem} and~\eqref{15-4}, we have 
	\begin{align*}
	&\Bl\| w_r (P, \cdot,  n^{-1})_p \Br\|^p_{L^p(G)}
	\leq C \sum_{(\bfi,j)\in\Ld_n^d} \int_{I_{\bfi,j}}  \f 1 {|I_{\bfi,j}^\ast|} \int_{I_{\bfi,j}^\ast(\xi)} |\tr_{(\eta-\xi)/r} ^r (P,\Og,\xi)|^p d\eta \, d\xi,
	\end{align*}
	where 
	$I_{\bfi,j}^\ast(\xi)=\{\eta\in I_{\bfi,j}^\ast:\  \  [\xi,\eta]\in\Og\}.$
	Note that by H\"older's inequality,
	\begin{align*}
	|\tr_{(\eta-\xi)/r}^r (f,\Og, \xi)|^p & \leq \int_{[0,1]^r} \Bl|\p_{(\eta-\xi)/r}^r f(\xi+ r^{-1}(\eta-\xi) (t_1+\dots+t_r))\Br|^p\, dt_1\dots dr_r\\&\leq C \int_{0}^1 \Bl|\p_{\eta-\xi}^r f(\xi+ t(\eta-\xi) )\Br|^p\, dt.
	\end{align*}
	Thus,
	\begin{align}
	&\Bl\| w_r (P, \cdot,  n^{-1})_p \Br\|^p_{L^p(G)}\notag\\
	&\leq C  \sum_{(\bfi,j)\in\Ld_n^d} \int_{I_{\bfi,j}}  \f 1 {|I_{\bfi,j}^\ast|} \int_{I_{\bfi,j}^\ast(\xi)} \int_0^1 \Bl|\p_{\eta-\xi}^r P(\xi+ t(\eta-\xi) )\Br|^p\, dt
	d\eta \, d\xi.\label{15-6}
	\end{align}
	
	To estimate the sum in this last equation,  we shall   use the Bernstein inequality stated in Corollary~\ref{cor-11-2}.   For convenience, given a parameter $\mu>1$, and two nonnegative integers $l_1, l_2$, we define 
	\begin{align*}
	M_{\mu,n}^{l_1,l_2} f(\xi) := &\max_{ u\in \Xi_{n,\mu}(\xi)}\max_{\zeta\in\sph} \Bigl|  ( z_\zeta(u)\cdot \nabla )^{l_1}\partial_{d+1}^{l_2}f(\xi)\Br|,\  \   \xi\in G,\  \  f\in C^\infty(\Og),	\end{align*}
	where $z_\zeta(u)=(\zeta, \p_\zeta g(u))$, and 
	$$\Xi_{n, \mu} (\xi):= \Bl\{ u\in [-2 a, 2 a]^d:\  \  \|u-\xi_x\|\leq \mu \vi_n(\xi)\Br\}.$$
	We choose the parameter $\mu$ large enough so that 
	$\Delta_{\bfi, 4m_0}^\ast \subset \Xi_{n,\mu}(\xi)$ for any $\xi\in I_{\bfi, j}^\ast$.
	By Corollary~\ref{cor-11-2} and Theorem~\ref{thm-10-1}, we have
	\begin{equation}\label{15-7}
	\| \vi_n^{l_2} M_{\mu,n}^{l_1,l_2} P\|_{L^p(G^\ast)} \leq C k^{l_1+l_2} \|P\|_{L^p(\Og)},\   \   \ \forall P\in\Pi_k^{d+1}.
	\end{equation}

	Now fix temporarily   $\xi=(\xi_x, \xi_y)\in I_{\bfi,j}$ and $\eta=(\eta_x, \eta_y)\in I_{\bfi, j}^\ast$.  Then $\|\xi_x-\eta_x\|\leq \f cn$, and
	$$ \eta_y-\xi_y =\eta_y-g(\eta_x)+g(\eta_x)-g(\xi_x)+g(\xi_x) -\xi_y.$$
	By  the mean value theorem, there exists $u\in [\xi_x, \eta_x]$ such that 
	$$ \|(\eta_y-\xi_y)-\nabla g(u) \cdot (\eta_x-\xi_x)\|\leq \al_{j+m_1}^\ast -\al_{j-m_1}^\ast \leq c_1 \f {\vi_n(\xi)}n,$$
	where 
	the last step uses~\eqref{15-5}. Thus, setting $\zeta=n(\eta_x-\xi_x)$, 
	we  have $\|\zeta\|\leq c$ and we may write $\eta-\xi$ in the form
	$$ \eta-\xi=\f 1n \Bl(\zeta, \p_{\zeta} g(u)  +s\vi_n(\xi)\Br),$$
	with 
	$$ s=\f {n(\eta_y-\xi_y) -\p_\zeta g(u)}{\vi_n(\xi)}\in [-c_1, c_1].$$
	It follows that  
	$$ \p_{\eta-\xi} =(\eta-\xi)\cdot \nabla = \f 1n \Bl(\p_{z_\zeta (u)} + s \vi_n(\xi) \p_{d+1}\Br),$$
	where $z_\zeta (u) = (\zeta, \p_{\zeta} g(u))$. 
	This  implies that  for any $(x,y)\in I_{\bfi,j}^\ast$, 
	\begin{align*}
	|\p_{\eta-\xi}^r P(x,y)|&\leq C n^{-r} \max _{0\leq k\leq r}\vi_n(\xi)^k |\p_{z_\zeta(u)}^{r-k} \p_{d+1}^{k} P(x,y)| \\
	&\leq C n^{-r} \max _{0\leq k\leq r}\vi_n(x,y)^k   M_{\mu,n}^{r-k,k} P(x,y).\end{align*}
	Thus, setting 
	$$ P_\ast (x,y):=\max _{0\leq k\leq r}\vi_n(x,y)^k   M_{\mu,n}^{r-k,k} P(x,y),$$
	we obtain from~\eqref{15-6} that 
	\begin{align*}
	&\Bl\| w_r (P, \cdot,  n^{-1})_p \Br\|^p_{L^p(G)}
	\leq C n^{-rp}  \sum_{(\bfi,j)\in\Ld_n^d} \int_{I_{\bfi,j}}  \f 1 {|I_{\bfi,j}^\ast|} \int_{I_{\bfi,j}^\ast(\xi)} \int_0^1 |P_\ast(\xi+t(\eta-\xi))|^p\, dt
	d\eta \, d\xi\\
	&\leq C n^{-rp}  \sum_{(\bfi,j)\in\Ld_n^d} \int_{I_{\bfi,j, 2m_0, 2m_1}^\ast}  |P_\ast(\eta)|^p
	d\eta\leq C n^{-rp} \|P_\ast \|_{L^p(G_\ast(2))}^p \leq C \Bl( \f kn\Br)^{rp} \|P\|_{L^p(\Og)}^p,
	\end{align*}
	where the last step uses~\eqref{15-7}. This proves~\eqref{15-2} for $1\leq p<\infty$. 
	
	Finally,~\eqref{15-2} for  $p=\infty$ can be proved  similarly. This completes the proof of Lemma~\ref{lem-15-1}.
\end{proof}

\begin{lem} \label{lem-15-2}  Let $\va\in (0,1)$ and 
	$\Og^\va:=\{\xi\in \Og:\  \  \dist(\xi, \Ga) >\va\}.$
	If  $r\in\NN$, $1\leq q\leq  p\leq \infty$ and $f\in L^p(\Og)$, then 
	\begin{equation}\label{inverse:15-8} \Bl\| w_r (f, \cdot,  n^{-1})_q \Br\|_{L^p(\Og^\va)}\leq C_{r,p} n^{-r} \sum_{s=0}^n (s+1)^{r-1} E_s (f)_{L^p(\Og)}.\end{equation}
\end{lem}
\begin{proof} 
	The proof is similar to that of Lemma~\ref{lem-15-1}, and in fact, is simpler. It relies on the following Bernstein inequality,
	$$ \|\p^{\b}  P\|_{L^p(\Og^\va)} \leq C k^{|\b|} \|P\|_{L^p(\Og)},\  \  \forall P\in\Pi_k^{d+1},\   \  \forall \b\in\ZZ_+^{d+1},$$
	which is a direct consequence of  the univariate Bernstein  inequality~\eqref{markov-Bern}.
\end{proof}

Now we are in a position to prove Theorem~\ref{thm-15-1} .

\begin{proof}[Proof of Theorem~\ref{thm-15-1}]
	By monotonicity, 	it suffices to consider the case $p=q$.
	By Lemma~\ref{lem-2-1-18},  there exist $\va\in (0,1)$ and domains $G_1,\dots, G_{m_0}\subset \Og$ of special type attached to $\Ga$ such that 
	$$\Ga_\va:=\{ \xi\in\Og:\  \  \dist(\xi, \Ga) \leq \va\} \subset \bigcup_{j=1}^{m_0} G_j.$$
	Setting  $\Og^\va:=\Og\setminus \Ga_\va$, we have 
	\begin{align*}
	\tau_r(f, n^{-1})_{p,p} &\leq \sum_{j=1}^{m_0} \|w_r (f, \cdot, n^{-1})_p\|_{L^p(G_i)} + \|w_r (f, \cdot, n^{-1})_p\|_{L^p(\Og^\va)},
	\end{align*}
	which,  using Lemma~\ref {lem-15-1}
	and  Lemma~\ref {lem-15-2}, is estimated above by a constant multiple of 
	\begin{align*}
	n^{-r}\sum_{s=0}^n (s+1)^{r-1} E_s (f)_{L^p(\Og)}.
	\end{align*}
	This completes the proof. 
\end{proof}

\chapter[Marcinkiewicz-Zygmund type ineq. and pos. cubature formula]{Marcinkiewicz-Zygmund type  inequality
	 and positive cubature formula}\label{sec:16}

%
%
%

Marcinkiewicz-Zygmund (MZ)  type inequalities provide a basic tool for the discretization of the $L^p$-norm and are widely used in the study of the convergence properties of Fourier series, interpolation processes and orthogonal expansions. Interesting results  on these inequalities  for univariate polynomials can be found in \cite{Lu1, Lu2, Lu3, MT2}. 
MZ inequalities  for multivariate polynomials  were  established  on various  multidimensional domains, including Euclidean balls, spheres,  polytopes, cones, spherical sectors, toruses, etc. (see~\cite{Kr2, MK, Da06, NPW}). More recently, these inequalities were studied in a more general setting for elements of finite dimensional spaces (see \cite{DPTT} and the references therein).  The main goal of this chapter is to extend these considerations to  general  $C^2$-domains.  Our main  tool is the  Bernstein inequality proved  in Section~\ref{sec:14}.

We start with a brief description of some necessary notations. Given a metric space $(X,\rho)$ and a number $\va>0$, we say  $\Ld\subset X$  is an  $\va$-separated subset of $X$ (with respect to the metric $\rho$)  if $\rho(\og, \og') \ge \va$ for any two distinct points $\og, \og'\in\Ld$, whereas  an  $\va$-separated  subset $\Ld$  of $X$  is called maximal  if 
$\inf_{\og\in\Ld} \rho(x, \og)<\va$ for any $x\in X$. 
As usual, we assume that $\Og\subset \RR^{d+1}$ is  the closure of an open,   connected, bounded  $C^2$-domain  in $\RR^{d+1}$ with boundary $\Ga=\p \Og$. Let $\rho_{\Og}$ denote the metric on $\Og$ defined by~\eqref{metric}.  Let   $ U( \xi, t):= \{\eta\in\Og:\  \  \rho_\Og(\xi,\eta) \leq t\}$ for  $\xi\in\Og$ and $t>0$.  According to Remark~\ref{rem-6-2}, we have 
\begin{equation}\label{16-1-00}
|U(\xi, t)| \sim t^{d+1} ( t +\sqrt{\dist (\xi, \Ga)}),\  \  \xi\in\Og,\   \  t\in (0,1).
\end{equation}
Finally, for each   bounded function $f$ on $\Og$, and  any $I\subset \Og $,  we  define 
$$\osc(f; I) :=\sup_{\xi,\eta\in I} |f(\xi)-f(\eta)|.$$
Note that unlike the definition of finite difference   we do not need to assume here  $[\xi, \eta]\subset \Og$. 

Our main result in this chapter can be stated as follows.

\begin{thm}\label{thm-16-1:MZ}  Let   $n\in\NN$, $\va\in (0,1)$ and let $\ell>1$ be a given parameter. If  $1\leq p\leq \infty$  and $\Ld\subset \Og$ is $\va/n$-separated  with respect to the metric $\rho_\Og$, then for any $f\in\Pi_{n}^{d+1}$,
	
	\begin{equation}\label{3-4:bern}\Bl( \sum_{\xi\in\Ld} \Bl|U(\xi, \f \va n)\Br|  \Bl( \osc \bl(f; U (\xi, \f {\ell \va} n )\br) \Br)^p \Br)^{\f 1p} \leq C_\ell \va \|f\|_{L^p(\Og)},\end{equation}
	with the usual change of $\ell^p$-norm when $p=\infty$,  and  the constant $C_\ell$ being  independent of $f$, $p$, $n$ and $\va$.
\end{thm}

We can deduce two useful corollaries from Theorem~\ref{thm-16-1:MZ}.
\begin{cor}\label{cor-16-2}
	There  exists a number $\va_0\in (0,1)$ which depends only on $\Og$ and $d$  and has  the following property:  if $\Ld$ is a maximal $\f \va n$-separated subset of $\Og$ (with respect to the metric $\rho_\Og$) with $\va \in (0, \va_0)$, then for any $1\leq p\leq \infty$ and any $f\in \Pi_n^{d+1}$,
	\begin{align*}
	C_1 \Bl(\sum_{ \xi\in\Ld} \Bl|U(\xi, \f \va n)\Br| \max_{\eta\in U(\xi, \va/n)} |f(\eta)|^p \Br)^{\f 1p}\leq \|f\|_{L^p(\Og)} \leq 	C_2 \Bl(\sum_{ \xi\in\Ld} \Bl|U(\xi, \f \va n)\Br| \min_{\eta\in U(\xi, \va/n)} |f(\eta)|^p \Br)^{\f 1p},
	\end{align*}
	with the usual change of $\ell^p$-norm when $p=\infty$, where  the constants $C_1, C_2>0$ are independent of $f$, $\va$ and $n$.	
\end{cor}
\begin{proof}Corollary~\ref{cor-16-2} is a direct consequence of 	 Theorem~\ref{thm-16-1:MZ} and~\eqref{16-1-00}. Since the proof is almost identical to that of Theorem 5.3.6 of \cite[p.~116]{DX2}, we omit the details here.
\end{proof}

\begin{cor}\label{cor-16-3}There  exists a number $\va_0\in (0,1)$ which depends only on $\Og$ and $d$  and has  the following property:  if $\Ld$ is a maximal $\f \va n$-separated subset of $\Og$ (with respect to the metric $\rho_\Og$) with $\va \in (0, \va_0)$,  then there exists a sequence $\{ \ld_\xi:\  \ \xi\in\Ld\}$ of positive numbers such that 
	$$\ld_\xi \sim \Bl|U(\xi, \f \va n)\Br|,\   \  \forall \xi\in\Ld,$$
	and 
	$$ \int_{\Og} f(\eta)\, d\eta =\sum_{\xi\in\Ld} \ld_\xi f(\xi),\   \   \  \forall f\in \Pi_n^{d+1}.$$

\end{cor}

\begin{proof} Corollary~\ref{cor-16-3} can be deduced from Theorem~\ref{thm-16-1:MZ},~\eqref{16-1-00} and the Farkas lemma (\cite[p.~137]{DX2}).  The proof is very similar to the proof of Theorem 4.3 of  \cite{NPW}, and is 	almost identical to that of Theorem 6.3.3 of \cite[p.~138]{DX2}. 	
\end{proof}

Note that the number of elements of any maximal $\f\va n$-separated subset of $\Omega$ with respect to the metric $\rho_\Omega$ is of order $n^{d+1}$, which is optimal as the dimension of $\Pi_{n}^{d+1}$ has the same order. See Remark~\ref{rem:cardinality-max-sep} below for details.

The rest of this chapter is devoted to the proof of Theorem~\ref{thm-16-1:MZ}, which relies on several lemmas. 
The first lemma  is a direct consequence of 	the Markov-Bernstein-type  inequality for algebraic polynomials in one variable:

\begin{lem}\label{lem-3-1:bern}
	Let   $\ell$ be a fixed positive integer, and $n>\ell$ a given positive integer. 
	Let  $0\leq \ta_1<\ta_2<\dots\leq\ta_m\leq  \pi$ satisfy that  $\min_{2\leq i\leq m} (\ta_i-\ta_{i-1})\ge \f 1n$. 
	Then for every  $f\in\Pi_k^1$, 
	$$ \sum_{ i=1}^{m}\Bl(\f {\sin \ta_i} n  +\f1{n^2} \Br) \max_{ |\ta-\ta_i|\leq \f \ell n} |f(\cos\ta )|\leq C_\ell \Bl( 1+\f kn\Br) \int_{-1}^1 |f(x)|\, dx. $$
\end{lem}

To state the  second lemma, we let $G\subset \Og$ be a domain of special type attached to $\Ga$. Without loss of generality, we may assume that
$$ G:=\{(x,y):\  \  x\in (-b, b)^d,\   \   g(x)-1<y\leq g(x)\},$$
where $b\in (0,(2d)^{-1})$ is a  small constant and  $g$ is a $C^2$-function on $\RR^d$ satisfying that $\min_{x\in [-b, b]^d} g(x)\ge 4$. Recall the metric $\wh{\rho}_G$ on the domain $G$ is defined as
\begin{equation}\label{rhog-0}
\wh{\rho}_G(\xi, \eta):=\max\Bl\{\|\xi_x-\eta_x\|,
\Bl|\sqrt{g(\xi_x)-\xi_y}-\sqrt{g(\eta_x)-\eta_y}\Br|\Br\}.
\end{equation}
where  for  $\xi=(\xi_x, \xi_y)$ and  $\eta=(\eta_x, \eta_y)\in G$. 
According to Lemma~\ref{metric-lem}, we have 
\begin{equation}\label{6-1-metric-0}\wh{\rho}_G(\xi,\eta)\sim \rho_{\Og} (\xi,\eta),\    \    \  \xi, \eta\in G.\end{equation}
For $\xi\in G$ and $r\in (0,1)$, define  
$$ B_G (\xi, r):=\Bl\{ \eta\in G:\  \  \wh{\rho}_G(\xi,\eta)\leq r\Br\}.$$

\begin{lem}\label{lem-16-2}Let $\mu>1$ be a parameter. Assume that  $n\in\NN$, $\va\in (0,1)$ and $1\leq p\leq \infty$.  If $\Ld\subset G$ is   $\f \va n$-separated  with respect to the metric $\wh{\rho}_G$,  then for any  $f\in\Pi_{n}^{d+1}$,   
	\begin{equation}\label{3-4:berna} \Bl(\sum_{\og\in\Ld} \ld_{\va/n} (\og) \Bl|\osc \bl(f; B_G \bl(\og, \f {\mu\va}n \br )\br)\Br|^p\Br)^{\f1p} \leq C_\mu\va \|f\|_{L^p(G^\ast)},\end{equation}
	where 
	\begin{equation}\label{3-3:bern}
	\ld_{r}(\xi)=\Bl| B_G (\xi, r)\Br|\sim r^{d+1} \Bl( \sqrt{g(\xi_x)-\xi_y} +r\Br),\   \   \ \xi\in G,\   \  r\in (0,1),
	\end{equation}
	and the constant $C_\mu$ is independent of $\va$, $n$ and $f$.
\end{lem}


\begin{proof}For simplicity,  we  assume that  $d=1$  (i.e., $G\subset \Og\subset \RR^2$).  The  proof below  with slight modifications works equally well for the case $d>1$.


	We first introduce  some necessary notations for  the proof. Let $n_1$ be the  integer such that $n/(c_0 \va)<n_1 \leq 1+n/(c_0 \va) $, where $c_0\in (0,1)$ is a small constant to be specified later.   
	Let  $L:=\max_{x\in [-4b, 4b]}|g(x)|+10$.   Denote by $m$ the nonnegative  integer such that   
	$$ \f {n_1 \arccos (1-\f 1L)} {\pi} -1 \leq m <  \f {n_1 \arccos (1-\f 1L)} {\pi};  $$
	that is, 
	$2L \sin^2 \f { m\pi} {2n_1} <1\leq 2L \sin^2 \f {(m+1)\pi} {2n_1}$.  Define   $\al_j:= 2L \sin^2 \f {j\pi}{2n_1}$ for $j=0,1,\dots, m-1$; $\al_m=1$;  $\al_j=0$ for $j<0$ and $\al_j:=5/4$ for $j> m$.  We also define  $x_i=-b+\f {2i}{n_1} b$ for $i=0,1,\dots, n_1$;  $x_i=-b$ for $i<0$ and $x_i =b$ for $i\ge n_1$.   
	Let 
	\begin{align*}
	I_{i,j}:&=\Bl\{ (x,y)\in G:\  \   x_{i-1}\leq  x\leq x_i,\   \  \al_{j-1}\leq  g(x)-y \leq \al_{j}\Br\},\\
	I^\ast_{i,j, \ell} :&= \Bl\{ (x,y)\in G:\  \   x_{i-1-\ell}\leq  x\leq x_{i+\ell},\   \  \al_{j-\ell-1}\leq g(x)-y \leq \al_{j+\ell}\Br\},
	\end{align*}
	where $\ell$ is a positive integer to be specified later. 
	Then $\overline{G}=\bigcup_{i=1}^{n_1}\bigcup_{j=0}^m I_{i,j}$. 
	
	Note that  for any $\xi=(\xi_x, \xi_y),\eta=(\eta_x, \eta_y)\in G$, 
	\begin{equation}\label{3-2:bern}
	\Bl| \sqrt{g(\xi_x) -\xi_y} - \sqrt{g(\eta_x) -\eta_y}\Br|=2\sqrt{2L} \Bl|\sin \f {\ta_\xi-\ta_\eta}4\Br|\cos \f {\ta_\xi+\ta_\eta} 4,
	\end{equation}
	where $$\ta_\xi =2 \arcsin  \sqrt{ \f{g(\xi_x)-\xi_y} {2L}},\   \   \xi=(\xi_x, \xi_y)\in G.$$
	It then follows from~\eqref{rhog-0}   that there exist two constants $c_1, c_2>0$ such that 
	\begin{equation}\label{3-5-0:bern}
	B_G(\og_{i,j}, c_1 n_1^{-1}) \subset I_{i,j}  \subset B_G(\og_{i,j}, c_2   n_1^{-1})
	\end{equation}	
	for  some  $\og_{i,j}\in I_{i,j}$ and all $1\leq i\leq n_1$ and $1\leq j\leq m$.   We may also choose a positive integer $\ell=\ell_\mu$  large enough so that 
	\begin{equation}\label{16-9}
	B_G(\og_{i,j}, c_2   n_1^{-1}) 	\subset B_G(\og_{i,j}, (3\mu +c_2)n_1^{-1}) \subset I_{i,j,\ell}^\ast.
	\end{equation}

	Now we specify the constant  $c_0 =\f 1 {4c_2}$.  Since $\Ld\subset G$ is $\f 1{n_1 c_0}$-separated with respect to the metric $\wh{\rho}_G$,  this   implies that   every $\og\in\Ld$ is contained in a unique set  $I_{i,j}$ with $1\leq i\leq n_1$ and $1\leq j\leq m$.  Thus,   by~\eqref{3-3:bern} and~\eqref{16-9},
	it is enough to prove   that 
	\begin{equation}\label{3-2-1:bern}
	\Bl(\sum_{i=1}^{n_1} \sum_{j=1}^{m}  |I_{i,j}|  \Bl|\osc \bl(f;  I_{i,j,\ell}^\ast\br)\Br|^p \Br)^{\f 1p}\leq C_{\ell, p} \va \|f\|_{L^p(G^\ast)},\   \   \  f\in\Pi_n^2.
	\end{equation}
	Here and in the rest of the proof, we assume that $1\leq p<\infty$ for simplicity.  The proof below with slight modifications works for $p=\infty$. 
	Note that the set $G$ may not be convex, so we cannot use the mean value theorem on the line segment $[\xi, \eta]$ to evaluate $|f(\xi)-f(\eta)|$ for $\eta, \xi\in I_{i,j,\ell}^\ast$.

	Now let $f\in\Pi_n^2$ and define $F(x,\al):= f(x, g(x)-\al)$ for $-2b\leq x\leq 2b$ and $\al\in\RR$. 
	Then we have 
	\begin{align*}
	\osc(f;  I_{i,j, \ell}^\ast) &\leq 2 \sup_{\al \in [\al_{j-1-\ell}, \al_{j+\ell}]}  \sup_{x\in [x_{i-\ell-1}, x_{i+\ell}]} \Bl|F(x, \al) -\f{n_1}{2b}\int_{x_{i-1}}^{x_i}F(u, \al_j)\, du\Br|\\
	&	\leq 2 \Bl[ a_{i,j} (f) +b_{i.j}(f)\Br],
	\end{align*}
	where 
	\begin{align*}
	a_{i,j}(f):&=	  \sup_{\sub{x\in [x_{i-\ell-1}, x_{i+\ell}]\\
			\al \in [\al_{j-1-\ell}, \al_{j+\ell}]}} \Bl|F(x, \al) -\f{n_1}{2b} \int_{x_{i-1}}^{x_i}F(u, \al)\, du\Br|,\\
	b_{i,j}(f)&:=\f{n_1}{2b} \sup_{\al\in [\al_{j-1-\ell}, \al_{j+\ell}]}\Bl| \int_{x_{i-1}}^{x_i}[F(u, \al)-F(u, \al_j)]\, du\Br|.
	\end{align*}
	It follows that 
	\begin{equation}
	\text{LHS of~\eqref{3-2-1:bern}}\leq   2 \bl( \Sigma_1+\Sigma_2\br), \label{3-7-0:bern}
	\end{equation}
	where 
	\begin{align*}
	\Sigma_1&:=\Bl(\sum_{i=1}^{n_1} \sum_{j=1}^{m}  |I_{i,j}| \bl| a_{i,j}(f)\br|^p\Br)^{\f 1p}\  \ \text{and}\  \  
	\Sigma_2:=\Bl(\sum_{j=1}^{m} \sum_{i=1}^{n_1}  |I_{i,j}|   |b_{i,j} (f)|^p\Br)^{\f1p}.
	\end{align*}
	By H\"older's inequality, we have 
	\begin{align}
	|a_{i,j} (f)|^p&\leq  \f {C_\ell} {n_1^{p-1}}  
	\int_{x_{i-1-\ell}}^{x_{i+\ell}} \sup_{\al \in [\al_{j-1-\ell}, \al_{j+\ell}]}|\p_1 F(v, \al)|^p dv,\label{16-12}\\
	|b_{i,j}(f)|^p &\leq  C n_1(\al_{j+\ell} -\al_{j-1-\ell} )^{p-1} \int_{\al_{j-1-\ell}}^{ \al_{j+\ell}} \int_{x_{i-1}}^{x_i}|\p_2  F(u, \al)|^p\, du d\al\notag\\
	&\leq \f {C_\ell}  {n_1^{p-2}}( \f 1{n_1} +\sqrt{\al_j})^{p-1}
	\int_{\al_{j-1-\ell}}^{ \al_{j+\ell}} \int_{x_{i-1}}^{x_i}|\p_2  F(u, \al)|^p\, du d\al.\label{16-13}
	\end{align}
	Also, note that 	
	by~\eqref{3-3:bern} and~\eqref{3-5-0:bern},
	\begin{equation}\label{3-6:bern}
	|I_{i,j}|\sim n_1^{-2} (\sqrt{\al_j} +n_1^{-1}),\  \  1\leq i\leq n_1,\   \   1\leq j\leq m.
	\end{equation}
	Thus, using~\eqref{16-12} and~\eqref{3-6:bern}, we have 
	\begin{align*}
	\Sigma_1
	&\leq \f {C}{n_1}\Bl[  \sum_{i=1}^{n_1}  \int_{x_{i-1-\ell}}^{x_{i+\ell}} \Bl(\sum_{j=1}^{m} \f 1 {n_1} (\sqrt{\al_j} +n_1^{-1})
	\sup_{\al \in [\al_{j-1-\ell}, \al_{j+\ell}]} |\p_1 F(v, \al)|^p\Br)dv\Br]^{\f1p}.
	\end{align*}
	Since   the function $\p_1 F(v,\al)=(\p_1+g'(v)\p_2) f(v, g(v)-\al)$ is an algebraic polynomial of degree at most $n$ in the variable $\al$ for each fixed $v\in [-2b,2b]$, it   follows from Lemma~\ref{lem-3-1:bern}  that 
	\begin{align*}
	\Sigma_1&\leq 	\f C {n_1} \Bl(  \sum_{i=1}^{n_1} \int_{x_{i-1-\ell}}^{x_{i+\ell}} \int_{0}^{\f 32 } |\p_1 F(v,\al)|^p\, d\al\, dv\Br)^{\f 1p}
	\leq \f{C_{\ell}}{n_1}\Bl( \int_{-b}^b  \int_0^{3/2} |\p_1 F(x,\al)|^p\, d\al\,  dx\Br)^{\f 1p}.\end{align*}
	Setting $\xi(v) =(1, g'(v))$ for $v\in [-2b, 2b]$, we obtain from Theorem~\ref{THM:2D BERN}  that  
	\begin{align}
	\Sigma_1\leq \f {C_{\ell} }{n_1} \Bl( \int_{-b}^b \int_{0}^{\f 32} |\p_{\xi(x)} f(x, g(x)-\al)|^p\, d\al\, dx\Br)^{\f 1p} \leq C_{\ell} \va \|f\|_{L^p(G^\ast)}.\label{3-8:bern}
	\end{align}
	
	Similarly,   using~\eqref{16-13} and~\eqref{3-6:bern}, we have  
	\begin{align*}
	\Sigma_2
	&\leq  \f C {n_1} \Bl(\sum_{i=1}^{n_1} \sum_{j=1}^m  \int_{x_{i-1}}^{x_i}\int_{\al_{j-1-\ell}}^{\al_{j+\ell}}  |\p_2 F(u,\al)|^p (\sqrt{\al} +n^{-1})^p\, d\al\, du\Br)^{\f1p}\\
	&\leq  \f C {n_1}  \int_{-b}^b  \int_0^1 |\p_2 F(u,\al) (\sqrt{\al} +n^{-1}) \, d\al\, du,\end{align*}
	which, using the Remez  type inequality  (see \cite[(7.17)]{MT2}) and 	the univariate  Markov-Bernstein-type  inequality  (see \cite[ Theorem 7.3]{MT2} ),  is estimated above by 
	\begin{align}
	&	\f C {n_1} \Bl(\int_{-b}^b  \int_0^1 |\p_2 F(u,\al)|^p (\sqrt{\al} )^p\, d\al\, du \Br)^{\f 1p}\notag\\
	&\leq C \va \Bl(\int_{-2b}^{2b} \int_0^{g(x)} | f(x,g(x)-\al)|^p d\al\, dx\Br)^{\f1p}\leq C\va \|f\|_{L^p(G^\ast)}.\label{3-9:bern}
	\end{align}
	Now combining~\eqref{3-9:bern},~\eqref{3-8:bern} with~\eqref{3-7-0:bern}, we obtain  the desired  estimate~\eqref{3-2-1:bern}.		
\end{proof}
\begin{rem}\label{rem:cardinality-dom-spec-type}
	The above proof implies that if $\Ld\subset G$ is $\f \va n$-separated with respect to the metric $\wh{\rho}_G$, then the cardinality of $\Ld$ is at most $C_\va n^{d+1}$. This follows immediately from the fact that every $\og\in\Ld$ is contained in a unique set $I_{i,j}$.
\end{rem}

To state the next lemma, let $B:=\{\xi\in\RR^{d+1}:\  \ \|\xi\|\leq 1\}$ denote   the unit ball of  $\RR^{d+1}$ equipped with the metric \begin{equation}\label{metric:ball}\rho(\xi, \eta) =\|\xi-\eta\|+|\sqrt{1-\|\xi\|^2}-\sqrt{1-\|\eta\|^2}|,\  \ \xi,\eta\in B.\end{equation}
For  $\xi\in B$ and $r>0$,  set 
$B_\rho(\xi, r):=\{\eta\in B:\  \  \rho(\xi, \eta) \leq r\}$, and define  $cB_\rho(\xi, r):=\{\eta\in B:\  \  \rho(\xi, \eta) \leq cr\}$ for $c>0$.

\begin{lem}\label{lem-16-2-0}
	Let $n\in\NN$ and $0<p\leq \infty$. 
	If  $\{\xi_j\}_{j=1}^m \subset B$ is  $\f 1n$-separated with respect to the metric $\rho$ given in~\eqref{metric:ball}, then  there exists a constant  $\va\in (0,1)$ depending only on $p$ and $d$ such that 
	$$\|f\|_{L^p(B)}\leq C \|f\|_{L^p(B\setminus E)},\   \   \  \forall f\in \Pi_n^{d+1}$$
	with  $E := \bigcup_{j=1}^m B_\rho(\xi_j, \f \va n)$,     and the constant $C$ being dependent  only on $d$ and $p$.\end{lem} 
\begin{proof} For simplicity, we will prove the result for $p<\infty$ only. 	The case $p=\infty$ can be treated similarly.
	According to  Theorem 11.6.1 of \cite[p.~290]{DX2}, there exists a constant $\va\in (0,1/4)$ depending only on $d$ and $p$ such that for every maximal $\f {4\va} n$-separated subset $\Ld$ of $B$ ( with respect to the metric $\rho$), and any $f\in\Pi_n^{d+1}$,
	\begin{align} C_1 &\Bl( \sum_{\xi\in\Ld} |B_\xi| \max _{\eta\in 4B_\xi} |f(\eta)|^p \Br)^{\f 1p} \leq \|f\|_{L^p(B)}
	\leq C_2 \Bl( \sum_{\xi\in\Ld} |B_\xi| \min _{\eta\in 4B_\xi} |f(\eta)|^p \Br)^{\f 1p},\label{16-5}\end{align}
	where $B_\xi=B_\rho(\xi, \f {\va} n)$ and the constants $C_1, C_2>0$ depend only on $d$ and $p$.
	We may choose the maximal $\f {4\va} n$-separated set $\Ld$ in~\eqref{16-5} in such a way  that $\{\xi_j\}_{j=1}^m \subset \Ld$. Then  using~\eqref{16-5} and the doubling property of the measure $|B_\xi|$, we have 
	\begin{align*}
	\|f\|_{L^p(B)}
	&\leq C_2 \Bl( \sum_{\xi\in\Ld}  \f {|B_\xi|} {  |(2B_\xi)\setminus B_\xi|} \int_{(2B_\xi)\setminus B_\xi} |f(\eta)|^p\, d\eta  \Br)^{\f 1p}\leq C \Bl( \sum_{\xi\in\Ld} \int_{(2B_\xi)\setminus B_\xi} |f(\eta)|^p\, d\eta  \Br)^{\f 1p}\\
	&\leq C \|f\|_{L^p(B\setminus E)},
	\end{align*}
	where the last step uses the fact that $B_\xi \cap (2B_\eta)=\emptyset$ for any two distinct points $\xi, \eta\in\Ld$.
\end{proof}

Next, 
we set  $\Og_\da:=\{ \xi\in\Og:\  \  \dist(\xi,\Ga) \ge \da\}$ for $\da\in (0,1)$. 

\begin{lem}\label{lem-16-2b}  Assume that  $n\ge \mu>1$, $\da\in (0,1)$, and $\va\in (0,\da/4)$.  Let  $\Ld\subset \Og_\da$ be   $\f \va n$-separated  with respect to the Euclidean metric, and let 
	$$V_\xi: = B_{\f {\mu\va} n} (\xi) =\Bl\{\eta\in\RR^{d+1}:\  \  \|\eta-\xi\|\leq \f {\mu\va} n\Br\},\  \ \xi\in\Ld.$$	
	Then any  $f\in\Pi_{n}^{d+1}$ and $1\leq p\leq \infty$,   
	\begin{equation}\label{16-17} \Bl( (\f \va n)^{d+1} \sum_{\xi\in\Ld} \Bl|\osc \bl(f; V_{\xi}\br)\Br|^p\Br)^{\f1p} \leq C_{\mu,\da}\va \|f\|_{L^p(\Og)},\end{equation}
	where the constant $C_{\mu, \da}$ is independent of $\va$, $n$ and $f$.
\end{lem}

\begin{proof}Let $n_1=n/\va$ .  	We  cover $\Og_\da$ with finitely many closed balls $B_1,\dots, B_{n_0}\subset \Og_{3\da/4}$ of radius $\da/4$, where $n_0$ depends only on $\da$ and $\Og$.  Clearly, it is sufficient to prove that for $f\in\Pi_n^{d+1}$, and $\Ld_i=B_i\cap \Ld$ with $i=1,\dots, n_0$,
	\begin{equation}\label{16-18} \Bl( (\f \va n)^{d+1} \sum_{\xi\in\Ld_i} \Bl|\osc \bl(f; V_\xi\br)\Br|^p\Br)^{\f1p} \leq C_{\mu,\da}\va \|f\|_{L^p(\Og)}.\end{equation}
	
	For $\xi\in \Ld$, we set 	$f_{V_\xi}=\f 1{|V_\xi|} \int_{V_\xi} f(\eta)\, d\eta$, and let 
	$\eta_\xi\in  V_{\xi}$ be such that 
	$$ \max_{ \eta\in  V_\xi} |f(\eta)-f_{ V_\xi} |=|f(\eta_\xi)-f_{ V_\xi} |.$$
	Then by  Poincare's inequality, we have that 
	for each $\xi\in\Ld_j$ with $1\leq j\leq n_0$, 
	\begin{align*}
	\osc \bl(f; V_\xi\br)& \leq 2 |f(\eta_\xi)-f_{ V_{\xi}}| \leq C_d \int_{V_\xi}
	\|\nabla f(\eta)\|\|\eta-\eta_\xi\|^{-d} \, d\eta\\
	&\leq C \int_{2B_j}
	\|\nabla f(\eta)\|\|\eta-\eta_\xi\|^{-d} (1+ n_1\| \eta-\eta_\xi\|)^{-d-2}\, d\eta,
	\end{align*}
	which, using Lemma~\ref{lem-16-2-0}, is estimated above by 
	\begin{align*}
	& C \int_{\{\eta\in 2 B_j:\  \  \|\eta-\eta_\xi\| \ge \f {\va_1} {n_1}\}}
	\|\nabla f(\eta)\|\|\eta-\eta_\xi\|^{-d} (1+ n_1\| \eta-\eta_\xi\|)^{-d-2}\, d\eta\\
	&\leq C n_1^d\int_{\Og_{\da/4}}
	\|\nabla f(\eta)\|(1+ n_1\| \eta-\eta_\xi\|)^{-2d-2}\, d\eta.
	\end{align*}
	By H\"older's inequality, this implies 
	\begin{align*}
	\Bl| \osc \bl(f; V_\xi\br)\Br|^p &\leq C n_1^{-p} n_1^{d+1}\int_{\Og_{\da/4}}
	\|\nabla f(\eta)\|^p (1+ n_1\| \eta-\eta_\xi\|)^{-2d-2}\, d\eta.
	\end{align*}
	It then follows that \begin{align}
	(\f \va n)^{d+1} &\sum_{\xi\in\Ld_j} \Bl|\osc \bl(f; V_\xi\br)\Br|^p\notag\\
	&\leq C  n_1^{-p}  \int_{\Og_{\da/4}}
	\|\nabla f(\eta)\|^p \sum_{\xi\in\Ld_j} (1+ n_1\| \eta-\eta_\xi\|)^{-2d-2}\, d\eta. \label{16-21}\end{align}
	Since $\Ld\subset \Og_\da$ is   $\f 1 {n_1}$-separated  with respect to the Euclidean metric,  we obtain that for any $\eta\in\RR^{d+1}$, 
	\begin{align*}
	\sum_{\xi\in\Ld_j} &(1+ n_1\| \eta-\eta_\xi\|)^{-2d-2}\leq C 	 n_1^{d+1} \sum_{\xi\in\Ld_j} \int_{B_{n_1^{-1}} (\xi)}  (1+ n_1\| \eta-\zeta\|)^{-2d-2}\, d\zeta\\
	&\leq C  n_1^{d+1}\int_{\RR^{d+1}}  (1+ n_1\| \eta-\zeta\|)^{-2d-2}\, d\zeta
	\leq  C_d <\infty,
	\end{align*}
	where the first step uses the fact that   $\eta_\xi\in  B_{\f {\mu}{n_1}}(\xi)$ for each $\xi\in\Ld_j$.
	This combined with~\eqref{16-21} yields  \begin{align*}
	\Bl( (\f \va n)^{d+1} \sum_{\xi\in\Ld_j} \Bl|\osc \bl(f; V_\xi\br)\Br|^p\Br)^{\f1p}&\leq C n_1^{-1} \Bl(\int_{\Og_{\da/4}}
	\|\nabla f(\eta)\|^p \, d\eta\Br)^{\f1p},\end{align*}
	which, by the Bernstein inequality, is bounded above by 
	\begin{align*}
	C  n n_1^{-1} \|f\|_{L^p(\Og)}\leq C\va \|f\|_{L^p(\Og)}.
	\end{align*}
	This completes the proof of Lemma~\ref{lem-16-2b}.
\end{proof}

\begin{proof}[Proof of Theorem~\ref{thm-16-1:MZ}]  Without loss of generality, we may assume that $1\leq p<\infty$ and $n\ge N_\Og$, where $N_\Og$ is a sufficiently large positive integer depending only on $\Og$. 
	Indeed, the case $p=\infty$ is simpler, and  can be treated similarly, while~\eqref{16-17} for $n\leq N_\Og$ can be deduced from the case $n=N_{\Og}$.
	
	By Lemma~\ref{lem-2-1-18},  we can find a constant $\ld_0\in (0,1)$ and finitely many domains $G_1,\dots, G_{m_0}\subset \Og$ of special type attached to $\Ga$ such that 
	$$ \Ga_\da :=\{\xi\in\Og:\  \ \dist(\xi, \Ga) \leq \da \} \subset \bigcup_{j=1}^{m_0} G_j(\ld_0)$$
	for some $\da\in (0,1)$  with $m_0$ and  $\da_0$ depending  only on $\Og$.  Then $\Ld=\bigcup_{j=1}^{m_0+1}\Ld_j$, where 
	$ \Ld_{j} =\Ld \cap G_j(\ld_0)$ for $j=1,\dots, m_0,$ and
	$\Ld_{m_0+1}  =\Ld\cap (\Og\setminus \Ga_\da)$.  Thus,
	\begin{align*}
	&\sum_{\xi\in\Ld} \Bl| U(\xi, \f \va n) \Br| \Bl|\osc (f; U (\xi, \f {\ell \va} n)) \Br|^p  \leq \sum_{j=1}^{m_0+1} S_j,\end{align*}
	where 
	\begin{align*}S_j:=\sum_{\xi\in\Ld_j} \Bl| U(\xi, \f \va n) \Br| \Bl|\osc (f; U (\xi, \f {\ell \va} n)) \Br|^p,\  \  \ 1\leq j\leq m_0+1.
	\end{align*}
	Clearly, by~\eqref{metric}  and~\eqref{rhog-0}, we  may choose $N_\Og$ sufficiently large so that 
	$$ B_{G_j} (\xi,  \f {\mu_1\va}n)\subset U(\xi, \f {\ell\va}n) \subset B_{G_j} (\xi,  \f {\mu_2\va}n)\subset G_j,\  \ \forall \xi\in\Ld_j,\  \ 1\leq j\leq m_0,$$
	and 
	$$B_{\f {\mu_1\va} n} (\xi)\subset  U(\xi, \f {\ell\va}n) \subset  B_{\f {\mu_2 \va} n} (\xi) \subset \Og \setminus \Ga_{\da/2},\  \  \forall \xi\in \Ld_{m_0+1},$$
	where $\mu_2>1>\mu_1>0$, and $\mu_1, \mu_2$ depend only on $\ell$ and $\Og$.

	For $1\leq j\leq m_0$,   we use~\eqref{rhog-0} and  Lemma~\ref{lem-16-2} to obtain  	
	\begin{align*}S_j&\leq C
	\sum_{\xi\in\Ld_{j}} \Bl| B_{G_j} (\xi, \f { \mu_2\va} n) \Br| \Bl|\osc (f; B_{G_j}  (\xi, \f {\mu_2\va} n)) \Br|^p	\leq \Bl(C \va \|f\|_{L^p(\Og)}\Br)^p.
	\end{align*}
	
	For $j=m_0+1$, we use~\eqref{metric}  and Lemma~\ref{lem-16-2b} to obtain 
	\begin{align*}	
	S_{m_0+1}:&\leq C\sum_{\xi\in\Ld_{m_0+1} } \Bl(\f \va n\Br)^{d+1}  \Bl|\osc (f; B_{\f {\mu_2 \va} n}(\xi) \Br|^p 
	\leq \Bl(C \va \|f\|_{L^p(\Og)}\Br)^p.
	\end{align*}
	
	Putting the above together, we prove the estimate~\eqref{3-4:bern}. This completes the proof of Theorem~\ref{thm-16-1:MZ}.
\end{proof}

\begin{rem}\label{rem:cardinality-max-sep}
	The proof implies that  the number of elements of any maximal $\f\va n$-separated subset of $\Omega$ with respect to the metric $\rho_\Omega$ is of order $n^{d+1}$. Indeed, note that with a small fixed $\delta>0$ the metric $\rho_\Omega$ is equivalent to the Euclidean metric on $\Omega_\delta$, so the lower bound on the number of elements follows from the standard volume estimate arguments. For the upper bound, one can use~\eqref{6-1-metric-0}, Remark~\ref{rem:cardinality-dom-spec-type} and the standard volume estimate in the context of Lemma~\ref{lem-16-2b}.
\end{rem}



\end{document}